\documentclass[11pt]{amsart}
\allowdisplaybreaks
\usepackage[top=1.1in,bottom=1.1in,left=1.2in,right=1.2in]{geometry}
\usepackage{amsmath, amscd, amssymb}
\usepackage[frame,cmtip,arrow,matrix,line,graph,curve]{xy}
\usepackage{graphpap, color}
\usepackage[mathscr]{eucal}
\usepackage{cancel}
\usepackage{verbatim}
\usepackage{adjustbox}
\usepackage{tikz}
\usepackage{float}
\usepackage{booktabs}
\usepackage{multirow}
\usepackage{tabularx}
\usepackage{tikz-cd}
\usepackage{bbm}
\usepackage{stmaryrd}

\usepackage{hyperref}

\numberwithin{equation}{section}

\newtheorem{theorem}{Theorem}[section]
\newtheorem{corollary}[theorem]{Corollary}
\newtheorem{proposition}[theorem]{Proposition}
\newtheorem{lemma}[theorem]{Lemma}

\theoremstyle{definition}
\newtheorem{definition}[theorem]{Definition}

\newtheorem{remark}[theorem]{Remark}

\newcommand{\Z}{\mathbb{Z}}
\newcommand{\Q}{\mathbb{Q}}

\newcommand{\C}{\mathbb{C}}

\newcommand{\A }{\mathbb{A}}
\newcommand{\PP}{\mathbb{P}}

\def\EE{\mathbb{E}}

\def\QQ{\mathbb{Q}}

\def\ZZ{\mathbb{Z}}

\def\sE{\mathscr{E}}

\def\sK{\mathscr{K}}
\def\sP{\mathscr{P}}
\def\sQ{\mathscr{Q}}

\def\sT{\mathscr{T}}

\newcommand{\cal}{\mathcal}

\def\cB{{\cal B}}
\def\cC{{\cal C}}
\def\cD{{\cal D}}
\def\cE{{\cal E}}
\def\cF{{\cal F}}

\def\cL{{\cal L}}
\def\cM{{\cal M}}
\def\cN{{\cal N}}
\def\cO{{\cal O}}
\def\cP{{\cal P}}

\def\cS{{\cal S}}

\def\cZ{{\cal Z}}

\def\fg{\mathfrak g}

\def\ft{\mathfrak t}

\def\tS{\widetilde{S}}

\def\tY{{\widetilde{Y}}}

\def\hbar{\overline{h}}

\def\GL{\mathrm{GL} }
\def\SL{\mathrm{SL} }
\def\PGL{\mathrm{PGL} }

\DeclareMathOperator{\Aut}{Aut}

\DeclareMathOperator{\id}{id}

\DeclareMathOperator{\Stab}{Stab}
\DeclareMathOperator{\codim}{codim}

\DeclareMathOperator{\Pic}{Pic}

\DeclareMathOperator{\Hom}{Hom}

\DeclareMathOperator{\Sym}{Sym}

\def\diag{\mathrm{diag} }

\def\Bl{\mathrm{Bl}}
\def\Fl{\mathrm{Fl}}

\newcommand{\hooklongrightarrow}{\lhook\joinrel\longrightarrow}

\def\and{\quad{\rm and}\quad}
\def\lra{\longrightarrow }

\def\beq{\begin{equation}}
\def\eeq{\end{equation}}
\def\ben{\begin{enumerate}}
\def\een{\end{enumerate}}

\def\res{\mathrm{res}}

\def\and{\quad\text{and}\quad}

\def\e{\epsilon}
\def\w{\omega}

\def\pt{\mathrm{pt}}

\def\Mps{\overline{M}^{\mathrm{ps}}}

\def\PH{P^{\mathrm{H}}}
\def\PK{P^{\mathrm{K}}}

\def\PCY{P^{\mathrm{CY}}}
\def\PGIT{P^{\mathrm{GIT}}}

\def\cPH{\cP^{\mathrm{H}}}
\def\cPK{\cP^{\mathrm{K}}}
\def\cPKs{\cP^{\mathrm{K},s}}
\def\cPCY{\cP^{\mathrm{CY}}}
\def\cPGIT{\cP^{\mathrm{GIT}}}
\def\cPGITs{\cP^{\mathrm{GIT},s}}

\def\cM{\mathcal{M}}
\def\cMps{\overline{\cM}^{\mathrm{ps}}}

\def\cDK{\cD^{\mathrm{K}}}

\def\cEK{\cE^{\mathrm{K}}}

\def\IH{I^{\mathrm{H}}}

\def\IK{I^{\mathrm{K}}}

\def\cPhat{\widehat{\cP}}
\def\Phat{\widehat{P}}
\def\cPhatCY{\widehat{\cP}^{\mathrm{CY}}}
\def\cDhat{\widehat{\cD}}

\def\cEhat{\widehat{\cE}}
\def\cEhatCY{\widehat{\cE}^{\mathrm{CY}}}
\def\cFhat{\widehat{\cF}}
\def\cFhatCY{\widehat{\cF}^{\mathrm{CY}}}
\def\Ehat{\widehat{E}}

\def\Uhat{\widehat{U}}
\def\ehat{\hat{e}}
\def\fhat{\hat{f}}
\def\ghat{\hat{g}}
\def\ihat{\hat{i}}
\def\jhat{\hat{j}}
\def\Ihat{\hat{I}}

\def\shat{\hat{s}}
\def\cShat{\widehat{\cS}}
\def\cChat{\widehat{\cC}}
\def\cShatH{\cShat^{\,\rmH}}
\def\cChatH{\cChat^{\,\rmH}}
\def\cShatK{\cShat^{\,\rmK}}
\def\cChatK{\cChat^{\,\rmK}}
\def\cShatCY{\cShat^{\mathrm{CY}}}
\def\cChatCY{\cChat^{\mathrm{CY}}}
\def\cZhat{\widehat{\cZ}}

\def\Shat{\widehat S}
\def\Chat{\widehat C}
\def\cLhat{\widehat \cL}
\def\sThat{\widehat \sT}

\def\DC{\mathrm{DC}}
\def\HT{\mathrm{HT}}

\def\pTac{\mathrm{tac}}

\def\Gm{\mathbb{G}_m}

\def\rmH{\mathrm{H}}
\def\rmK{\mathrm{K}}
\def\flex{{\mathrm{flex}}}
\def\tr{{\mathrm{tr}}}

\def\chartheta{{\theta}}
\def\charthetabar{\bar\chartheta}

\newcommand{\pHT}{p_{\HT}}
\newcommand{\pDC}{p_{\DC}}

\newcommand{\pr}{\mathrm{pr}}

\title[Chow and cohomology of moduli stacks of plane quartics]{Chow and cohomology rings of moduli stacks of plane quartics}
\date{May 17, 2026}

\author{Kenneth Ascher} 
\address{Department of Mathematics, University of California Irvine,
Irvine, CA 92697,  USA}
\email{kascher@uci.edu}

\author{Donggun Lee}
\address{June E Huh Center for Mathematical Challenges, Korea Institute for Advanced Study, 85 Hoegiro, Dongdaemun-gu, Seoul 02455, Republic of Korea}
\email{dglee@kias.re.kr}

\begin{document}

\begin{abstract}
This paper studies the Chow and cohomology rings of the Hacking moduli stack $\cP^{\mathrm{H}}$ of plane quartics. We construct a smooth proper Deligne--Mumford stack resolving the Calabi--Yau wall crossing between the KSBA and K-moduli compactifications for plane quartics via stack-theoretic weighted blowups. Its coarse moduli space is, up to normalization, the fiber product of the natural diagram relating the KSBA, K-moduli, and boundary polarized Calabi--Yau compactifications. From this, we compute the Poincar\'e polynomial of $\cP^{\mathrm{H}}$, show that the cycle class map is an isomorphism with rational coefficients, and determine generators and relations for its Chow ring in terms of tautological classes. Analogous results are established for the GIT and K-moduli stacks.
\end{abstract}

\maketitle

\section{Introduction}

Compact moduli spaces of stable log pairs, introduced by Koll\'ar--Shepherd-Barron and Alexeev (KSBA), provide higher-dimensional analogues of the Deligne--Mumford compactification of the moduli space of stable curves.
While the general theory has undergone major developments in recent years (see, for instance, \cite{kollar-modulibook}), explicit examples with a detailed understanding of their topology and intersection theory are extremely rare.

In \cite{Hac-thesis,Hac}, Hacking constructed a KSBA compactification $\cPH_d$ of the moduli stack of smooth plane curves of degree $d$, parametrizing $\QQ$-Gorenstein slc degenerations of pairs $(\PP^2, (\frac{3}{d}+\epsilon)C_d)$ 
for $\epsilon>0$ sufficiently small. 
The moduli stacks $\cPH_d$ are among the earliest and most prominent examples of KSBA compactifications.

More recently, K-stability has led to the construction of compact moduli spaces of log Fano pairs; see, for instance, \cite{xu-survey, xu-book}. 
These K-moduli spaces often provide rich birational models with modular interpretations.

In the case of plane curves,
\cite{ADL} constructed K-moduli stacks $\cPK_d(c)$, parametrizing K-polystable degenerations of $(\PP^2, cC_d)$ for $0< c<\frac{3}{d}$, and studied their wall-crossing behavior as $c$ varies. 
The stack $\cPK_d(c)$ is closely related to the GIT quotient stack $\cPGIT$ of plane curves of degree $d$.
For sufficiently small $c$, it is canonically isomorphic to $\cPGIT$, and the first wall-crossing is given by a weighted blowup of $\cPGIT$. 
Subsequent wall-crossings are generally more complicated,
but tend to simplify the geometry of the moduli space as $c$ decreases. 
This is reflected in \cite[Question~9.3 and Theorem~9.5]{ADL}, where it is asked whether the Picard rank of $\cPK_d(c)$ decreases as $c$ decreases, and verified when $3\mid d$ or $d<13$.

This yields the wall-crossing diagram
\beq\label{intro:K-wc}\begin{tikzcd}
	\cPH_d \arrow[r,dashed]&\cPK_d\arrow[r,dashed]&~\cdots~\arrow[r,dashed]&\cPK_d(\epsilon)\cong \cPGIT
\end{tikzcd}\eeq
where $\cPK_d:=\cPK_d(\frac{3}{d}-\epsilon)$ denotes the K-moduli stack corresponding to the maximal value of $c$.
The same diagram exists for the coarse moduli space $\PH_d$ and the good moduli spaces $\PK_d(c)$. 

Moreover, \cite{bpCY} constructed a compactification $\cPCY_d$ parametrizing boundary polarized Calabi--Yau degenerations of $(\PP^2,\tfrac{3}{d}C_d)$. 
It contains both $\cPH_d$ and $\cPK_d$ as open substacks, and admits a projective moduli space $\PCY_d$, which serves as a base for the wall crossing
\beq\label{intro:flip}
\begin{tikzcd}
	\cPH_d \arrow[rr,dashed] \arrow[rd] && \cPK_d \arrow[ld] \\
	& \PCY_d
\end{tikzcd}
\eeq
between KSBA stability and K-stability \cite[Theorem~1.3]{bpCY}.
Here the stack $\cPH_d$ is Deligne--Mumford, whereas $\cPK_d$ is not.
The diagrams \eqref{intro:K-wc} and \eqref{intro:flip} relate $\cPH_d$ to $\cPGIT$, whose topology and intersection theory are often more amenable to explicit computation.

\smallskip

We now restrict to the case $d=4$ and omit $d$ from the notation. 
We work over $\C$, and all cohomology and Chow groups are taken with $\Q$-coefficients throughout.

In this case, the K-moduli wall-crossing consists of a single step, and in particular, $\cPK\to \cPGIT$ is the Kirwan-type stack-theoretic weighted blowup.
From the viewpoint of the Hassett--Keel program for genus three curves, \eqref{intro:K-wc} and \eqref{intro:flip} fit into the following diagram
\[\begin{tikzcd}
\overline{M}_3 \arrow[r] 
  & \PH \arrow[rd, "\varphi"'] \arrow[rr, dashed] 
    && \PK \arrow[ld, "\psi"] \arrow[rd] \\
  && \PCY 
    && \PGIT,
\end{tikzcd}\]
where $\overline{M}_3$ denotes the moduli space of stable curves of genus three, which is also isomorphic to the KSBA compactification of pairs with coefficient one constructed by Hassett \cite{Has99}, and $\PGIT$ denotes the GIT moduli space. The flip corresponds to the first flip in the log minimal model program for $\overline{M}_g$, specialized to the case $g=3$ \cite[p.~4]{HL}; see also \cite{HH13}. 

From the perspective of birational geometry, this flip is well understood.
The morphisms to $\PCY$ contract two natural loci to a single point: 
\begin{itemize}
	\item $\varphi$ contracts the closed codimension-two boundary stratum $\cZ_2\subset \cPH$, parametrizing unions of two genus one curves meeting at two nodes;
	\item $\psi$ contracts the strictly polystable locus $\cZ_\pTac\subset \cPK$, parametrizing tacnodal curves admitting $\Gm$-stabilizers.
\end{itemize}
However, for the purposes of topology and intersection theory,
it is not sufficient to understand the birational map $\cPH \dashrightarrow \cPK$ merely as a flip.

Our first main result 
constructs a smooth proper Deligne--Mumford stack resolving this Calabi--Yau wall crossing
via stack-theoretic weighted blowups \cite{abramovich-temkin-wlodarczyk,QuekRydh} centered at $\cZ_2$ and $\cZ_\pTac$.
\begin{theorem}\label{thm:mainthmintro}
There exists a smooth proper Deligne--Mumford stack $\cPhat$ 
\[\begin{tikzcd}
	& \cPhat \ar[ld,"\Phi"'] \ar[rd,"\Psi"] \\
\cPH && \cPK
\end{tikzcd}\]
resolving the Calabi--Yau wall crossing over $\PCY$,
such that
\begin{enumerate}
	\item $\Phi$ is a stack-theoretic weighted blowup along $\cZ_2$; 
	\item $\Psi$ is a Kirwan-type stack-theoretic weighted blowup along $\cZ_\pTac$; 
	\item the coarse moduli space $\Phat$ of $\cPhat$ is the normalization of the fiber product $\PH \times_{\PCY} \PK$.
\end{enumerate}
\end{theorem}
This resolution interpolates between the KSBA and K-moduli compactifications at the level of both stacks and families.
We construct a family of pairs over $\cPhat$ simultaneously realizing the degenerations on the KSBA and K-moduli sides.
More precisely, over the exceptional divisor, the fibers are chains of three components whose outer components recover the pairs parameterized by $\cZ_2\subset \cPH$, while the middle component is given by minimal dlt modifications of the K-polystable pairs parameterized by $\cZ_\pTac\subset \cPK$.

Moreover, $\Phat$ is naturally determined by the boundary polarized Calabi--Yau compactification $\cPCY$: it arises as a weighted analogue of the canonical reduction of stabilizers in the sense of \cite{ER}; see Corollary~\ref{cor:bpCY} and Remark~\ref{rem:bpCY}.

The construction of the resolution proceeds by first constructing $\cPhat$ as a Kirwan-type stack-theoretic weighted blowup of $\cPK$ along $\cZ_\pTac$, using the decomposition of the normal bundle $N_{\cZ_\pTac/\cPK}$ according to the weights of the $\Gm$-stabilizer action (Section~\ref{s:normal}).
We then modify the pulled-back universal family to construct $\Phi$, and use the universal property of the blowup to identify $\Phi$ with the stack-theoretic weighted blowup along $\cZ_2$ (Section~\ref{s:resolution}).

Finally, the morphism $\Phat\to \PGIT$ gives a weighted Kirwan partial desingularization \cite{Kir85}.

\smallskip

This resolution is also closely related to Hacking's stratification
\[\cPH = \cZ_0 \sqcup \cZ_1 \sqcup \cZ_2,\]
where the open stratum $\cZ_0$ parametrizes quartic curves on $\PP^2$, the codimension-one stratum $\cZ_1$ parametrizes curves on $\PP(1,1,4)$, and the codimension-two stratum $\cZ_2$ parametrizes curves on the reducible surface $\PP(1,1,2)\cup \PP(1,1,2)$.
In particular, the morphism $\Phi:\cPhat\to \cPH$ is the stack-theoretic weighted blowup along the deepest stratum $\cZ_2$.

Moreover, the boundary divisor 
\[\partial\cPH=\overline\cZ_1=\cZ_1\sqcup \cZ_2\] 
is the image of the exceptional divisor of $\cPK\to \cPGIT$ under the flip. 
By Theorem~\ref{thm:mainthmintro}, its coarse moduli space is a weighted Kirwan partial desingularization of the GIT moduli space of unordered eight points on $\PP^1$; see Remark~\ref{rem:Z1.Kirwan} for details.

\medskip

Our second main result determines the Betti numbers of $\cPH$ and establishes that the cycle class map is an isomorphism. 
Using Kirwan's method \cite{Kir84,Kir85,KL1,KL2} (see Section~\ref{s:Betti}), together with the above resolution of the flip, we prove the following.

\begin{theorem}[Theorem~\ref{thm:Betti.PH}]\label{thm:Poincare.intro}
The Poincar\'e polynomial of $\cPH$ is
\beq \label{eq:Poincare.intro} 1+2t^2+4t^4+5t^6+4t^8+2t^{10}+t^{12}.\eeq
Moreover, the cycle class map $A^*(\cPH) \to H^*(\cPH)$ is an isomorphism.
\end{theorem}

The proof proceeds by first establishing the corresponding results for the GIT quotient stack $\cPGIT$ using Kirwan's equivariantly perfect stratification \cite{Kir84} (Theorem~\ref{thm:poincare.git}). 

For the cycle class map, we prove a more general result asserting that the equivariant cycle class map for projective space is an isomorphism (Theorem~\ref{thm:cl.projsp}). As an immediate consequence, the cycle class map for the GIT moduli stack of hypersurfaces in $\PP^n$ is an isomorphism and its cohomology vanishes in odd degrees
(Corollary~\ref{cor:cl.hypersurf}).
 
We then deduce the corresponding results for $\cPK$ and $\cPhat$, and finally for $\cPH$, by applying the results of \cite{Kir84,Kir85} together with the explicit structure of the weighted blowups relating these stacks.
We also compute the intersection Betti numbers of the good moduli spaces $\PGIT$ and $\PK$ using \cite{kirwan-ic}; see Theorem~\ref{thm:IP.GIT.K}.

\medskip

Our third main result gives an explicit presentation of the Chow ring of $\cPH$ in terms of tautological classes. 
First, we recall the two natural Hodge-theoretic bundles on $\cPH$.

The first is the Hodge line bundle associated to the family of Calabi--Yau pairs $(\cS,\frac34\cC)$, where $(\cS,\cC)\to \cPH$ is the universal family. The line bundle
\[\Lambda:=\cO_{\cS}(4K_{\cS/\cPH}+3\cC)\]
is fiberwise trivial and hence descends to a line bundle on $\cPH$, again denoted by $\Lambda$. 
This line bundle is defined more generally on $\cPCY$ (see \cite[Section~4.2]{bpCY}), and descends to an ample $\Q$-line bundle on $\PCY$ \cite[Theorem~14.14]{bpCY}; see also \cite[Theorem~9.17]{ADL}.

The second is the Hodge bundle
\[\EE^\rmH:=\pi_{\cC*}(\omega_{\cC/\cPH})\]
associated to the universal family of curves $\pi_{\cC}:\cC\to \cPH$.
This is a rank $3$ vector bundle with determinant 
$\det \EE^\rmH=\Lambda(-\partial\cPH)$; see Corollary~\ref{cor:Hodge}. 

We introduce the following tautological classes:
\[\lambda:=c_1(\Lambda), \quad \delta=[\partial\cPH], \and
c_2^\rmH:=c_2\left((\EE^\rmH)^*\otimes(\det \EE^\rmH)^{1/3}\right).\]
We obtain the following presentation of the Chow ring of $\cPH$.
\begin{theorem}
\label{thm:Chowring.intro}
The Chow ring of $\cPH$ admits the explicit presentation
\[A^*(\cPH)\cong\Q[\lambda,\delta,c_2^\mathrm H]/(r_3,r_4,r_4',r_4'',r_5),\]
with homogeneous relations of degrees $3$, $4$, $4$, $4$, and $5$ (see Theorem~\ref{thm:Chowring.Hacking}).
\end{theorem}

\begin{remark}
The relation $r_3=\lambda^2\delta+2\lambda\delta^2$ is equivalent to $\lambda\cdot[\cZ_2]=0$.
\end{remark}

The Chow rings of $\cPGIT$, $\cPK$, and $\cPhat$ are computed in Sections~\ref{s:chow.pgit}, \ref{s:chow.pk}, and \ref{s:chow.phat}, respectively; see Theorem~\ref{thm:Chowring.GIT}, Theorems~\ref{thm:Chowring.PK} and \ref{thm:ChowK.lambda}, and Theorem~\ref{thm:Chow.cPhat}.
Our approach is based on Kirwan's method together with the blowup formulas of Arena--Obinna \cite{ArenaObinna, Arena}. 

Our approach largely follows that of Canning--Oprea--Pandharipande \cite{COP} on the Chow ring of the moduli space of quasi-polarized K3 surfaces of degree two, viewed as an open subvariety of $\PK_6$.
Indeed, for $\cPGIT$ and $\cPK$, we compute the relations arising from excising unstable strata via localization sequences, use incidence varieties to resolve the strata, and perform explicit pushforward computations.

The computation for $\cPhat$ fits into the same framework, but is more subtle and relies in part on the study of the normal bundle $N_{\cZ_\pTac/\cPK}$ in Section~\ref{s:normal}.

Finally, the Chow ring of $\cPH$ is obtained from those of $\cPK$ and $\cPhat$.
The generators are obtained via the flip to $\cPK$, while the relations are obtained by pushing forward relations from $\cPhat$.
To relate the tautological classes on $\cPH$ and $\cPK$, we compare the Hodge bundles $\EE^\rmH$ and $\EE^\rmK$ via the resolution $\cPhat$.
In particular, we obtain a short exact sequence
\[0\lra \Psi^*\EE^\rmK(-\cEhat)\lra \Phi^*\EE^\rmH\lra \cO_{\cEhat}\lra 0,\]
where $\cEhat\subset \cPhat$ denotes the exceptional divisor; see Theorem~\ref{thm:Hodge}.

\smallskip

The proofs of the two assertions in Theorem~\ref{thm:Poincare.intro} and of Theorem~\ref{thm:Chowring.intro} are independent.
Moreover, once Theorem~\ref{thm:Chowring.intro} is established, either assertion in Theorem~\ref{thm:Poincare.intro} implies the other.

\smallskip

We expect $\cPhat$ to admit a modular interpretation.
More precisely, in the case of plane quartics, we conjecture that $\Phat$ is isomorphic to the toroidal compactification of the moduli space of polarized K3 surfaces that arise as cyclic degree $4$ covers of $\PP^2$ branched along smooth quartic curves, as constructed and studied in \cite{AEH24,ADH}.

\smallskip

We also expect our construction to extend beyond plane quartics.
Namely, for general $d$, we expect a natural proper Deligne--Mumford stack resolving the Calabi--Yau wall crossing between $\cPH_d$ and $\cPK_d$, whose coarse moduli space is given by the normalization of the main component of $\PH_d\times_{\PCY_d}\PK_d$.
More generally, we expect an analogous statement in the setting of pairs $(S,C)$ with $S$ a del Pezzo surface and $C\sim_{\QQ}-rK_S$.
For such pairs, a boundary polarized Calabi--Yau compactification has been constructed in \cite{blum-liu}. 
These questions will be pursued in future work.

\subsection*{Outline.}

In Sections~\ref{s:Hacking} and \ref{s:K}, we study some of the boundary strata of $\cPH$ and $\cPK$. 
Section~\ref{s:Hacking} gives an explicit description of the codimension-two closed stratum $\cZ_2\subset \cPH$, which serves as the center of $\Phi$ (Proposition~\ref{prop:Z2}), 
while Section~\ref{s:K} identifies the strictly polystable locus $\cZ_\pTac\subset \cPK$, the center of the weighted blowup $\Psi$ (Proposition~\ref{prop:gerbe}).

In Sections~\ref{s:normal} and \ref{s:resolution}, we construct the stack $\cPhat$ resolving the Calabi--Yau wall crossing between $\cPH$ and $\cPK$. 
Section~\ref{s:normal} studies the decomposition of the normal bundle $N_{\cZ_\pTac/\cPK}$. 
Section~\ref{s:resolution} constructs the stack $\cPhat$ and the morphism $\Phi$ via a modification of the universal family, and proves that $\Phi$ is the stack-theoretic weighted blowup along $\cZ_2$ (Theorem~\ref{thm:roof.diagram}).

In Section~\ref{s:Hodge}, we compare the Hodge bundles on $\cPH$, $\cPK$, and $\cPGIT$. 

In Section~\ref{s:Betti}, we compute the Poincar\'e series of $\cPGIT$, $\cPK$, $\cPhat$, and $\cPH$, and prove that their cycle class maps are isomorphisms. 
Using Kirwan's equivariantly perfect stratification, we first compute the Poincar\'e series of $\cPGIT$ (Theorem~\ref{thm:poincare.git}), and obtain those of the other stacks via the blowup formula and Kirwan's theory of partial desingularization (Theorems~\ref{thm:PK.poincare} and~\ref{thm:Betti.PH}). 
For the cycle class map, we establish a general result for projective space (Theorem~\ref{thm:cl.projsp}), apply it to $\cPGIT$ (Corollary~\ref{cor:cl.hypersurf}), and extend this to the remaining stacks. 
We also record the intersection Betti numbers of the good moduli spaces $\PGIT$ and $\PK$ (Theorem~\ref{thm:IP.GIT.K}).

Sections~\ref{s:chow.pgit}--\ref{s:chow.ph} compute the Chow rings. 
In particular, Sections~\ref{s:chow.pgit}--\ref{s:chow.phat} treat $\cPGIT$, $\cPK$, and $\cPhat$ via localization sequences, blowup formulas, and explicit resolutions of unstable loci (Theorems~\ref{thm:Chowring.GIT}, \ref{thm:Chowring.PK}, \ref{thm:ChowK.lambda}, and~\ref{thm:Chow.cPhat}). 
We also compute the Chow rings of the stable loci $\cPGITs\subset \cPGIT$ and $\cPKs\subset \cPK$ (Theorem~\ref{thm:Chowring.stable}).
The Chow ring of $\cPH$ is obtained in Section~\ref{s:chow.ph} via pushforward from $\cPhat$ and expressed in terms of tautological classes (Theorem~\ref{thm:Chowring.Hacking}).

\subsection*{Acknowledgments}
We thank Dan Abramovich, Dori Bejleri, Harold Blum, Kristin DeVleming, and Yongnam Lee for helpful discussions.
We thank Woonam Lim for asking about the intersection Betti numbers of the moduli spaces, which led us to include Theorem~\ref{thm:IP.GIT.K}.

KA was partially supported by the National Science Foundation Grant DMS-2302550 and a UCI Chancellor's Fellowship.
DL was partially supported by the National Research Foundation of Korea(NRF) grant funded by the Korea government(MSIT) (2021R1A2C1093787), the Institute for Basic Science (IBS-R032-D1), and the KIAS Individual Grant (HP109201).

\section{Boundary of $\cPH$}\label{s:Hacking}
The moduli stack $\cPH$ parametrizes $\QQ$-Gorenstein slc degenerations of pairs $\left(\PP^2, (\frac{3}{4}+\epsilon)C_4\right)$. 
We describe the codimension-two closed stratum $\cZ_2\subset \cPH$, consisting of pairs $(S,C)$ with
\[
S=\PP(1,1,2)\cup_\ell \PP(1,1,2),\]
where the two components of $S$ are glued along a line $\ell$ through the cone point. 
This stratum is precisely the center of the stack-theoretic weighted blowup
$\cPhat \to \cPH$ in Theorem~\ref{thm:mainthmintro}(1).

The stack $\cZ_2$ is described as a quotient of a product of weighted projective stacks by a finite group; see Proposition~\ref{prop:Z2}.

\subsection{Stratification of $\cPH$}

We recall Hacking's stratification of $\cPH$ 
\[\cPH=\cZ_0\sqcup \cZ_1\sqcup \cZ_2,\]
where the index records the codimension of the stratum.
More explicitly,
\begin{itemize}
	\item $\cZ_0$ parametrizes pairs $(S,C)$ with $S=\PP^2$,
	\item $\cZ_1$ parametrizes pairs $(S,C)$ with $S=\PP(1,1,4)$, and
	\item $\cZ_2$ parametrizes pairs $(S,C)$ with $S = \PP(1,1,2) \cup_\ell \PP(1,1,2)$ as above.
\end{itemize}

If $\left[(S,C)\right] \in \cZ_2$, then $C=C_1\cup C_2$, where $C_i\subset \PP(1,1,2)$ are degree $4$ curves,
and they meet along  $\ell$ in two points away from the cone point.

We now make this description explicit by studying a single component.

\subsection{Normal form}
Let $x,y,z$ be homogeneous coordinates on $\PP(1,1,2)$. 

\begin{lemma}\label{lem:normal.form.Z2}
Let $C\subset \PP(1,1,2)$ be a degree $4$ curve, and let $\ell$ be a line. 
Suppose that the pair 
$\left(\PP(1,1,2),(\tfrac{3}{4}+\epsilon)C+\ell\right)$
is slc for sufficiently small $\epsilon>0$. 
Then, up to an automorphism of $\PP(1,1,2)$, we may write
\beq\label{eq:normal.form.Z2}
C:~z^2=x^4+a_2x^2y^2+a_3xy^3+a_4y^4,\quad \ell:~y=0,
\eeq
for $(a_2,a_3,a_4)\neq (0,0,0)$.
The subgroup of $\Aut \PP(1,1,2)$ preserving \eqref{eq:normal.form.Z2} is
\[\Gm\times \Gamma, \quad \text{with }~ \Gamma=\mu_2,\]
where $t\in \Gm$ acts by $[x:y:z]\mapsto[x:ty:z]$, and $\Gamma$ acts by $z\mapsto -z$.
\end{lemma}

\begin{proof}
Write
\[C:f_0(x,y)z^2+f_2(x,y)z+f_4(x,y)=0, \quad \ell:f_1(x,y)=0.\]
By the slc assumption, $C$ does not pass through the cone point $[0:0:1]$, hence we have $f_0\neq 0$, and may assume $f_0=1$.
Using automorphisms of $\PP(1,1,2)$, we may eliminate the linear term in $z$ and send $\ell$ to $y=0$, so that
\[C: z^2=f_4(x,y), \quad \ell:y=0.\]
The slc assumption implies that $f_4(x,0)\neq 0$.

The subgroup of $\Aut\PP(1,1,2)$ preserving this form consists of transformations of the form
\[(x,y,z)\mapsto (x+ay,by,cz),\]
from which we can eliminate the $x^3y$-term in $f_4$ and obtain \eqref{eq:normal.form.Z2}.
The description of the stabilizer follows from the requirement that the equation and the line $y=0$ are preserved, which forces $a=0$ and $c^2=1$.
\end{proof}

\subsection{Explicit description of $\cZ_2$}
For $w_0,\dots,w_m\in \Z_{>0}$, we denote by
\[\cP(w_0,\dots,w_m):=[(\A^{m+1}\setminus\{0\})/\Gm]\]
the weighted projective stack defined by the $\Gm$-action with weights $w_i$.
Its coarse moduli space is the weighted projective space $\PP(w_0,\dots,w_m)$.

By Lemma~\ref{lem:normal.form.Z2}, curves $C$ in the normal form \eqref{eq:normal.form.Z2} are determined by the coefficients $(a_2,a_3,a_4)\in \A^3\setminus\{0\}$ modulo the $\Gm$-action, and are thus parametrized by $\cP(2,3,4)$.
The involution $\Gamma=\mu_2$ appearing in Lemma~\ref{lem:normal.form.Z2} acts on $C$ by the elliptic involution.

We now construct a morphism
\beq\label{eq:maptoZ2}
[(\cP(2,3,4)\times \cP(2,3,4))/\mu_2]\times B\Gamma \longrightarrow \cZ_2,
\eeq
where $\mu_2$ interchanges the two factors.
A point of $\cP(2,3,4)\times \cP(2,3,4)$ determines two pairs $(\PP(1,1,2),C_1,\ell)$ and $(\PP(1,1,2),C_2,\ell)$ in normal form.
Gluing the two copies of $\PP(1,1,2)$ along $\ell$, we obtain a pair $(S,C)\in\cZ_2$, where the two curves meet $\ell$ at the points $[1:0:\pm 1]$.

Over $\cP(2,3,4)\times\cP(2,3,4)$, we have families parametrizing
\[(\PP(1,1,2), C_1, \ell) \and (\PP(1,1,2), C_2, \ell),\]
where the divisor $\ell=\{y=0\}$ is fixed pointwise by the $\Gm$-action.
These families are canonically identified along $\ell$, and hence glue to a family of pairs
\[(\PP(1,1,2)\cup_\ell \PP(1,1,2),\, C_1 \cup C_2)\]
over $\cP(2,3,4)\times \cP(2,3,4)$, inducing a morphism $\cP(2,3,4)\times\cP(2,3,4)\to \cZ_2$. 
This morphism is invariant under the $\mu_2$-action exchanging the two factors and under the $\Gamma$-action acting on each component by the involution $z\mapsto -z$, and hence defines the desired morphism.

\begin{proposition}\label{prop:Z2}
The morphism \eqref{eq:maptoZ2} is an isomorphism. 
In particular, if $Z_2$ denotes the coarse moduli space of $\cZ_2$, then
\[Z_2\cong (\PP(2,3,4)\times \PP(2,3,4))/\mu_2.\]
\end{proposition}

\begin{proof}
By Lemma~\ref{lem:normal.form.Z2},  \eqref{eq:maptoZ2} is stabilizer-preserving and bijective on closed points, hence it is the normalization.
Since $\cZ_2$ is smooth by \cite[Theorems~3.12 and~8.2]{Hac}, it follows that the morphism is an isomorphism.
\end{proof}

\begin{remark}
Let $\Mps_{1,2}$ denote the moduli space of pseudostable curves of genus one with two marked points \cite{schubert-pseudostable}.
In terms of pseudostable curves, Proposition~\ref{prop:Z2} together with the identification $\PH\cong \Mps_3$ (\cite[Proposition~20]{HL}) shows that the clutching morphism
$(\Mps_{1,2}\times \Mps_{1,2})/\mu_2 \to \Mps_3$
induces an isomorphism of coarse spaces
$Z_2 \cong (\Mps_{1,2}\times \Mps_{1,2})/\mu_2$.

Moreover, since $\cMps_{1,2}\cong \cP(2,3,4)$ (see \cite{Inc}), Proposition~\ref{prop:Z2} can be viewed as a stack-theoretic refinement of this description.
%
\end{remark}

\section{Boundary of $\cPK$}\label{s:K}

The K-moduli stack $\cPK$ parametrizes $\QQ$-Gorenstein K-semistable degenerations of pairs $\left(\PP^2, (\frac{3}{4}-\epsilon)C_4\right)$.
In this section, we first recall the description of $\cPK$ as a stack-theoretic weighted blowup of $\cPGIT$, following \cite{ADL}, and describe the geometry of its exceptional divisor $\cEK$. 
We then study the strictly polystable locus $\cZ_\pTac\subset \cPK$, showing that it is a gerbe over a stacky curve $[\PP^1/\Gamma]$ and, moreover, that it intersects $\cEK$ transversely.

\subsection{The morphism $\cPK \to \cPGIT$ and its exceptional divisor}
We recall the description of $\cPK$ and the morphism to the GIT moduli stack $\cPGIT$, following \cite[Sections~5 and~6.1]{ADL}. 
In the case of plane quartics, this is the unique K-moduli wall crossing, and it coincides with the first step of a weighted Kirwan partial desingularization \cite{Kir85} for $\cPGIT$. 

Let $G=\SL_3$ and $\bar G=\PGL_3$, and set
\[X=\lvert\cO_{\PP^2}(4)\rvert=\PP^{14}.\]
Let $X^{ss}\subset X$ be the GIT semistable locus with respect to the unique $G$-linearization on $\cO_X(1)$.
Then the GIT moduli stack of plane quartics is
\[\cPGIT:=[X^{ss}/\bar G],\]
and the quotient stack $[X^{ss}/G]$ is a $\mu_3$-gerbe over $\cPGIT$. 

Let $Z_\DC\subset X^{ss}$ denote the locus of double conics. Then 
\[Z_\DC \cong \bar G/\PGL_2 \and \pDC:=[Z_\DC/\bar G]\cong B\PGL_2,\]
and in particular $Z_\DC$ and $\pDC$ are smooth.

The following result identifies $\cPGIT$ with a K-moduli stack. 

\begin{theorem}[{\cite[Theorem~1.3]{ADL}}]\label{thm:Kwallcrossing}
Let $0<c<\frac{3}{8}$. Then $\cPGIT$ is canonically isomorphic to the K-moduli stack of log Fano pairs $(\PP^2,cC)$ with plane quartic curves $C$. Moreover, there is a wall-crossing morphism
\[\cPK \lra \cPGIT.\]
It is the Kirwan-type stack-theoretic weighted blowup at $\pDC$ with weight two. 
\end{theorem}

Let us describe this theorem in terms of quotients by $\bar G$.
Let $U\to X^{ss}$ be the blowup along $Z_\DC$ with exceptional divisor $E_\DC$, and let $U^{1/2}\to U$ be the square root stack along $E_\DC$. 
The $\bar G$-action on $X^{ss}$ lifts to $U$ and $U^{1/2}$.

Let $U^{ss}\subset U$ be the GIT semistable locus with respect to 
\[\cO(a)|_U\otimes \cO(-E_\DC)\] 
for sufficiently large $a> 0$, and let $U^{1/2,ss}$ be its inverse image in $U^{1/2}$. 
We thus obtain $\bar G$-equivariant morphisms $U^{1/2,ss}\to U^{ss}\to X^{ss}$.

By \cite[Theorems~5.14--5.15]{ADL}, we have an isomorphism
\[\cPK \cong [U^{1/2,ss}/\bar G].\]
Under this, the morphism $\cPK \to \cPGIT$ is induced by $U^{1/2,ss}\to X^{ss}$.

The exceptional divisor $\cEK\subset \cPK$ is the descent of the inverse image of $E_\DC$ in $U^{1/2,ss}$.
The divisor $\cEK$ parametrizes pairs $(S,C)$ with $S=\PP(1,1,4)$.

To describe $\cEK$, we need to compute the normal space of $Z_\DC \subset X^{ss}$ and apply Luna's \'etale slice theorem \cite{luna-slice}.
The description of the normal space follows from \cite[(5.5)]{ADL}.
\begin{lemma}\label{lem:etalenbd.DC}
	Let $Q$ be a smooth conic with $\cO_Q(1)=\cO_{\PP^1}(1)$.
	Let $Q_2$ be the double conic supported on $Q$. The normal space of $Z_\DC$ in $X^{ss}$ at $[Q_2]$ is
	\beq\label{eq:normalspace.DC}\begin{split}
		N_{Z_\DC/X^{ss}}|_{[Q_2]}
		&\cong H^0(Q,\cO_Q(8)).
	\end{split}
	\eeq
	In particular, $\cPGIT$ is \'etale locally isomorphic at $[Q_2]$ to the quotient stack
	\[\big[H^0(Q,\cO_Q(8))/\PGL_2\big].\] 
\end{lemma}
\begin{proof}
	By \cite[(5.5)]{ADL}, the normal space is isomorphic to the cokernel of
	\beq\label{eq:coker.DC}H^0(\PP^2,\cO(4)\otimes \cO(-Q))\lra H^0(\PP^2,\cO(4)),\eeq
	which is $H^0(Q,\cO_Q(8))$ since $\cO_{\PP^2}(4)|_Q=\cO_Q(8)$. 
	The second assertion follows from Luna's \'etale slice theorem, since the stabilizer of a double conic in $\bar G$ is isomorphic to $\PGL_2$.
\end{proof}

We identify a given smooth conic $Q$ with $\PP^1$, and let 
\[Y_8=\lvert \cO_{\PP^1}(8)\rvert=\PP^8,\]
which carries a natural $\SL_2$-action.
Let $Y^{ss}_8$ be the GIT semistable locus with respect to the unique $\SL_2$-linearization on $\cO_Y(1)$.
From the above results and \cite[Theorem~5.9]{ADL}, we obtain the following.

\begin{theorem}\label{thm:cEK}
	The divisor $\cEK$ is a $\Gamma$-gerbe over the GIT quotient stack 
	\[[Y^{ss}_8/\PGL_2],\]
	where a degree $8$ polynomial $f_8(u,v)$ parametrized by $Y^{ss}_8$ corresponds to 
	\beq\label{eq:curvesinP114}\{w^2=f_8(u,v)\}\subset \PP(1,1,4),\eeq
	where $u,v,w$ denote the homogeneous coordinates of $\PP(1,1,4)$.
\end{theorem}
Recall that $\Gamma=\mu_2$.
We write $\Gamma$ for $\mu_2$ here to emphasize that it acts by sending $w$ to $-w$, thus corresponds to the hyperelliptic involution of \eqref{eq:curvesinP114}.

\subsection{Strictly polystable tacnodal curves}

We describe the preimage in $U^{1/2,ss}$ of the strictly polystable locus of $\cPK$.
Since this locus is the $G$-orbit of the fixed locus of a one-parameter subgroup of $G$, it suffices to describe the fixed locus itself.

By Mumford's GIT classification, the strictly polystable locus in $X^{ss}$ consists of double conics and tacnodal curves with $\Gm$-stabilizer (namely, oxes and cateyes) \cite[p.~80]{GIT}. 
In particular, the points in $X^{ss}$ with maximal-dimensional reductive stabilizer are precisely the double conics, which form the center of the weighted blowup $\cPK\to \cPGIT$. 
Since this locus is blown up, such maximal-dimensional stabilizers do not occur in $U^{ss}$ or $U^{1/2,ss}$.
Consequently, any positive-dimensional stabilizer in these spaces is, up to finite index, isomorphic to $\Gm$, and hence its identity component is conjugate to a one-parameter subgroup of $G$.

Fix a one-parameter subgroup $R\subset G$ defined by
\[\lambda_R:\Gm\lra G,\quad t\mapsto \diag(t,t^{-1},1).\]
Then the strictly polystable locus in $U^{1/2,ss}$ is the proper transform of the $G$-orbit of the $R$-fixed locus $(X^{ss})^R$.

The fixed locus $(X^{ss})^R$ is isomorphic to $\A^2_{s,p}\setminus\{0\}$, parametrizing quartics 
\[x^2y^2+sxyz^2+pz^4=0,\qquad (s,p)\neq (0,0).\]
Equivalently, such a quartic can be written as
\beq\label{eq:plane.tacnodal.eq}
(xy+az^2)(xy+bz^2)=0,\qquad (a,b)\neq (0,0),
\eeq
with $(s,p)=(a+b,ab)$.
Let $\Delta_R\subset (X^{ss})^R$ denote the locus of double conics. It is the vanishing locus of the discriminant $s^2-4p$, hence smooth.

\begin{lemma}\label{lem:Rfixedlocus}
	The natural map $(U^{ss})^R\to (X^{ss})^R$ is an isomorphism. 
	In particular, it restricts to $(U^{ss})^R\cap E_\DC\cong \Delta_R$, and hence the intersection $(U^{ss})^R\cap E_\DC$ is transverse.
	\end{lemma}
\begin{proof}
Since $X^{ss}$ and $U^{ss}$ are smooth, their $R$-fixed loci are smooth as well. It is therefore enough to show that the map $(U^{ss})^R\to (X^{ss})^R$ is bijective.

For injectivity, it suffices to consider points over $\Delta_R$, since $U^{ss}\to X^{ss}$ is injective away from $Z_\DC$. Let $Q_2$ be an $R$-fixed double conic. 
The map \eqref{eq:coker.DC} is $R$-equivariant, and the induced $R$-action on the cokernel \eqref{eq:normalspace.DC} has a one-dimensional invariant subspace.
This determines a unique point in $[Y_8^{ss}/\PGL_2]$, represented by $f_8(u,v)=u^4v^4$, so the fiber over $[Q_2]$ consists of a single point.

For surjectivity, \eqref{eq:plane.tacnodal.eq} with $a\neq b$ remains semistable in $U$, and the point corresponding to $f_8=u^4v^4$ is also semistable. 
This proves the first assertion.
The second assertion follows from the first and the smoothness of $\Delta_R$.
\end{proof}

By abuse of notation, we use the same notation $\Delta_R$ for its inverse image in $(U^{ss})^R$.
Comparing $(U^{ss})^R$ and $(U^{1/2,ss})^R$, we obtain the following description of $(U^{1/2,ss})^R$.

\begin{lemma}\label{lem:Rfixedlocus2}
	The natural map $(U^{1/2,ss})^R\to (U^{ss})^R$ is the square-root stack along $\Delta_R$.
	Thus, 
	\beq\label{eq:Rfixedlocus2}(U^{1/2,ss})^R\cong [(\A^2_{a,b}\setminus\{0\})/\Gamma],\eeq
	where $\Gamma$ interchanges the coordinates $a$ and $b$ of $\A^2_{a,b}$, and the map to $(X^{ss})^R$ is induced by $(a,b)\mapsto(a+b,ab)=(s,p)$. 
	In particular, $(U^{1/2,ss})^R$ is smooth.
\end{lemma}
\begin{proof}
	This follows immediately from Lemma~\ref{lem:Rfixedlocus} and smoothness of $\Delta_R$.	
\end{proof}

\subsection{Strictly polystable locus of $\cPK$}

Let $\cZ_\pTac \subset \cPK$ denote the strictly polystable locus. 
It parametrizes quartics of the form \eqref{eq:plane.tacnodal.eq} with $a\neq b$, as well as the K-semistable replacements of double conics, namely pairs isomorphic to \eqref{eq:curvesinP114} with $f_8=u^4v^4$ (see Theorem~\ref{thm:cEK}).

We describe this locus via the $R$-fixed loci and their quotients by the normalizer of $R$, following \cite[Section~5]{Kir85}.
Let $N_G(R)$ denote the normalizer of $R$ in $G$, so that
\[N_G(R)\cong T\rtimes \mu_2,\]
where $T\cong \Gm^2$ is the diagonal torus of $G$, and $\mu_2$ permutes the first two diagonal entries.
Similarly, for the images $\bar R$ and $\bar T$ in $\bar G$ of $R$ and $T$, we have $N_{\bar G}(\bar R)\cong \bar T\rtimes\mu_2$. 
Then the $R$-fixed loci of $U^{1/2,ss}$ and $U^{ss}$ are preserved by $N_{\bar G}(\bar R)$, and their $G$-orbits are given by
\[\begin{aligned}
	Z_\pTac^{1/2}&:=G\cdot (U^{1/2,ss})^R\cong G\times^{N_{G}(R)}(U^{1/2,ss})^R, \\ 
	Z_\pTac&:=G\cdot (U^{ss})^R\cong G\times^{N_{G}(R)}(U^{ss})^R.
\end{aligned}
\]
Here $G\times^H Z$ denotes the quotient of $G\times Z$ by the action of a subgroup $H\subset G$ given by $h\cdot (g,z)=(gh^{-1},hz)$, and the inverse isomorphisms are given by  $[g,z]\mapsto g\cdot z$.
Consequently, 
\[\cZ_\pTac \cong [Z_\pTac^{1/2}/ \bar G ]\cong [(U^{1/2,ss})^R / N_{\bar G}(\bar R)] \and [Z_\pTac/ \bar G ] \cong [(U^{ss})^R / N_{\bar G}(\bar R)].\]

\begin{remark}\label{rem:exact.TR}
	We will freely use $G$, $N_G(R)$, and $T$ when working with the corresponding structures for $\bar G$, $N_{\bar G}(\bar R)$, and $\bar T$.
	Note that the morphism $R\to \bar R$ is an isomorphism, while $G\to \bar G$ and $T\to \bar T$ are $\mu_3$-covers. These maps fit into commuting exact sequences
	\beq\label{eq:exact.TR}
	\begin{tikzcd}
		1 \arrow[r]&R\arrow[r]\arrow[d,"\cong"']&T\arrow[r,"\chartheta"]\arrow[d]& \Gm\arrow[d]\arrow[r]&1\\
		1 \arrow[r]&\bar R\arrow[r]	&\bar T\arrow[r,"\charthetabar"]	& \Gm	\arrow[r]&1
	\end{tikzcd}\eeq
	where $\chartheta$ and $\charthetabar$ are the characters defined by 
	\[\chartheta(t)=t_3^{-1}\and \charthetabar(t)=t_1t_2t_3^{-2}\]
	for $t=\diag(t_1,t_2,t_3)$ in $T$ or $\bar T$. The vertical map $\Gm\to\Gm$ sends $s$ to $s^3$.
	
	The exact sequences in \eqref{eq:exact.TR} for $T$ and $\bar T$ admit natural analogues for $N_G(R)$ and $N_{\bar G}(\bar R)$, with $R\rtimes \mu_2$ and $\bar R\rtimes \mu_2$ in place of $R$ and $\bar R$, respectively. These sequences split. In particular,
	\[N_G(R)\cong (R\rtimes\mu_2)\times \Gm \and N_{\bar G}(\bar R)\cong (\bar R \rtimes \mu_2)\times \Gm.\]
\end{remark}

We now show that $\cZ_\pTac$ is a trivial gerbe over $[\PP^1/\Gamma]$ arising from \eqref{eq:Rfixedlocus2}.

\begin{proposition}\label{prop:gerbe}
	There are isomorphisms
	\[\cZ_\pTac\cong B(\bar R\rtimes \mu_2)\times [\PP^1/\Gamma] \and [Z_\pTac/\bar G]\cong B(\bar R\rtimes \mu_2)\times \cP(1,2)\]
	Under these identifications, the natural morphism between them is given by 
	\[[\PP^1/\Gamma]\lra \cP(1,2),\quad [a:b]\mapsto [a+b:ab].\]
\end{proposition}
\begin{proof}
	By Lemmas~\ref{lem:Rfixedlocus} and~\ref{lem:Rfixedlocus2}, $(U^{ss})^R\cong \A^2_{s,p}\setminus\{0\}$ and $(U^{1/2,ss})^R\cong [(\A^2_{a,b}\setminus\{0\})/\Gamma]$.
	Since the $N_{\bar G}(\bar R)$-actions on them factor through the character $\charthetabar:N_{\bar G}(\bar R)\to \Gm$, the corresponding quotients are $\bar R\rtimes\mu_2$-gerbes over $\cP(1,2)$ and $[\PP^1/\Gamma]$, respectively.
	Since the exact sequence in the lower row of \eqref{eq:exact.TR} splits, these gerbes are trivial.
	This proves the first statement.
	
	The second statement follows from Lemma~\ref{lem:Rfixedlocus2}.
\end{proof}

\begin{remark}\label{rem:P1}
Under the above identifications, the good moduli space map $\cPK\to \PK$ restricts to the projection $\cZ_\pTac\to \PP^1$. The image $\PP^1$ is precisely the exceptional locus of $\PK\to \PCY$, contracted to a point; see \cite[Proposition~21]{HL} and \cite[Section~6.2.1]{ADL}.
\end{remark}
We record for later use the following result on the intersection $\cZ_\pTac\cap\cEK$.

\begin{proposition}\label{prop:transverse}
	$\cZ_\pTac\cap \cEK$ consists of a single  point and is transverse.
\end{proposition}
\begin{proof}
	This is immediate from the second statement of Lemma~\ref{lem:Rfixedlocus}.
\end{proof}

\begin{definition}
	We denote by $\pHT$ the hyperelliptic tacnodal point, i.e. the point $\cZ_\pTac \cap \cEK$ corresponding to the K-polystable pair
	\beq\label{eq:HT}\{w^2=u^4v^4\}\subset \PP(1,1,4).\eeq
	It is the unique closed point of $\cZ_\pTac$ lying over $[1:1]\in [\PP^1/\Gamma]$.
\end{definition}

\section{Normal bundle of the strictly polystable locus}\label{s:normal}

In this section, we study the normal bundle of the strictly polystable locus $\cZ_\pTac\subset \cPK$.
Our main result is a decomposition of $N_{\cZ_\pTac/\cPK}$ into indecomposable rank $2$ bundles (Proposition~\ref{prop:normalbundle.decomp}). 
This decomposition is a key ingredient in the construction of $\cPhat$ (Section~\ref{s:resolution}) and in the computation of its Chow ring (Section~\ref{s:chow.phat}).

To state the result, we work with an \'etale double cover of $\cZ_\pTac$, 
\[\widetilde\cZ_\pTac:=\big[(U^{1/2,ss})^R/\bar T\big]\cong B\bar R\times [\PP^1/\Gamma],\]
on which the normal bundle splits into line subbundles corresponding to the $\bar T$-characters.

For a character $\chi:\bar T\to \Gm$ and the associated $\bar T$-representation $\C_\chi$,
let
\[\cO_{\widetilde \cZ_\pTac}(\chi):=\big[\big((U^{1/2,ss})^R\times \C_\chi\big)/\bar T\big] ~\in \Pic \widetilde\cZ_\pTac\]
denote the associated line bundle. 
Define $\mu,\nu\in \Hom(\bar T,\Gm)$ by
\[\mu(t)=t_1^{-1}t_3 \and \nu(t)=t_2^{-1}t_3\]
for $t=\diag(t_1,t_2,t_3)\in \bar T$. 
Then $N_{\bar G}(\bar R)/\bar T\cong \mu_2$ interchanges $\mu$ and $\nu$, so
\[ N_j:=\cO_{\widetilde\cZ_\pTac}(j\mu)\oplus \cO_{\widetilde\cZ_\pTac}(j\nu),\quad j\in \Z,\]
descends to a rank $2$ vector bundle on $\cZ_\pTac$, which we also denote by $ N_j$.

\begin{proposition}\label{prop:normalbundle.decomp}
	The normal bundle $N_{\cZ_\pTac/\cPK}$ 
	decomposes as
	\[N_{\cZ_\pTac/\cPK}\cong   N_2\oplus  N_3\oplus  N_4,\]
	where each $ N_j$ is the descended rank $2$ bundle defined above.
\end{proposition}
While the decomposition of each fiber of $N_{\cZ_\pTac/\cPK}$ into weight spaces with respect to the identity component of the stabilizer (isomorphic to $\Gm$) is immediate, it is not clear how these weight spaces fit together globally over $\cZ_\pTac$. The key point is to identify the corresponding subbundles with the descended bundles $ N_2, N_3, N_4$.

\smallskip

The proof proceeds in several steps. 
We first show that $N_{\cZ_\pTac/\cPK}$ descends along the square-root stack morphism described in Proposition~\ref{prop:gerbe}:
\[\varrho:\cZ_\pTac\lra \big[Z_\pTac/\bar G\big]\cong \big[(U^{ss})^R/N_{\bar G}(\bar R)\big].\]
\begin{lemma}\label{lem:normal.reduction}
	There is a canonical isomorphism $N_{\cZ_\pTac/\cPK}\cong \varrho^*\cN$, where 
	\[\cN:=\big[(N_{Z_\pTac/U^{ss}}|_{(U^{ss})^R})/N_{\bar G}(\bar R)\big].\]
\end{lemma}
\begin{proof}
	The morphism $\varrho$ is the square-root stack morphism along the divisor induced by $E_\DC$. Since $Z_\pTac$ intersects $E_\DC$ transversely by Lemma~\ref{lem:Rfixedlocus}, it follows that the normal bundle $N_{\cZ_\pTac/\cPK}$ is pulled back from the normal bundle of $[Z_\pTac/\bar G]\subset [U^{ss}/\bar G]$.
	Hence it suffices to identify the latter normal bundle with $\cN$. This follows from the fact that the natural map
	\[G\times^{N_{G}(R)}\big(N_{Z_\pTac/U^{ss}}|_{(U^{ss})^R}\big)\lra N_{Z_\pTac/U^{ss}}, \quad [g,v]\longmapsto g\cdot v\]
	is an isomorphism of $G$-vector bundles over $G\times^{N_{G}(R)}(U^{ss})^R\cong Z_\pTac$.
\end{proof}

Since $(U^{ss})^R\subset Z_\pTac\subset U^{ss}$, there is an $N_G(R)$-equivariant exact sequence
\beq\label{eq:ses.normal}
0\lra N_{(U^{ss})^R/Z_\pTac}\lra N_{(U^{ss})^R/U^{ss}}\lra N_{Z_\pTac/U^{ss}}|_{(U^{ss})^R}\lra 0.
\eeq
Thus, to compute $\cN$, it suffices to understand the first two terms of \eqref{eq:ses.normal}.

To describe these bundles equivariantly, we introduce the following notation.
For a character $\chi$ of $\bar T$ and a $\bar T$-scheme $Z$, we denote by
\[\cO_Z(\chi):=\cO_Z\otimes \C_\chi\] 
the associated $\bar T$-equivariant line bundle on $Z$. 
We use the same notation with $T$ in place of $\bar T$, and omit the subscript when the underlying scheme is clear from the context.
\begin{lemma}\label{lem:normal.g/t}
	There is a canonical $N_{G}(R)$-equivariant isomorphism 
	\[N_{(U^{ss})^R/Z_\pTac}\cong\cO_{(U^{ss})^R}\otimes(\fg/\ft),\]
	where $\fg$ and $\ft$ denote the Lie algebras of $G$ and $T$ respectively, endowed with the adjoint action of $N_{G}(R)$. 
	In particular,
	\[N_{(U^{ss})^R/Z_\pTac}\cong \big(\cO(\mu)\oplus \cO(\nu)\big)\oplus \big(\cO(-\nu)\oplus \cO(-\mu)\big)\oplus \big(\cO(\mu-\nu)\oplus \cO(\nu-\mu)\big).\]
\end{lemma}
\begin{proof}
	The closed immersion $(U^{ss})^R\hookrightarrow Z_\pTac$ can be identified with
	\[
	(U^{ss})^R\cong N_{G}(R)\times^{N_{G}(R)}(U^{ss})^R\hooklongrightarrow  G\times^{N_{G}(R)}(U^{ss})^R \cong Z_\pTac.
	\]
	Hence its normal bundle is the trivial bundle with fiber the tangent space of $G/N_G(R)$ at the identity point, namely $\fg/\ft$.
\end{proof}

We compute $N_{(U^{ss})^R/U^{ss}}$ by comparison with $N_{(X^{ss})^R/X^{ss}}$.
We begin by describing the latter explicitly, and then account for the correction along the exceptional divisor.
Let $\C[x,y,z]_4$ denote the space of quartic polynomials, with the natural $G$-action.

\begin{lemma}\label{lem:normal.Rfixed.X}
	There is a canonical $N_G(R)$-equivariant isomorphism
	\[N_{(X^{ss})^R/X^{ss}}\cong\frac{\C[x,y,z]_4}{\langle x^2y^2,xyz^2,z^4\rangle}\otimes \cO_X(1)\big|_{(X^{ss})^R}.\]
	The line bundle $\cO_X(1)\big|_{(X^{ss})^R}$ has the trivial $R\rtimes\mu_2$-linearization. 
	Consequently, the normal bundle admits a decomposition
	\[N_{(X^{ss})^R/X^{ss}}\cong (M'_{-1}\oplus M'_{1})\oplus (M'_{-2}\oplus M'_{2})\oplus (L'_{-3}\oplus L'_{3})\oplus (L'_{-4}\oplus L'_{4}), \]
	where $M'_j$ (resp.\ $L'_j$) is a $T$-equivariant vector bundle of rank $2$ (resp.\ rank $1$) with $R$-weight $j$,
	and the involution $\mu_2$ acts by interchanging the summands of weights $-j$ and $j$.
	More explicitly, these summands are given by
	\[
	\begin{aligned}
		M'_{-1}&\cong \C\cdot x^2yz \otimes\cO_X(1)\oplus \C\cdot xz^3\otimes \cO_X(1),\\
		M'_{1}&\cong \C\cdot xy^2z \otimes \cO_X(1)\oplus \C\cdot yz^3\otimes \cO_X(1),\\
		M'_{-2}&\cong \C\cdot x^3y \otimes \cO_X(1)\oplus \C\cdot x^2z^2\otimes \cO_X(1),\\
		M'_{2}&\cong \C\cdot xy^3 \otimes \cO_X(1)\oplus \C\cdot y^2z^2\otimes \cO_X(1),\\
		L'_{-3}&\cong \C\cdot x^3z\otimes \cO_X(1),\\
		L'_{3}&\cong \C\cdot y^3z\otimes \cO_X(1),\\
		L'_{-4}&\cong \C\cdot x^4\otimes \cO_X(1),\\
		L'_{4}&\cong \C\cdot y^4\otimes \cO_X(1).
	\end{aligned}
	\]
\end{lemma}

\begin{proof}
	The first assertion follows immediately from the Euler sequence for $X$, since $X^R\cong \PP^2$ is the projectivization of the subspace
	\[\langle x^2y^2, xyz^2, z^4\rangle \subset \C[x,y,z]_4.\]
	Moreover, every point $[f]\in X^R$ is fixed by $R\rtimes\mu_2$, so the fiber $\cO_X(-1)\big|_{[f]}\cong \C\cdot f$ is the trivial $R\rtimes \mu_2$-representation. Therefore, $\cO_X(1)\big|_{(X^{ss})^R}$ has the trivial $R\rtimes\mu_2$-linearization.
\end{proof}

We compare $N_{(U^{ss})^R/U^{ss}}$ and $N_{(X^{ss})^R/X^{ss}}$.
The natural morphism
\beq\label{eq:normaltonormal}N_{(U^{ss})^R/U^{ss}}\lra N_{(X^{ss})^R/X^{ss}},\eeq
is an isomorphism away from $\Delta_R$, and its cokernel is supported on $\Delta_R$. 
\begin{lemma}\label{lem:normaltonormal}
	The map \eqref{eq:normaltonormal} fits into a short exact sequence
	\beq\label{eq:ses.normaltonormal}0\lra N_{(U^{ss})^R/U^{ss}}\lra N_{(X^{ss})^R/X^{ss}}\lra \iota_{\Delta_R*}E_\mathrm{exc}\lra 0,\eeq
	where $\iota_{\Delta_R}:\Delta_R\hookrightarrow (U^{ss})^R$ is the inclusion, and $E_\mathrm{exc}$ is the excess bundle
	\[E_\mathrm{exc}:=\Big(N_{(X^{ss})^R/X^{ss}}\big|_{\Delta_R}\Big)\big/N_{\Delta_R/Z_\DC}\]
	associated to the Cartesian diagram
	\beq\label{eq:diagram.DeltaR}\begin{tikzcd}
		\Delta_R\arrow[r,hook]\arrow[d,hook']&Z_\DC\arrow[d,hook']\\ (X^{ss})^R\arrow[r,hook]&X^{ss}.
	\end{tikzcd}\eeq
	Moreover, the short exact sequence \eqref{eq:ses.normaltonormal} is $N_G(R)$-equivariant, and the rank $8$ bundle $E_{\mathrm{exc}}$ has $R$-weights $\pm1,\pm2,\pm3,\pm 4$, each with multiplicity one.
\end{lemma}
\begin{proof}
	Since the map \eqref{eq:normaltonormal} is an isomorphism away from $\Delta_R$ and $N_{(U^{ss})^R/U^{ss}}$ is locally free, its kernel is torsion, hence zero.
	Over $\Delta_R$, \eqref{eq:normaltonormal} restricts to
	\[N_{\Delta_R/E_\DC^{ss}}\lra N_{(X^{ss})^R/X^{ss}}|_{\Delta_R}\]
	where $E_\DC^{ss}=E_\DC\cap U^{ss}$. 
	By \eqref{eq:diagram.DeltaR}, the image identifies with $N_{\Delta_R/Z_\DC}$, and hence the cokernel is $E_\mathrm{exc}$. 
	Therefore the cokernel of \eqref{eq:normaltonormal} is $\iota_{\Delta_R*}E_\mathrm{exc}$.
	
	The short exact sequence \eqref{eq:ses.normaltonormal} is $N_G(R)$-equivariant. To compute the $R$-weights of the excess bundle, observe that $N_{\Delta_R/Z_\DC}$ has $R$-weights $\pm1,\pm2$, each with multiplicity one. Indeed, its fiber at the point $[(xy+az^2)^2]$ with $a\neq 0$ is canonically identified with the subspace
	\[(xy+az^2)\cdot\langle x^2,xz,yz,y^2\rangle\]
	inside the tangent space
	$\C[x,y,z]_4\big/\langle (xy+az^2)^2\rangle$
	of $Z_\DC$ at $[(xy+az^2)^2]$.
	The last assertion follows from the fact that $N_{(X^{ss})^R/X^{ss}}$ has $R$-weights $-1,1,-2,2,-3,3,-4,4$ with multiplicities $2,2,2,2,1,1,1,1$ respectively.
\end{proof}

To refine \eqref{eq:ses.normaltonormal} according to $R$-weights, recall that $N_G(R)=T\rtimes\mu_2$ and $R$ is central in $T$. On a $T$-scheme with trivial $R$-action, any $T$-equivariant vector bundle decomposes canonically into its $R$-weight subbundles. 

\begin{lemma}\label{lem:normal.Rfixed.U}
	The normal bundle $N_{(U^{ss})^R/U^{ss}}$ decomposes as
	\[N_{(U^{ss})^R/U^{ss}}\cong (M_{-1}\oplus M_{1})\oplus (M_{-2}\oplus M_{2})\oplus (L_{-3}\oplus L_{3})\oplus (L_{-4}\oplus L_{4}), \]
	as an $N_G(R)$-equivariant vector bundle,
	where $M_j$ (resp.\ $L_j$) is a $T$-equivariant vector bundle of rank $2$ (resp.\ rank $1$) with $R$-weight $j$. 
	The involution $\mu_2$ acts by interchanging the summands of weights $-j$ and $j$.
	Moreover,
	\[\det M_j=\det M'_j\otimes \cO(-\Delta_R) \and L_j=L'_j\otimes \cO(-\Delta_R)\]
	as $T$-equivariant line bundles.
\end{lemma}
\begin{proof}
	Let $(E_{\mathrm{exc}})_j$ denote the $T$-equivariant $R$-weight-$j$ subbundle of $E_{\mathrm{exc}}$. 
	The short exact sequence \eqref{eq:ses.normaltonormal} decomposes into eight short exact sequences, one for each $R$-weight $j$:
	\[\begin{aligned}
		0\lra M_j\lra M_j'\lra \iota_{\Delta_R*}(E_{\mathrm{exc}})_j\lra 0 \quad \text{for }j=\pm1,\pm2;\\
		0\lra L_j\lra L_j'\lra \iota_{\Delta_R*}(E_{\mathrm{exc}})_j\lra 0 \quad \text{for }j=\pm3,\pm4.
	\end{aligned}\]
	This proves all assertions except the statement on determinants.

	By Lemma~\ref{lem:normaltonormal}, each $(E_{\mathrm{exc}})_j$ is a line bundle on $\Delta_R$. Since $\Delta_R$ is a single $T$-orbit with stabilizer $R$, any $T$-equivariant line bundle on $\Delta_R$ is uniquely determined by its $R$-weight. In particular, $(E_{\mathrm{exc}})_j$ is isomorphic to the restriction of any $T$-equivariant line bundle on $(U^{ss})^R$ of $R$-weight $j$, for instance $\cO_{(U^{ss})^R}(j\mu)\big|_{\Delta_R}$. 
	Therefore, $\det(\iota_{\Delta_R*}(E_{\mathrm{exc}})_j)\cong \cO(\Delta_R)$.
	
	The final assertion follows from the eight short exact sequences above.
\end{proof}

Combining \eqref{eq:ses.normal} with Lemmas~\ref{lem:normal.g/t} and~\ref{lem:normal.Rfixed.U}, we compute the $N_G(R)$-equivariant decomposition of $N_{Z_\pTac/U^{ss}}|_{(U^{ss})^R}$, and hence of $\cN$.

To determine the quotient in \eqref{eq:ses.normal}, we rewrite the first two terms in terms of $\mu$ and $\nu$ and compare them. 
We first express $\cO_{(U^{ss})^R}(\Delta_R)$ and $\cO_X(1)|_{(U^{ss})^R}$ in terms of $\chartheta=-\frac{1}{3}(\mu+\nu)$.

\begin{lemma}\label{lem:linear.relation}
	There are $N_{G}(R)$-equivariant isomorphisms
	\[\cO_{(U^{ss})^R}(\Delta_R)\cong \cO_{(U^{ss})^R}(6\chartheta) \and \cO_X(1)|_{(U^{ss})^R}\cong \cO_{(U^{ss})^R}(2\chartheta).\]
\end{lemma}
\begin{proof}
	An element $t=\diag(t_1,t_2,t_3)\in T$ acts on $(s,p)\in (U^{ss})^R$ by
	\[t\cdot(s,p)=(t_3^{-3}s,t_3^{-6}p)=(\chartheta(t)^3s,\chartheta(t)^6p)\]
	since $t_1t_2t_3=1$. The first isomorphism follows since $\Delta_R\subset (U^{ss})^R$ is cut out by $s^2-4p=0$.

	For the second isomorphism, recall that $X=\PP(\C[x,y,z]_4)$. 
	The total space of $\cO_X(1)$ can be identified with the complement 
	\[\PP(\C[x,y,z]_4\oplus \C)\setminus\{[0:1]\}\] 
	of the vertex in the projective cone over $X$.
	Restricting to $X^R=\PP(\C[x,y,z]_4^R)$, the total space of $\cO_X(1)|_{X^R}$ is $N_G(R)$-equivariantly identified with a subscheme of $\PP(\C[x,y,z]_4^R\oplus \C)=\PP^4$.
	
	Writing homogeneous coordinates on $\PP(\C[x,y,z]_4^R\oplus \C)$ as $[1:s:p:c]$, 
	$t\in T$ acts by
	\[t\cdot[1:s:p:c]=[t_1^{-2}t_2^{-2}:t_1^{-1}t_2^{-1}t_3^{-2} s:t_3^{-4}p:c]=[1:\chartheta(t)^3s:\chartheta(t)^6p:\chartheta(t)^2c].\]
	This shows that $\cO_X(1)|_{(X^{ss})^R}$ corresponds to $2\chartheta$, completing the proof.
\end{proof}

\begin{proposition}\label{prop:normalbundle.R-fixed}
	There is an $N_G(R)$-equivariant isomorphism
	\[N_{Z_\pTac/U^{ss}}\big|_{(U^{ss})^R}\cong \bigoplus_{j=2,3,4}\big(\cO(j\mu)\oplus \cO(j\nu)\big).\]
	Consequently, the descended bundle $\cN$ inherits the same decomposition.
\end{proposition}

\begin{proof}
	From Lemmas~\ref{lem:normal.Rfixed.U} and~\ref{lem:linear.relation}, we obtain, as $T$-equivariant bundles,
	\[\begin{aligned}
		&\det M_{-1}\cong \cO(\mu-\nu), \quad
		&&\det M_{1}\cong \cO(\nu-\mu),\\
		&\det M_{-2}\cong \cO(3\mu-\nu),\quad
		&&\det M_{2}\cong \cO(3\nu-\mu),\\
		&\qquad L_{-3}\cong \cO(3\mu),
		&&\qquad L_{3}\cong \cO(3\nu),\\
		&\qquad L_{-4}\cong \cO(4\mu),
		&&\qquad L_{4}\cong \cO(4\nu).
	\end{aligned}
	\]

	The short exact sequence \eqref{eq:ses.normal} is compatible with the $R$-weight decomposition.
	For weights $\pm1$, both $N_{(U^{ss})^R/Z_\pTac}$ and $N_{(U^{ss})^R/U^{ss}}$ have rank $2$, so these weights do not contribute to $N_{Z_\pTac/U^{ss}}|_{(U^{ss})^R}$.
	For weights $\pm2$, the bundle $N_{(U^{ss})^R/Z_\pTac}$ contributes $\cO(\mu-\nu)\oplus \cO(\nu-\mu)$, while $N_{(U^{ss})^R/U^{ss}}$ contributes $M_{-2}\oplus M_2$.
	Hence the quotient is given by
	\[\left(\det M_{-2}\otimes \cO(\mu-\nu)^*\right)\oplus \left( \det M_2\otimes \cO(\nu-\mu)^*\right)\cong \cO(2\mu)\oplus \cO(2\nu).\]

	For weights $\pm3$ and $\pm4$, the bundle $N_{(U^{ss})^R/Z_\pTac}$ has no such weight spaces, so the quotient is simply
	$L_{-3}\oplus L_3\oplus L_{-4}\oplus L_4\cong\cO(3\mu)\oplus \cO(3\nu)\oplus \cO(4\mu)\oplus \cO(4\nu)$.
	This proves the claim.
\end{proof}

\begin{proof}[Proof of Proposition~\ref{prop:normalbundle.decomp}]
	This is now immediate from Lemma~\ref{lem:normal.reduction} and Proposition~\ref{prop:normalbundle.R-fixed}.
\end{proof}

\section{Resolution of the flip}\label{s:resolution}
In this section, we use the decomposition of $N_{\cZ_\pTac/\cPK}$ obtained in Section~\ref{s:normal} to construct the smooth Deligne--Mumford stack $\cPhat$ as a Kirwan-type stack-theoretic weighted blowup of $\cPK$ along $\cZ_\pTac$.
We then construct a morphism $\cPhat\to \cPH$ and identify 
$\cPhat$ with the stack-theoretic weighted blowup of $\cPH$ along $\cZ_2$ with weight two. 
In particular, this proves Theorem~\ref{thm:mainthmintro}.

\subsection{Construction of $\cPhat$}
We refer the reader to \cite[Definition~5.10]{ADL} or \cite{abramovich-temkin-wlodarczyk,QuekRydh} for the definition of a stack-theoretic weighted blowup.

\begin{definition} \label{def:Phat}
Let $N_{\cZ_\pTac/\cPK}\cong \bigoplus_{i=2,3,4} N_i$ as in Proposition~\ref{prop:normalbundle.decomp}.
Define
\[\cPhat'\lra \cPK\]
to be the stack-theoretic weighted blowup along $\cZ_\pTac$ with weights $2,3,4$ on the summands $ N_2, N_3, N_4$, 
namely the weighted blowup determined by the weight filtration on the ideal of $\cZ_\pTac$ induced by the above decomposition.
Let $\cEhat'$ denote its exceptional divisor.
Define
\[\cPhat:=(\cPhat')^{ss}\and \cEhat:=(\cEhat')^{ss}\]
to be the semistable loci with respect to the line bundle
\[\cO_{\cPGIT}(m_1)|_{\cPhat'}\otimes \cO_{\cPK}(-m_2\cEK)|_{\cPhat'}\otimes \cO_{\cPhat'}(-\cEhat')\]
for $m_1\gg m_2\gg1$ in the sense of \cite[Section~11]{alper}.

As quotient stacks,
\[\cPhat' \cong [\Uhat/\bar G] \and \cEhat' \cong [\Ehat/\bar G],\]
where $\Uhat\to U^{1/2,ss}$ is the stack-theoretic weighted blowup along $Z_\pTac^{1/2}$ with weights $2,3,4$, with a natural $\bar G$-action, and $\Ehat\subset \Uhat$ is the exceptional divisor.

Let $\Uhat^{ss}\subset \Uhat$ denote the GIT semistable locus with respect to 
\[\cO_X(m_1)\big|_{\Uhat}\otimes \cO(-m_2E_\DC)\big|_{\Uhat}\otimes \cO(-\Ehat),\]
for $m_1\gg m_2\gg 1$, and let $\Ehat^{ss}:=\Ehat \cap \Uhat^{ss}$.
Then
\[\cPhat \cong [\Uhat^{ss}/\bar G] \and \cEhat \cong [\Ehat^{ss}/\bar G].\]

We write $\Psi:\cPhat\to \cPK$ for the natural projection.
\end{definition}

\begin{remark}\label{rem:cPhat.local}
Near $\cEhat'$, the morphism $\cPhat'\to \cPK$ is \'etale locally modeled on the stack-theoretic weighted blowup of $[\A^6/(R\rtimes \mu_2)]\times B$ along the closed substack $B(R\rtimes \mu_2)\times B$, with weights $(2,2,3,3,4,4)$, where $B=\A^1$ or $[\A^1/\Gamma]$.
In particular, $\cPhat'$ is smooth, and hence so is its open substack $\cPhat$.

Moreover, by \cite{Kir85}, $\cPhat$ is a proper Deligne--Mumford stack, and its coarse moduli space $\Phat$ is the weighted Kirwan partial desingularization of $\PGIT$.
\end{remark}

We next describe the exceptional divisor of the morphism $\Psi:\cPhat\to\cPK$.
\begin{proposition}\label{prop:cEhat}
	The exceptional divisor $\cEhat$ is a $\mu_2$-gerbe over the stack
	\beq\label{eq:exceptional.fiber}\cEhat^\dagger:=\big[\big(\cP(2,3,4)\times \cP(2,3,4)\big)/\mu_2\big]\times [\PP^1/\Gamma].\eeq
	Its coarse moduli space $\Ehat$ is isomorphic to $Z_2\times \PP^1$.
\end{proposition}
\begin{proof}
	By Proposition~\ref{prop:normalbundle.decomp}, the normal bundle $N_{\cZ_\pTac/\cPK}$ decomposes as
	\[\big(\cO(2\mu)\oplus \cO(3\mu)\oplus \cO(4\mu)\big)\oplus\big(\cO(2\nu)\oplus \cO(3\nu)\oplus \cO(4\nu)\big)\]
	with $R$-weights $-2,-3,-4,2,3,4$, where the weighted blowup is taken with weights $2,3,4,2,3,4$ on the six one-dimensional weight spaces.
	
	Fix a closed point of $[\PP^1/\Gamma]$.
	The fiber of $\cEhat'$ over this point is 
	\beq\label{eq:normal.space.Zptac}
	\Big[\big(\A^3\times\A^3\big)/\big((R\rtimes\mu_2)\times\Gm\big)\Big],
	\eeq
	where $(t,u)\in R\times\Gm\cong\Gm^2$ acts by
	\[(t,u)\cdot(a_2,a_3,a_4,b_2,b_3,b_4)
	=(t^{-2}u^2a_2,t^{-3}u^3a_3,t^{-4}u^4a_4, t^{2}u^{2}b_2,t^{3}u^{3}b_3,t^{4}u^{4}b_4),\]
	and $\mu_2$ interchanges $a_j$ and $b_j$ for each $j$.

	By the Hilbert--Mumford criterion \cite[Theorem~2.1]{GIT}, a point is unstable precisely when either all coordinates of positive $R$-weight vanish or all coordinates of negative $R$-weight vanish.
	Thus the semistable locus is
	\beq\label{eq:exceptional.fiber2}
	\Big[\big((\A^3\setminus\{0\})\times(\A^3\setminus\{0\})\big)/\big((R\rtimes\mu_2)\times\Gm\big)\Big].
	\eeq
	This quotient stack is precisely the fiber of $\cEhat\to[\PP^1/\Gamma]$.

	The action of $R\times\Gm$ factors through the homomorphism
	\[R\times\Gm \lra \Gm^2,\qquad (t,u)\longmapsto (t^{-1}u,tu),\]
	whose kernel is the central subgroup $\mu_2$ generated by $(-1,-1)$.
	Under this identification, $t^{-1}u$ (resp.\ $tu$) acts on the first (resp.\ second) factor $\A^3\setminus\{0\}$ with weights $2,3,4$.
	It follows that the stack \eqref{eq:exceptional.fiber2}  
	is a $\mu_2$-gerbe over $\big[\big(\cP(2,3,4)\times \cP(2,3,4)\big)/\mu_2\big]$.

	To complete the proof, it remains to show that the fibration $\cEhat\to[\PP^1/\Gamma]$ is trivial.
	Each summand $\cO(j\mu)$ in $\cO(2\mu)\oplus \cO(3\mu)\oplus \cO(4\mu)$ is the $j$-th power of the fixed line bundle $\cO(\mu)$, and the exponent $j$ coincides with the corresponding $\Gm$-weight.
	It follows that the weighted projectivization of $\cO(2\mu)\oplus \cO(3\mu)\oplus \cO(4\mu)$ is a trivial $\cP(2,3,4)$-bundle over $[\PP^1/\Gamma]$.
	The same argument applies to $\cO(2\nu)\oplus\cO(3\nu)\oplus \cO(4\nu)$. 
	This proves the proposition.
\end{proof}

\subsection{Construction of the morphism $\Phi$}
Let
\beq\label{eq:univ.fam.K}
\pi^\rmK:(\cS^\rmK,\cC^\rmK)\lra \cPK
\eeq
be the universal family of K-semistable pairs over $\cPK$.
We construct the morphism $\Phi:\cPhat \to \cPH$ by modifying the pullback of the universal family \eqref{eq:univ.fam.K} along the exceptional divisor $\cEhat$.

The fibers of $\cC^\rmK$ over $\cZ_\pTac$ have two tacnodes. 
Pulling back to $\cPhat$, we get
\[\cZhat\lra\cEhat\]
whose fibers consist of two distinct points. 
Let
$(\cShatK,\cChatK):=(\cS^\rmK,\cC^\rmK)|_{\cPhat}$
be the pullback of \eqref{eq:univ.fam.K} to $\cPhat$. 
Then $\cZhat$ is a closed substack in $\cChatK$ and $\cShatK$, of codimension $2$ and $3$, respectively.
\[\begin{tikzcd}
	\cZhat\arrow[r,hook]\arrow[d]&(\cShatK,\cChatK)\arrow[d]\arrow[r]&(\cS^\rmK,\cC^\rmK)\arrow[d]\\\cEhat\arrow[r,hook]&\cPhat \arrow[r,"\Psi"]&\cPK
\end{tikzcd}\]
Let $\cShat$ and $\cChat$ be the weighted blowups of $\cShatK$ and $\cChatK$ along $\cZhat$ with weights $(1,1,2)$ and $(1,2)$, respectively. 
Here the weight-two coordinate corresponds to the direction transverse (in the surface fiber) to the common tangent line of the tacnodal branches, so that the proper transforms of the two branches are disjoint along the exceptional divisor.
\[(\cShat,\cChat) \xrightarrow{\Bl^{w}_{\cZhat}} (\cShatK,\cChatK)\lra \cPhat.\]
The exceptional divisor $\cFhat$ of $\cShat\to \cShatK$ has fibers $\PP(1,1,2)\sqcup \PP(1,1,2)$ over $\cEhat$.

The blowups only affect the fibers over $\cEhat$.
We describe these fibers.

\begin{proposition}\label{prop:fiber.Shat}
	Fix a closed point of $\cEhat$, and let $S^\rmK$ be the fiber of $\cShatK$ over this point, so that $S^\rmK\cong \PP^2$ or $\PP(1,1,4)$.
	Let $\Shat$ be the fiber of $\cShat$ over this point.
	Then
	\[\Shat = \Shat_1 \cup_{\ell_1} \Shat_0 \cup_{\ell_2}\Shat_2,\]
	where
	\begin{itemize}
		\item $\Shat_0$ 
		is the weighted blowup of $S^\rmK$ at two distinct smooth points $q_1,q_2$
		with weights $(1,2)$.
	\item Let $\ell_i\subset \Shat_0$ denote the exceptional divisor over $q_i$.
	Then $\Shat_i$ is glued to $\Shat_0$ along $\ell_i$, viewed as a line in $\PP(1,1,2)$ through the cone point.
\end{itemize}
\end{proposition}
\begin{proof}
	The stack $\cShatK$ is smooth along $\cZhat$.
	Fix one of the two points in a fiber of $\cZhat$ over $\cEhat$.
	\'Etale locally at this point, 
	the morphism $\cShatK 
	\to \cPhat$ is isomorphic to the projection
	\[\id_E\times \pr_1: E\times \A^3 \lra E\times \A^1,\]
	for some smooth scheme $E$.  
	Under this identification, $\cZhat$ and $\cEhat$ correspond to $E\times\{0\}$ in $E\times \A^3$ and $E\times \A^1$, respectively.
	Thus the morphism $\cShat\to \cPhat$ is \'etale-locally isomorphic to 
	\[E\times \Bl^{w}_0\A^3 \lra E\times \A^1\]
	where $\Bl^w_0\A^3$ denotes the weighted blowup of $\A^3$ at $0$ with weights $(1,1,2)$.
		
	Hence the fiber of this morphism over a point in $E\times\{0\}$ is the union of two components:
	the weighted blowup $\Bl^{w}_0\A^2$ of $\A^2$ at the origin with weights $(1,2)$, and a copy of $\PP(1,1,2)$, glued along the exceptional divisor of $\Bl^{w}_0\A^2$ and the line in $\PP(1,1,2)$ through its cone point.

	Applying this to the two points in the fiber of $\cZhat\to \cEhat$, we obtain the desired description.
\end{proof}

Recall that the restriction $\cChatK|_{\cEhat} \to \cEhat$ parametrizes tacnodal curves admitting $\Gm$-stabilizers, namely cateyes, oxes, and the curve \eqref{eq:HT} in $\PP(1,1,4)$ parametrized by $p_{\HT}$.
\begin{proposition}\label{prop:fiber.Chat}
	Fix a closed point of $\cEhat$, and let $C^\rmK$ be the fiber of $\cChatK$ over this point, which is the union of two branch curves $C^\rmK_1$ and $C^\rmK_2$ meeting at two tacnodes $q_1$ and $q_2$.
	Let $\Chat$ be the fiber of $\cChat$ over this point.
	Then
	\[\Chat = \Chat_1 \cup \Chat_0 \cup \Chat_2,\]
	where
	\begin{itemize}
		\item $\Chat_0\subset \Shat_0$ is the disjoint union of the proper transforms of the two branch curves $C^\rmK_1$ and $C^\rmK_2$ under the weighted blowup $\Shat_0\to S^\rmK$.
		\item For $i=1,2$, $\Chat_i\subset \Shat_i\cong\PP(1,1,2)$ is a degree $4$ curve disjoint from the cone point and meeting $\ell_i$ transversely at two distinct points.
	\end{itemize}
	For each $i=1,2$, the pair $(\Shat_i,\Chat_i)$ is of the form~\eqref{eq:normal.form.Z2}.
\end{proposition}

\begin{proof}
	Fix a closed point of $\cEhat$ lying over $\cZ_\pTac\setminus\{\pHT\}$, and write its image in $[(\cP(2,3,4)\times \cP(2,3,4))/\mu_2]$ as $([a_2:a_3:a_4],[b_2:b_3:b_4])$.
	This point lies on the proper transform of the curve defined by the one-parameter family
	\[
	\left(xy+\frac{a+b}{2}z^2\right)^2 = \left(\frac{a-b}{2}\right)^2 \bigg(z^4+\sum_{j=2,3,4}a_jt^jx^jz^{4-j} + \sum_{j=2,3,4}b_jt^jy^jz^{4-j}\bigg)\]
	over $\A^1_t$, whose central fiber at $t=0$ is~\eqref{eq:plane.tacnodal.eq}. 

	Take the weighted blowup along $(q_1,0)$ and $(q_2,0)$ with weights $(1,1,2)$.
	On the exceptional divisor $\Shat_1\cong\PP(1,1,2)$ over $(q_1,0)$, the proper transform of the family 
	is defined by 
	\[\left(y+\frac{a+b}{2}z^2\right)^2 = \left(\frac{a-b}{2}\right)^2\big(z^4+a_2t^2z^2+a_3t^3z+a_4t^4\big).\]
	Up to automorphisms of $\PP(1,1,2)$, this curve is of the normal form~\eqref{eq:normal.form.Z2}, and similarly for the curve $\Chat_2$ over $(q_2,0)$ obtained by replacing $a_j$ with $b_j$.

	For a closed point of $\cEhat$ lying over $\pHT\in \cPK$, since $\cZ_\pTac$ and $\cEK$ intersect transversely, 
	the same argument applies to the one-parameter family of curves in $\PP(1,1,4)$ over $\A^1_t$ defined by
	\[w^2=u^4v^4+a_2t^2u^6v^2+a_3t^3u^7v+a_4t^4u^8+b_2t^2u^2v^6+b_3t^3uv^7+b_4t^4v^8.\]
	The same construction produces curves $\Chat_i \subset \Shat_i$ of the form~\eqref{eq:normal.form.Z2}.
\end{proof}

We construct a morphism $\Shat_0 \to \PP^1$ by a base-point free linear system.

\begin{lemma}\label{lem:contr.main.comp}
	Let $L_0$ be the line bundle on $\Shat_0$ associated to the proper transform of a branch curve of $C^\rmK$. Then $L_0$ is base-point free and satisfies
	\beq\label{eq:cohom.L0}
	\dim H^0(\Shat_0,L_0)=2\and H^i(\Shat_0,L_0)=0 \text{ for } i>0.
	\eeq
	In particular, the complete linear series $|L_0|$ is a pencil whose two distinguished members are the proper transforms of the two branch curves, and the induced morphism $\Shat_0 \to \PP^1$ restricts to an isomorphism $\ell_i \cong \PP^1$ for $i=1,2$.
\end{lemma}
\begin{proof}
	When $S^\rmK=\PP^2$ and $C^\rmK$ is a tacnodal curve of the form \eqref{eq:plane.tacnodal.eq}, the two branch curves are conics meeting in two tacnodes $q_1$ and $q_2$. Thus their proper transforms are linearly equivalent to
	$c_1(\cO_{\PP^2}(2))|_{\Shat_0}-m(\ell_1+\ell_2)$
	for some $m$. Intersecting with $\ell_i$ shows that $m=2$.
	
	Since $\ell_1$ and $\ell_2$ are $2$-Cartier divisors and the class of $C^\rmK$ is $2$-divisible, this divisor class is Cartier. Hence the line bundle $L_0$ is well-defined.

	Tensoring $L_0$ with the short exact sequence associated to the Cartier divisor $2\ell_1+ 2\ell_2$ and taking cohomology, we obtain
	\[0\lra H^0(\Shat_0,L_0)\lra H^0(\PP^2,\cO(2))\xrightarrow{~\res~} H^0(\cO_{2q_1}) \oplus H^0(\cO_{2q_2})\lra H^1(\Shat_0,L_0)\lra 0\]
	with $H^2(\Shat_0,L_0)=0$. Since $\res$ is surjective, this proves \eqref{eq:cohom.L0}.

	Moreover, $H^0(\Shat_0,L_0)$, viewed as a linear system on $\PP^2$, has base locus supported at $q_1$ and $q_2$, 
	so after the blowup the base locus of $L_0$ is supported on $\ell_1\sqcup \ell_2$. On the other hand, $L_0$ restricts to $\cO_{\PP^1}(1)$ on $\ell_i$, and the restriction map $H^0(\Shat_0,L_0)\to H^0(\ell_i,\cO(1))$ is surjective (in fact, an isomorphism).
	Consequently, $L_0$ is base-point free, and its complete linear system defines the desired morphism to $\PP^1$. This proves the assertion when $S^\rmK=\PP^2$.

	When $S^\rmK=\PP(1,1,4)$ and $C^\rmK$ is as in \eqref{eq:HT}, the same argument applies, replacing $\PP^2$ and $\cO_{\PP^2}(2)$ by $\PP(1,1,4)$ and $\cO_{\PP(1,1,4)}(4)$, respectively.
\end{proof}

We extend this construction to the whole fiber.
By abuse of notation, we identify the target $\PP^1$ of the pencil with $\ell$.

\begin{lemma}\label{lem:contr.fibers}
	Fix a closed point of $\cEhat$, and let $(\Shat,\Chat)$ be the fiber of $(\cShat,\cChat)$ over this point. 
	Suppose that $L$ is a line bundle on $\Shat$ such that
	\[L|_{\Shat_0}\cong L_0 \and L|_{\Shat_i}\cong \cO_{\PP(1,1,2)}(2) \quad\text{for }\,i=1,2.\]
	Then $L$ is base-point free and satisfies $H^i(\Shat,L)=0$ for $i>0$. 
	The complete linear system of $L$ defines a morphism
	\[\Shat \lra \Shat_1\cup_\ell \Shat_2\]
	contracting $\Shat_0$ onto $\ell$. Restricting to $\Chat$, we obtain the induced contraction
	\[\Chat \lra \Chat_1\cup \Chat_2.\]
\end{lemma}
\begin{proof}
	We first compute the cohomology of $L$.
	Tensoring the exact sequence
	\[0\lra \cO_{\Shat}\lra \cO_{\Shat_0}\oplus \cO_{\Shat_1}\oplus \cO_{\Shat_2}\lra \cO_{\ell_1}\oplus \cO_{\ell_2}\lra 0\]
	with $L$ and taking cohomology, we obtain an exact sequence
	\[\begin{split}
		0&\lra H^0(\Shat,L)\lra H^0(\Shat_0,L_0)\oplus \bigoplus_{i=1,2}H^0(\Shat_i,\cO(2))\lra \bigoplus_{i=1,2}H^0(\ell_i,\cO(1))\\
		&\lra H^1(\Shat,L)\lra 0,
	\end{split}
	\]
	and $H^2(\Shat,L)=0$ by Lemma~\ref{lem:contr.main.comp}. Since the restriction maps $H^0(\Shat_i,\cO(2))\to H^0(\ell_i,\cO(1))$ are surjective for $i=1,2$, it follows that $H^1(\Shat,L)=0$.
	
	We show that $L$ is base-point free.
	By Lemma~\ref{lem:contr.main.comp}, $L_0$ is base-point free on $\Shat_0$, and $\cO_{\PP(1,1,2)}(2)$ is very ample on each $\Shat_i$. Thus it suffices to show that the  restriction maps
	\[H^0(\Shat,L)\lra H^0(\Shat_0,L_0)\quad\text{and}\quad H^0(\Shat,L)\lra H^0(\Shat_i,\cO(2)),\ i=1,2, \]
	are surjective.
	This follows from the exact sequence and the surjectivity of 
	$H^0(\Shat_0,L_0)\to H^0(\ell_i,\cO(1))$ and $H^0(\Shat_i,\cO(2))\to H^0(\ell_i,\cO(1))$
	for $i=1,2$.
	
	The morphism defined by the complete linear system of $L$ restricts to the morphism $\Shat_0\to \ell$ constructed in Lemma~\ref{lem:contr.main.comp}, and to the immersion induced by $\cO_{\PP(1,1,2)}(2)$ on each $\Shat_i$ for $i=1,2$.
	Consequently, it contracts $\Shat_0$ onto $\ell$ and induces the corresponding contraction of $\Chat$.
\end{proof}

\begin{remark} \label{rem:flat.cShat} 
	Since the fibers of $\cShat\to \cPhat$ are Cohen--Macaulay of pure dimension two and $\cPhat$ is smooth, $\cShat$ is flat over $\cPhat$ by the miracle flatness theorem \cite[Exercise~26.2.H]{vakil}. Since $\cChat\subset \cShat$ is a Cartier divisor on each fiber, it follows that $\cChat\to \cPhat$ is also flat. 
\end{remark}

We globalize the fiberwise contractions of Lemma~\ref{lem:contr.fibers}.

\begin{proposition}\label{prop:contract}
	There exist morphisms
	\[\cShat \lra \cShatH \qquad\text{and}\qquad \cChat \lra \cChatH\]
	over $\cPhat$, which are isomorphisms over $\cPhat\setminus\cEhat$ and 
	whose fibers over $\cEhat$ agree with the contractions of $\Shat$ and $\Chat$ in Lemma~\ref{lem:contr.fibers}.
	The pair $\left(\cShatH,(\frac{3}{4}+\epsilon)\cChatH\right)$ defines a flat family of Hacking stable pairs over $\cPhat$ for all sufficiently small $\e>0$, and hence defines a morphism
	\[\Phi:\cPhat \lra \cPH\]
	which is an isomorphism over $\cPhat\setminus\cEhat$. Moreover, $\Phi^{-1}(\cZ_2)$ is supported on $\cEhat$.
\end{proposition}

\begin{proof}
	To globalize the fiberwise contractions in Lemma~\ref{lem:contr.fibers}, we use the line bundle $\cL\to \cS^\rmK$ in \cite[page~45]{ADL}, which satisfies $\cL^{\otimes 2}\cong \cO_{\cS^\rmK}(\cC^\rmK)$. Define
	\[\cLhat :=\cL|_{\cShat}\otimes \cO_{\cShat}(-2\cFhat)\]
	on $\cShat$,
	where $\cFhat$ is the exceptional divisor of the weighted blowup $\cShat\to \cShatK$. 
	The divisor $\cFhat$ is $2$-Cartier and supported over $\cEhat$, with fibers $\Shat_1\sqcup \Shat_2$. 
	
	The line bundle $\cLhat$ is relatively very ample over $\cPhat\setminus\cEhat$. 
	Over $\cEhat$, it restricts to a line bundle satisfying the assumptions 
	in Lemma~\ref{lem:contr.fibers}. 
	Hence, for every point $p\in \cPhat$, we have $H^i(\Shat_p,\cLhat|_{\Shat_p})=0$ for $i>0$, and $\cLhat|_{\Shat_p}$ is base-point free.

	Let $\hat\pi:\cShat\to\cPhat$ be the structure morphism. 
	By Remark~\ref{rem:flat.cShat}, $\hat\pi$ is flat.
	By cohomology and base change (\cite[Theorem~25.1.6]{vakil}), we have $R^i\hat\pi_*\cLhat=0$ for $i>0$, and $\hat\pi_*\cLhat$ is locally free and compatible with base change. In particular, the natural evaluation morphism
	\[\hat\pi^*\hat\pi_*\cLhat \lra \cLhat\]
	is surjective, i.e., $\cLhat$ is relatively base-point free over $\cPhat$. It defines a morphism
	\[\cShat \lra \cShatH\]
	over $\cPhat$.
	Its restriction to each fiber agrees with the contraction in Lemma~\ref{lem:contr.fibers}.
	The induced contraction $\cChat\to \cChatH$ is obtained by restricting to $\cChat$.	
	Moreover, by the argument of Remark~\ref{rem:flat.cShat}, $\cShatH$ and $\cChatH$ are also flat over $\cPhat$.

	Over $\cPhat\setminus\cEhat$ the fibers of $(\cShatH,\cChatH)$ coincide with those of $(\cS^\rmK,\cC^\rmK)$, and are therefore Hacking stable with respect to the coefficient $\frac{3}{4}+\epsilon$ (\cite[Theorem~9.18]{ADL}).
	On the other hand, by Proposition~\ref{prop:fiber.Chat} and Lemma~\ref{lem:contr.fibers}, the fibers over $\cEhat$ are isomorphic to the pairs parametrized by $\cZ_2$, and hence are also Hacking stable with respect to the same coefficient.
	Hence $(\cShatH,\cChatH)$ defines a morphism $\Phi:\cPhat\to\cPH$. It is an isomorphism over $\cPhat\setminus\cEhat$, and the scheme-theoretic inverse image $\Phi^{-1}(\cZ_2)$ is supported on $\cEhat$.
\end{proof}

\subsection{Identification of $\Phi$ with the weighted blowup}
We complete the proof of Theorem~\ref{thm:mainthmintro}. 
We begin by determining the inverse image $\Phi^{-1}(\cZ_2)$.

\begin{proposition}\label{prop:mult}
The scheme-theoretic inverse image $\Phi^{-1}(\cZ_2)$ is $2\cEhat$.
\end{proposition}
\begin{proof}
	For the one-parameter slices $\A^1_t\to\cPhat$ constructed in the proof of Proposition~\ref{prop:fiber.Chat}, the corresponding families are invariant under the involution $\iota:([x:y:z],t)\mapsto([-x:-y:z],-t)$ (or $([u:v:w],t)\mapsto ([u:-v:w],-t)$). Hence the composite map $\A^1_t\to\cPhat\to\cPH$ factors through $\A^1_t/\iota$. By the explicit description of the degeneration in Proposition~\ref{prop:fiber.Chat}, the scheme-theoretic inverse image of $\cZ_2$ in $\A^1_t$ is precisely the Cartier divisor $t^2=0$.

	Since such one-parameter families run over local curves in the normal direction to $\cEhat$ at every point of $\cEhat$, the vanishing order along the normal direction is everywhere equal to two. 
	Hence $\Phi^{-1}(\cZ_2)$ has no embedded components and coincides with the Cartier divisor $2\cEhat$.
\end{proof}

We next show that the $\mu_2$-gerbe structure along $\cEhat$ in $\cPhat$ can be eliminated.

\begin{lemma}\label{lem:Phi}
The relative coarse moduli space $\cPhat^\dagger\to \cPH$ of $\Phi$ exists, and $\Phi$ factors as
$\cPhat\to \cPhat^\dagger\to \cPH$
and restricts to $\cEhat\to \cEhat^\dagger\to \cZ_2$.
Moreover, $\cPhat^\dagger$ is smooth, $\cEhat^\dagger\subset \cPhat^\dagger$ is Cartier, 
and $\cPhat\to \cPhat^\dagger$ is the square-root stack along $\cEhat^\dagger$.
\end{lemma}
\begin{proof}
The morphism $\Phi$ is an isomorphism over $\cPhat\setminus \cEhat$, and along $\cEhat$, it factors as $\cEhat\to \cEhat^\dagger\to \cZ_2$, where $\cEhat\to \cEhat^\dagger$ is the $\mu_2$-gerbe in Proposition~\ref{prop:cEhat} and $\cEhat^\dagger\to \cZ_2$ is representable.
It follows that $\Phi$ has finite inertia and admits a relative coarse moduli space $\cPhat^\dagger \to \cPH$ by \cite[Theorem~3.1]{abramovich-olsson-vistoli}.

Moreover, at each point of $\cEhat$, the generator of the relative stabilizer subgroup $\mu_2$ 
corresponds to $(-1,-1)\in R\times \Gm$ in the local model (cf.~Remark~\ref{rem:cPhat.local} and Proposition~\ref{prop:cEhat}), and in particular it is central in its stabilizer. 

By Luna's \'etale slice theorem \cite{alper-hall-rydh-luna}, the map $\cPhat\to\cPhat^\dagger$ is \'etale locally 
\[[W/(H\times \mu_2)] \to [(W/\mu_2)/H],\]
where $W$ is a smooth variety, $H$ is finite, and $\mu_2$ acts on $W$ by a pseudo-reflection.
Therefore $\cPhat^\dagger$ is smooth, $\cEhat^\dagger$ is a Cartier divisor in $\cPhat^\dagger$, and $\cPhat\to\cPhat^\dagger$ is the square-root stack along $\cEhat^\dagger$.
\end{proof}

Combining the above results, we identify $\Phi$ as the weighted blowup.

\begin{theorem}\label{thm:roof.diagram}
	The morphism $\cPhat^\dagger \to \cPH$ is the blowup along $\cZ_2$ with exceptional divisor $\cEhat^\dagger$. Equivalently, $\Phi$ is the stack-theoretic weighted blowup along $\cZ_2$ with weight two.
\end{theorem}
\begin{proof}
	By Proposition~\ref{prop:mult} and Lemma~\ref{lem:Phi}, the scheme-theoretic inverse image of $\cZ_2$ in $\cPhat^{\dagger}$ is the Cartier divisor $\cEhat^{\dagger}$.
	Since $\cPhat^{\dagger}\to \cPH$ is representable, by the universal property of blowups, it factors through a morphism
	$\cPhat^{\dagger}\to \Bl_{\cZ_2}\cPH$,
	where $\Bl_{\cZ_2}\cPH$ denotes the (ordinary) blowup of $\cPH$ along $\cZ_2$. 
	This morphism is proper and an isomorphism over $\cPH\setminus\cZ_2$, hence surjective. 
	Moreover, it is quasi-finite, since fiberwise over $\cZ_2$ it induces a surjective morphism $\PP^1\to\PP^1$, hence finite.
	By Zariski's Main Theorem, it follows that $\cPhat^{\dagger}\to \Bl_{\cZ_2}\cPH$ is an isomorphism. 
\end{proof}
Let $\Phat$ be the coarse moduli space of $\cPhat$.
Since the blowup centers $\cZ_2\subset \cPH$ and $\cZ_\pTac\subset \cPK$ are both contracted to the unique boundary point of $\PCY$,
the two induced morphisms $\Phat \to \PH \to \PCY$ and $\Phat \to \PK \to \PCY$ agree. 

\begin{corollary}
	The morphism $\Phat\to \PH\times_{\PCY}\PK$ is the normalization.
\end{corollary}
\begin{proof}
	The morphism is proper and bijective on closed points by Proposition~\ref{prop:cEhat}. 
	Hence the assertion follows from Zariski's Main Theorem.
\end{proof}
This completes the proof of Theorem~\ref{thm:mainthmintro}.

\begin{remark}\label{rem:Z1.Kirwan}
	Restricting $\Phi$ and $\Psi$, we obtain the following diagram:
	\[\xymatrix{&\Psi^{-1}(\cEK)\ar[ld]\ar[rd]&\\ \overline\cZ_1&&\cEK\cong[Y_8^{ss}/\PGL_2]}\]
	where the left arrow is the square-root stack along $\cZ_2$, and the right arrow is the stack-theoretic weighted blowup of Kirwan type at the point $\pHT$, with weights $(2,3,4)$. 
	Since $\pHT=\cZ_\pTac\cap \cEK$ is transverse (Proposition~\ref{prop:transverse}), $N_{\pHT/\cEK}$ is the restriction of $N_{\cZ_\pTac/\cPK}$. 
	
	The exceptional divisor of the right arrow is a $\mu_2$-gerbe over $\cZ_2$, and
	in particular, the coarse moduli space of $\overline\cZ_1$ is obtained as a weighted Kirwan partial desingularization $\widehat Y_8/\!\!/\SL_2$ of the GIT quotient $Y_8/\!\!/\SL_2$, and its exceptional divisor is the coarse moduli space of $\cZ_2$.\end{remark}

\subsection{Reduction of stabilizers}\label{ss:can.red}

In this section we prove Corollary~\ref{cor:bpCY}, showing that $\Phat$ is naturally determined by $\cPCY$, as a weighted analogue of the canonical reduction of stabilizers in the sense of \cite{ER}. 

By \cite[16.1.2]{bpCY}, the moduli stack $\cPCY$ of boundary polarized Calabi--Yau pairs has a unique strictly polystable point $p_\infty$, represented by the pair 
	\beq\label{eq:bpCYpt}S=\{x_1x_2=0\}\subset \PP(1,1,1,2) \and C=\{x_0^4=x_3^2\}\cap S.\eeq
	Here $x_0,x_1,x_2,x_3$ are the homogeneous coordinates of $\PP(1,1,1,2)$. 
	
	Let $p_1=[0:1:0:0]$ and $p_2=[0:0:1:0]$, and let $S_1=\{x_2=0\}$ and $S_2=\{x_1=0\}$ so that $p_i\in S_i$. Then,
	$S=S_1\cup S_2$ with $S_i\cong \PP(1,1,2)$,
	and $C_i:=C\cap S_i$ is the union of two degree $2$ curves avoiding the cone point and meeting at $p_i$ (see \cite[Figure~1]{bpCY}).
\begin{lemma}\label{lem:etalenbd.bpCY}
	The stack $\cPCY$ is \'etale-locally at $p_\infty$ isomorphic to 
	\beq\label{eq:etale.nbd.CY}[\A^8/((\Gm^2\rtimes \mu_2)\times \Gamma)],\eeq
	where, on the coordinates $a,b,a_2,b_2,a_3,b_3,a_4,b_4$ of $\A^8$, $(t_1,t_2)\in \Gm^2$ acts by 
	\[(t_1,t_2)\cdot \left(a,b,a_2,b_2,a_3,b_3,a_4,b_4\right)=\left(t_1t_2a,t_1t_2b, t_1^{-2}a_2,t_2^{-2}b_2, t_1^{-3}a_3,t_2^{-3}b_3,t_1^{-4}a_4,t_2^{-4}b_4\right),\]
	$\mu_2$ fixes $a,b$ and swaps $a_i$ and $b_i$, and $\Gamma$ sends $b$ to $-b$ and fixes the others.
\end{lemma}
\begin{proof}
	The stabilizer of $p_\infty$ is $(\Gm^2\rtimes \mu_2)\times \Gamma$, where $\Gm^2$ acts by 
\[(t_1,t_2)\cdot[x_0:x_1:x_2:x_3]=[x_0:t_1x_1:t_2x_2:x_3],\] 
$\mu_2$ swaps $x_1$ and $x_2$, and $\Gamma$ sends $x_3$ to $-x_3$  (cf.~\cite[Proposition~8.11]{bpCY}).

An explicit \'etale slice is given as follows. Consider the family of pairs
\[\begin{aligned}
	\cS'&=\{x_1x_2=ax_0^2+bx_3\}\subset \A^8\times \PP(1,1,1,2)\\
	\cC'&=\{x_3^2=x_0^4+a_2x_0^2x_1^2+a_3x_0x_1^3+a_4x_1^4+b_2x_0^2x_2^2+b_3x_0x_2^3+b_4x_2^4\}\cap \cS'
\end{aligned}\]
over $\A^8$. This family is preserved by the action of $(\Gm^2\rtimes \mu_2)\times \Gamma$. The fiber over $0$ is \eqref{eq:bpCYpt}, and the induced map from $[\A^8/((\Gm^2\rtimes \mu_2)\times \Gamma)]$ to $\cPCY$ is \'etale at the origin.
\end{proof}
\begin{corollary}\label{cor:bpCY}
For any $k\in\ZZ_{>0}$,
the coarse moduli space of the Kirwan-type stack-theoretic weighted blowup of $\cPCY$ at the point $p_\infty$ with weights $(k,k,2,2,3,3,4,4)$ is isomorphic to $\Phat$. 
\end{corollary}

\begin{proof}
	Let $\cPhatCY\to \cPCY$ be the Kirwan-type stack-theoretic weighted blowup at $p_\infty$ with weights $(k,k,2,2,3,3,4,4)$, with respect to the local coordinates in Lemma~\ref{lem:etalenbd.bpCY}, and let $\cEhatCY\subset \cPhatCY$ denote its exceptional divisor. 
	
	We first observe that, away from $\cEhatCY$, the induced family of pairs is already Hacking stable. Indeed, the complement $\cPCY\setminus \cPH$ has the unique closed point $p_\infty$, and hence this complement is precisely the strictly semistable locus of $\cPCY$. The semistable locus of the weighted blowup, which is $\cPhatCY$, is then the complement of the proper transform of $\cPCY\setminus \cPH$. This implies that $\cPhatCY \to \cPCY$ restricts to an isomorphism $\cPhatCY\setminus \cEhatCY\to \cPH$.
	
	Therefore, to construct a morphism $\cPhatCY\to \cPH$, it suffices to work on the \'etale-local model in Lemma~\ref{lem:etalenbd.bpCY}.	
	Arguing as in the proof of Theorem~\ref{thm:mainthmintro}, we construct such a map to $\cPH$ by constructing a family of Hacking stable pairs over $\cPhatCY$, and then identify the map with the stack-theoretic weighted blowup along $\cZ_2$ using the universal property of blowups.
	The second assertion then follows immediately from the first by Lemma~\ref{lem:Phi}.
	
	We blow up the origin in \eqref{eq:etale.nbd.CY} stack-theoretically with the stated weights. 
	With respect to the inverse of the line bundle associated with the exceptional divisor, the unstable locus is the proper transform of
	\[\{a=b=0\}\cup \{a_2=a_3=a_4=0\}\cup \{b_2=b_3=b_4=0\}.\]
	After removing this unstable locus, $\cEhatCY$ is the quotient stack of 
	\[\left(\A^2_{a,b}\setminus \{0\}\right) \times \left(\A^3_{a_2,a_3,a_4}\setminus \{0\}\right) \times \left(\A^3_{b_2,b_3,b_4}\setminus\{0\}\right)\]
	by $\Gm\times (\Gm^2\rtimes \mu_2)\times \Gamma$, where the identity component $\Gm^3$ acts by
	\[\begin{split}
		&(t,t_1,t_2)\cdot \left(a,b,a_2,b_2,a_3,b_3,a_4,b_4\right)\\
		&=\left(t^kt_1t_2a,t^kt_1t_2b, (tt_1^{-1})^2a_2,(tt_2^{-1})^2b_2, (tt_1^{-1})^3a_3,(tt_2^{-1})^3b_3,(tt_1^{-1})^4a_4,(tt_2^{-1})^4b_4\right).
	\end{split}
	\]
	Hence, $(t,t_1,t_2)\in \Gm^3$ acts trivially if and only if $t=t_1=t_2$ and $t^{k+2}=1$. 
	
	Since $\mu_2$ and $\Gamma$ act as described before, this induces a natural morphism 
	\[\cEhatCY\lra\cEhat^\dagger,\] 
	which is a $\mu_{k+2}$-gerbe. 
	In particular, $\cPhatCY$ is a Deligne--Mumford stack.
	
	We now construct the family over $\cPhatCY$.
	We pull back the universal family $(\cS^{\mathrm{CY}},\cC^{\mathrm{CY}})$ over $\cPCY$ to $\cPhatCY$, and modify it along the exceptional divisor $\cEhatCY$.
	More precisely, over $\cEhatCY$,
	consider the two sections corresponding to the points $p_1$ and $p_2$. 
	Let 
	\beq\label{eq:wtblowup.CY}\cShatCY\lra \cS^{\mathrm{CY}}\times_{\cPCY}\cPhatCY\eeq
	denote the weighted blowup along these two sections over $\cEhat$ with weights $(1,1,2)$, where the weight-two direction corresponds to the direction toward the cone point $[0:0:0:1]$ on each component $S_i\cong \PP(1,1,2)$. 
	Let $\cChatCY$ denote the proper transform of the pullback of $\cC^{\mathrm{CY}}$.
	Thus we obtain a pair $(\cShatCY,\cChatCY)$ over $\cPhatCY$.
	Let $\hat \pi^{\mathrm{CY}}:\cShatCY\to \cPhatCY$ denote the projection.
	
	Fiberwise over $\cEhatCY$, the weighted blowup produces a chain
	\beq\label{eq:fiber.CY}S_1'\cup \widetilde S_1\cup \widetilde S_2\cup S_2',\eeq
	where $S_i'\cong \PP(1,1,2)$ is the exceptional component lying over $p_i$, and $\widetilde S_i$ is the weighted blowup of $S_i\cong \PP(1,1,2)$ at $p_i$ with weights $(1,2)$.
	Let $\rho_i:\widetilde S_i\to S_i$ be the weighted blowup with exceptional divisor $\ell_i\subset \widetilde S_i$. 
	The line bundle
	\beq\label{eq:linebundle.CY}\rho_i^*\cO_{S_i}(2)\otimes \cO_{\widetilde S_i}(-2\ell_i)\eeq
	is base-point free and defines a contraction
	$\widetilde S_i\to \ell\cong \PP^1$
	onto the double curve. Thus, after contracting the two middle components $\widetilde S_1$ and $\widetilde S_2$, the resulting surface is the union $S_1'\cup_\ell S_2'$.
	
	By the same local calculation as in Proposition~\ref{prop:fiber.Chat}, the proper transform of the family of curves on each component $S_i'$ is a degree $4$ curve disjoint from the cone point and meeting $\ell$ transversely in two points.
	Hence the fibers over the exceptional divisor are precisely the Hacking stable pairs parametrized by $\cZ_2$.

	As in the proof of Proposition~\ref{prop:contract}, this contraction globalizes via 
	\[(\hat \pi^\mathrm{CY})^*\hat\pi^{\mathrm{CY}}_*\cL^{\mathrm{CY}}\lra \cL^{\mathrm{CY}},\]
	where
	$\cL^{\mathrm{CY}}:=\omega_{\cS^{\mathrm{CY}}/\cPCY}^{-2}(\cC^{\mathrm{CY}})|_{\cShatCY} \otimes \cO_{\cShatCY}(-2\cFhatCY)$,
	and $\cFhatCY$ denotes the exceptional divisor of \eqref{eq:wtblowup.CY}. 
	Indeed, $\cL^{\mathrm{CY}}$ restricts to \eqref{eq:linebundle.CY}  
	on $\widetilde S_i$, and to $\cO_{S_i'}(2)$ on $S_i'$ in \eqref{eq:fiber.CY}.

	Together with the previous observation that $\cPhatCY\setminus\cEhatCY\xrightarrow{\cong}\cPH$, this modified family defines a morphism
	$\cPhatCY\to \cPH$.
	By the same argument as in the proof of Theorem~\ref{thm:roof.diagram}, its relative coarse moduli space is $\cPhat^\dagger\to \cPH$, and the induced morphism $\cPhatCY\to \cPhat^\dagger$ is a $(k+2)$-th root stack of $\cPhat^\dagger$ along $\cEhat^\dagger$. In particular, the coarse moduli space of $\cPhatCY$ is isomorphic to $\Phat$.
\end{proof}
\begin{remark}\label{rem:bpCY}
	Heuristically, the stack $\cPhat$ may be viewed as a ``$k=0$'' version of the stack $\cPhatCY$ constructed in the proof of Corollary~\ref{cor:bpCY}, in which the smoothing directions of the surface are unweighted.
	
	The same argument shows that the coarse moduli space is unchanged if the weights are replaced by
	$(k,k,2i,2j,3i,3j,4i,4j)$ for any $i,j,k\in \Z_{>0}$.
\end{remark}

\section{Comparison of Hodge bundles}\label{s:Hodge}
In this section, we determine the relationship between the Hodge bundles on $\cPH$, $\cPK$, and $\cPGIT$. 
These bundles are defined as the pushforwards 
\[\EE^\rmH:= \pi_{\cC*}\omega_{\cC/\cPH}, \qquad  \EE^\rmK:=\pi_{\cC^\rmK*}\omega_{\cC^\rmK/\cPK}, \qquad \EE^{\mathrm{GIT}}:=\pi_{\cC^{\mathrm{GIT}}*}\omega_{\cC^{\mathrm{GIT}}/\cPGIT},\]
where
\[\pi_\cC:\cC\lra \cPH, \qquad \pi_{\cC^\rmK}:\cC^{\rmK}\lra \cPK, \qquad \pi_{\cC^{\mathrm{GIT}}}:\cC^{\mathrm{GIT}}\lra \cPGIT\]
denote the universal families of curves.
All fibers $C$ are Gorenstein with $\dim H^0(C,\omega_C)=3$.
Hence they are vector bundles of rank three. 

We fix $V=\C^3$ so that $\PP^2=\PP V$ and $G=\SL(V)$.

\begin{lemma}\label{lem:Hodge.GIT}
	$\EE^{\mathrm{GIT}}$  is the descent of the $\bar G$-linearized bundle $V^*\otimes \cO_{X^{ss}}(1)$.
\end{lemma}
\begin{proof}
	The divisor $\cC^{\mathrm{GIT}}\subset \cS^{\mathrm{GIT}}=[(\PP^2\times X^{ss})/\bar G]$ is the descent of a divisor $D^{\mathrm{GIT}} \subset \PP^2\times X^{ss}$ of bidegree $(4,1)$. 
	By adjunction, $\omega_{\pi_{\cC^{\mathrm{GIT}}}}$ is the descent of $\cO_{\PP^2\times X^{ss}}(1,1)\big|_{D^{\mathrm{GIT}}}$. 
	Taking the pushforward, $\EE^{\mathrm{GIT}}$ is the descent of $V^*\otimes \cO_{X^{ss}}(1)$.
	Note that the natural $G$-linearization on $V^*\otimes\cO_{X^{ss}}(1)$ is trivial on the center $\mu_3\subset G$, and hence it descends to a $\bar G$-linearization.
\end{proof}

We define the Hodge line bundle $\Lambda^{\mathrm{GIT}}$ on $\cPGIT$ to be the descent of
\beq\label{eq:Hodge.GIT}\cO_{\cS^{\mathrm{GIT}}}(4K_{\cS^{\mathrm{GIT}}/\cPGIT}+3\cC^{\mathrm{GIT}}).\eeq
\begin{proposition}\label{prop:Lambda.GIT}
	We have $\det \EE^{\mathrm{GIT}}\cong\Lambda^{\mathrm{GIT}}$.
\end{proposition}
\begin{proof}
	By Lemma~\ref{lem:Hodge.GIT}, $\det \EE^{\mathrm{GIT}}$ is the descent of $\cO_{X^{ss}}(3)$. 
	The same is true for $\Lambda^{\mathrm{GIT}}$, since the line bundle \eqref{eq:Hodge.GIT} is the descent of $\cO_{\PP^2\times X^{ss}}(0,3)$.
\end{proof}
We compare the Hodge line bundles on $\cPGIT$ and $\cPK$.
\begin{proposition}\label{prop:Lambda.K}
	On $\cPK$, we have $\Lambda\cong \Lambda^{\mathrm{GIT}}|_{\cPK}\otimes \cO_{\cPK}(-2\cEK)$.
\end{proposition}
\begin{proof}
	Over $\cPK\setminus\cEK$, $\Lambda$ and $\Lambda^{\mathrm{GIT}}$ are naturally identified. Thus $\Lambda\cong\Lambda^{\mathrm{GIT}}|_{\cPK}\otimes \cO_{\cPK}(-m\cEK)$ 
	for some $m$. 
	Since $\Lambda$ is trivial on $\cZ_\pTac$ (see \cite{ADL}), we conclude $m=2$ by Lemma~\ref{lem:linear.relation}. 
\end{proof}

We now relate the Hodge bundles on $\cPGIT$, $\cPK$, and $\cPH$.

\begin{theorem}\label{thm:Hodge}
\begin{enumerate}
	\item On $\cPK$, there is an isomorphism 
	\[\EE^\rmK\cong \EE^{\mathrm{GIT}}|_{\cPK}\otimes \cO_{\cPK}(-\cEK).\]
	\item On $\cPhat$, there is a short exact sequence
	\[0\lra\Psi^*\EE^\rmK(-\cEhat)\lra \Phi^*\EE^\rmH \lra \cO_{\cEhat}\lra 0.\]
\end{enumerate}	
\end{theorem}
\begin{proof}
	(1) 
	Let $(\cS^{\mathrm{GIT}}|_{\cPK},\cC^{\mathrm{GIT}}|_{\cPK})$ denote the pullback of $(\cS^{\mathrm{GIT}},\cC^{\mathrm{GIT}})$. 
	By \cite[Section~5]{ADL}, there is a diagram
	\[\cS^\rmK\longleftarrow (\cS^{\rmK})' \xrightarrow{~b~} \cS^{\mathrm{GIT}}|_{\cPK},\] 
	in which  $b$ is the blowup along the reduced support of the double conics over $\cEK$ and the map to $\cS^\rmK$ 
	is the contraction of the proper transform $\cDK$ of the $\PP^2$-bundle over $\cEK$ under $b$.
	
	Let $(\cC^\rmK)' \subset (\cS^\rmK)'$ denote the proper transform of $\cC^{\mathrm{GIT}}|_{\cPK}$. 
	Then $(\cC^\rmK)'$ does not meet $\cDK$, and it is mapped isomorphically onto $\cC^\rmK$ over $\cPK$.
	In particular, we have $\omega_{\pi_{\cC^\rmK}} \cong \omega_{(\cC^\rmK)'/\cPK}$.

	Let $\cF^\rmK\subset (\cS^\rmK)'$ denote the exceptional divisor of the blowup.
	On $(\cS^\rmK)'$, 
	\[K_{(\cS^\rmK)'/\cPK}\sim b^*(K_{\cS^{\mathrm{GIT}}|_{\cPK}/\cPK})+\cF^\rmK \and 
	(\cC^\rmK)'\sim b^*(\cC^{\mathrm{GIT}}|_{\cPK})-2\cF^\rmK.\]
	Hence,
	\[\begin{split}
		\omega_{\pi_{\cC^\rmK}}\cong \omega_{(\cC^\rmK)'/\cPK}
		&\cong \omega_{(\cS^\rmK)'/\cPK}((\cC^\rmK)')|_{(\cC^\rmK)'}\\
		&\cong \omega_{\cS^{\mathrm{GIT}}|_{\cPK}/\cPK}(\cC^{\mathrm{GIT}}|_{\cPK})|_{(\cC^\rmK)'}\otimes (-\cF^\rmK)\\
		&\cong \omega_{\cS^{\mathrm{GIT}}/\cPGIT}(\cC^{\mathrm{GIT}})|_{(\cC^\rmK)'} \otimes (-\cF^\rmK)\\
		&\cong \omega_{\pi_{\cC^{\mathrm{GIT}}}}|_{(\cC^\rmK)'}\otimes (-\cF^\rmK)\\
		&\cong \omega_{\pi_{\cC^{\mathrm{GIT}}}}|_{\cC^\rmK}(-\cEK),
	\end{split}
	\]
	where the last isomorphism follows from  $\cEK\sim \cDK+\cF^\rmK$, $(\cC^\rmK)'\cap \cDK=\emptyset$, and $(\cC^\rmK)'\cong \cC^\rmK$.
	
	By cohomology and base change (\cite[Theorem~25.1.6]{vakil}), $\EE^{\mathrm{GIT}}|_{\cPK}$ is isomorphic to the pushforward of $\omega_{\pi_{\cC^{\mathrm{GIT}}}}|_{(\cC^{\mathrm{GIT}}|_{\cPK})}$.
	By adjunction, this coincides with the pushforward of $\omega_{\pi_{\cC^{\mathrm{GIT}}}}|_{\cC^\rmK}$.
	Applying $\pi_{\cC^\rmK *}$ and using the projection formula, we obtain the first assertion.
	
	(2) 
	We compare $\EE^\rmH$ and $\EE^\rmK$ after pulling back to $\cPhat$.
	In Section~\ref{s:resolution}, we construct morphisms
	\[(\cShatH,\cChatH)\longleftarrow(\cShat,\cChat)\lra(\cShatK,\cChatK),\]
	where $(\cShatH,\cChatH)$ and $(\cShatK,\cChatK)$ are obtained from the Cartesian diagrams:
	\[\begin{tikzcd}
		\cS\arrow[d,"\pi"'] &\cShatH\arrow[l]\arrow[d,"\hat \pi^\rmH"] \\ \cPH &\cPhat\arrow[l,"\Phi"']
	\end{tikzcd}
	\and
	\begin{tikzcd}
		\cShatK \arrow[r]\arrow[d,"\hat \pi^\rmK"']&\cS^\rmK\arrow[d,"\pi^\rmK"] \\ \cPhat \arrow[r,"\Psi"] &\cPK
	\end{tikzcd}
	\]
	By cohomology and base change, together with adjunction, we have 
	\[\Phi^*\EE^\rmH \cong (\hat\pi^\rmH|_{\cChatH})_* \Big(\omega_{\hat\pi^\rmH}(\cChatH)\big|_{\cChatH}\Big) \and
		\Psi^*\EE^\rmK \cong (\hat\pi^\rmK|_{\cChatK})_* \Big(\omega_{\hat\pi^\rmK}(\cChatK)\big|_{\cChatK}\Big).
	\]
	
	By the construction in Section~\ref{s:resolution}, we have
	$K_{\hat\pi^\rmH}|_{\cShat}\sim K_{\hat\pi^\rmK}|_{\cShat}+ 3\cFhat$ and $\cChatH|_{\cShat}\sim \cChatK|_{\cShat}-4\cFhat$ on $\cShat$.
	Since $\hat\pi^*\cEhat\sim \cFhat+\cDhat$ on $\cShat$, 
	it follows that
	\beq\label{eq:Hodge.HK}\omega_{\hat\pi^\rmH}(\cChatH)|_{\cShat}\otimes \cO_{\cShat}(-\cDhat)\cong  \omega_{\hat\pi^\rmK}(\cChatK)|_{\cShat}\otimes \hat\pi^*\cO_{\cPhat}(-\cEhat).\eeq
	
	Since $\cDhat$ is Cartier in a neighborhood of $\cChat$, we have a short exact sequence
	\[0 \lra \cO_{\cShat}(-\cDhat)\lra \cO_{\cShat}\lra \cO_{\cDhat}\lra 0\]
	in that neighborhood.
	Tensoring with $\omega_{\hat\pi^\rmH}(\cChatH)$ and restricting to $\cChat$, we get
	\[
	0 \lra\omega_{\hat\pi^\rmH}(\cChatH)|_{\cChat}\otimes \cO_{\cShat}(-\cDhat)|_{\cChat}
	\lra \omega_{\hat\pi^\rmH}(\cChatH)|_{\cChat}
	\lra \omega_{\hat\pi^\rmH}(\cChatH)|_{\cChat\cap \cDhat}\lra 0.
	\]
	Using \eqref{eq:Hodge.HK}, we write the left term as $\omega_{\hat\pi^\rmK}(\cChatK)|_{\cChat}\otimes \hat\pi_{\cChat}^*\cO_{\cPhat}(-\cEhat)$,
	with $\pi_{\cChat}:=\hat\pi|_{\cChat}$.
	
	Applying $\pi_{\cChat *}$ and using the projection formula, we obtain
	\beq\label{eq:exact.Hodge}0\lra \Psi^*\EE^\rmK\otimes \cO_{\cPhat}(-\cEhat)\lra \Phi^*\EE^\rmH\lra \pi_{\cChat*}(\omega_{\pi_{\cChat}}|_{\cChat\cap \cDhat}).\eeq
	Since each of the fibers of $\cChat\cap \cDhat$ is the union $\Chat_0=\PP^1\sqcup\PP^1$ (Proposition~\ref{prop:fiber.Chat}) and $\omega_{\pi_{\cChat}}|_{\Chat_0}\cong\cO_{\Chat_0}$, the right-most term is the rank $2$ bundle on $\cEhat$ whose pullback to the $\Gamma$-cover of $\cEhat$ is trivial, with $\Gamma$ exchanging the two summands.
	
	For $(S,C)\in \cZ_2$, by adjunction, the restriction map
	\[H^0(S,\omega_S(C))\xrightarrow{~\cong~} H^0(C,\omega_C)\lra H^0(\ell,\omega_C|_{C\cap \ell})\cong \C^2\]
	factors through $ H^0(\ell,\omega_S(C)|_\ell)\cong H^0(\ell,\cO_\ell)\cong\C$, with image identified with the diagonal in $\C^2$.
	Thus the image of the right-most map in \eqref{eq:exact.Hodge} is the diagonal line subbundle of 
	the above rank $2$ bundle on $\cEhat$, which is hence isomorphic to $\cO_{\cEhat}$.
	This completes the proof.
\end{proof}

As a consequence, we obtain the following.
\begin{corollary}\label{cor:Hodge}
	We have $\det \EE^\rmK\cong \Lambda(-\cEK)$ and $\det \EE^\rmH\cong\Lambda(-\partial \cPH)$.
\end{corollary}
\begin{proof}
	The first isomorphism follows from Propositions~\ref{prop:Lambda.GIT} and \ref{prop:Lambda.K} and Theorem~\ref{thm:Hodge}(1). 
	For the second isomorphism, Theorem~\ref{thm:Hodge}(2) gives
	\beq\label{eq:Hodge.HK2}\Phi^*\det\EE^\rmH \cong \Psi^*\det\EE^\rmK\otimes \cO_{\cPhat}(-2\cEhat).\eeq
	Restricting \eqref{eq:Hodge.HK2} to $\cPhat\setminus \cEhat$, where $\Phi$ and $\Psi$ are isomorphisms, we obtain
	\[\det \EE^\rmH \cong \det \EE^\rmK \cong \Lambda(-\cEK)\]
	on $\cPH\setminus \cZ_2$, identified with the corresponding open substack of $\cPK$.
	
	Since $\cO_{\cPH}(\partial\cPH)$ corresponds to $\cO_{\cPK}(\cEK)$ under $\Pic(\cPH)\cong\Pic(\cPK)$, we deduce
	$\det \EE^\rmH \cong \Lambda(-\partial\cPH)$
	on $\cPH\setminus \cZ_2$.
	Since $\Pic(\cPH)\cong \Pic(\cPH\setminus \cZ_2)$, this isomorphism extends to $\cPH$.
\end{proof}

\begin{remark}\label{rem:Hodge.HK}
	Combining Corollary~\ref{cor:Hodge} with \eqref{eq:Hodge.HK2}, we obtain on $\cPhat$
	\[\Phi^*\partial\cPH\sim \Psi^*\cEK+2\cEhat.\]
	A more direct proof of this relation will be given in Lemma~\ref{lem:Z1.relation}.
\end{remark}

\section{Poincar\'e series and cycle class map}\label{s:Betti}

In this section, we compute the Poincar\'e series of the stacks $\cPGIT$, $\cPK$, $\cPhat$, and $\cPH$, and prove that their cycle class maps are isomorphisms. 
The argument uses Kirwan's equivariantly perfect stratification for GIT quotients \cite{Kir84} and the blowup formula in cohomology and Chow theory. 
We begin by recalling general facts and then apply them to the GIT stratification of plane quartic curves and to the birational modifications constructed earlier.

\subsection{Notation and basic facts}

Let $M$ be a Deligne--Mumford stack. We denote by $A_*(M)$ the Chow groups with rational coefficients. 
If an algebraic group $K$ acts on $M$, we denote by $A_*^K(M)$ the equivariant Chow groups.
For a quotient stack $[M/K]$, we define 
\[A_*([M/K]) := A_*^K(M).\]

If $K'\to K$ is a homomorphism of algebraic groups with finite kernel, then the natural morphism $[M/K]\to [M/K']$ induces canonical isomorphisms
\[A_*^{K'}(M)\cong A_*^{K}(M).\]

For a closed substack $N\subset M$, there is a right exact localization sequence 
\beq\label{eq:loc}
A_*(N)\lra A_*(M)\lra A_*(M\setminus N)\lra 0,\eeq
and similarly for equivariant Chow groups.

If $M$ is smooth and pure-dimensional, we use cohomological grading by 
\[A^i(M):=A_{\dim M-i}(M).\]
In this case, $A^*(M)$ is a ring under the intersection product, called the Chow ring of $M$.
Similarly, we set $A^i_K(M):=A_{\dim M-i}^K(M)$ and denote by $A^*_K(M)$ the $K$-equivariant Chow ring (equivalently, the Chow ring of $[M/K]$).

We refer to \cite{Fulton-intersection-theory, vistoli,EG-Inv} for definitions and basic properties of Chow groups and their equivariant versions.

\medskip

We write $H^*(M)$ (resp.\ $H^*_K(M)$) for the (resp.\ $K$-equivariant) cohomology with rational coefficients, and set
$H^*([M/K]) := H^*_K(M)$.
We define the Poincar\'e polynomial by
\[P_t(M):=\sum_{i\geq 0}\dim_\Q H^i(M)t^i \in \Z[t],\]
and the equivariant Poincar\'e series by
\[P_t^K(M):=\sum_{i\geq 0}\dim_\Q H^i_K(M)t^i \in \Z\llbracket t\rrbracket.\]
For a quotient stack $[M/K]$, we set $P_t([M/K]) := P_t^K(M)$.

\subsection{Equivariantly perfect stratification}\label{ss:Kirwan.strat}
We recall Kirwan's equivariantly perfect stratification for GIT quotients \cite{Kir84}.
Let $M$ be a smooth projective variety equipped with an action of a reductive group $K$. We fix a $K$-equivariant closed immersion of $M\hookrightarrow \PP^m=\PP(\C^{m+1})$ for a $K$-representation $\C^{m+1}$.
We fix a maximal torus $T_K\subset K$. 

There is a natural $K$-equivariant stratification
\[M=M^{ss}\sqcup \bigsqcup_{\beta\in \cB}S_\beta,\]
where $M^{ss}$ is the (open) GIT semistable locus and $S_\beta$ are $K$-invariant locally closed strata indexed by a finite subset $\cB$ of $\Hom(\Gm,T_K)_\Q$. 
This is a stratification in the sense that there is a partial order on $\cB$ such that $\overline S_\beta\subset \bigsqcup_{\beta'\leq \beta} S_{\beta'}$ for each $\beta\in \cB$.

This stratification is \emph{equivariantly perfect} in the sense that we have
\[H^*_K(M)\cong H^*_K(M^{ss})\oplus \bigoplus_{\beta\in \cB}H^{*-2\codim S_\beta}_K(S_\beta),\]
where $\codim S_\beta$ denotes the codimension of $S_\beta$ in 
$M$. In particular, we have 
\beq\label{eq:Poincare.Sbeta}P_t^K(M^{ss})=P_t^K(M)-\sum_{\beta\in \cB}t^{2\codim  S_\beta}P_t^K(S_\beta).\eeq
This follows from the splitting of the Thom--Gysin sequence into 
\beq\label{ses.coh}0\lra H^{*-2\codim S_{\beta_j}}_K(S_{\beta_j})\lra H^*_K(M_j)\lra H^*_K(M_{j-1})\lra 0,\eeq
where after refining the partial order on $\cB$ into any total order, we set 
\[M_j:=M^{ss}\sqcup \bigsqcup_{1\leq i\leq j}S_{\beta_i}\]
for $0\leq j\leq r$ and $\cB=\{\beta_1,\dots, \beta_r\}$ with $\beta_1>\cdots>\beta_r$.
Indeed, the normal bundle of $S_{\beta_j}$ in $M_j$ has a $K$-equivariant top Chern class that is not a zero divisor, so the composition
$H^{*-2\codim S_{\beta_j}}_K(S_{\beta_j}) \to H^*_K(M_j) \to H^*_K(S_{\beta_j})$
is injective; see \cite[Lemma~2.18]{Kir84}.

By the same argument, the localization exact sequence
\beq\label{ses.Chow}0\lra A^{*-\codim S_{\beta_j}}_K(S_{\beta_j})\lra A^*_K(M_j)\lra A^*_K(M_{j-1})\lra 0\eeq
is exact from the left. The short exact sequences \eqref{ses.coh} and \eqref{ses.Chow} are compatible by the $K$-equivariant cycle class maps $A^{*}_K(-)\to H^{2*}_K(-)$.

Moreover, we have
\beq\label{eq:Poincare.V}P_t^K(M)=P_t(M)P_t(BK),\eeq
since the spectral sequence associated to $[M/K]\to BK$ degenerates \cite[Proposition~5.8]{Kir84}.

As a consequence of \eqref{ses.Chow}, the cycle class map for $M^{ss}$
\beq\label{eq:cl.Vss}A^*_K(M^{ss})\lra H^*_K(M^{ss})\eeq
is an isomorphism if the cycle class maps for $M$ and $S_\beta$
\[A^*_K(M)\lra H^*_K(M)\and A^*_K(S_\beta)\lra H^*_K(S_\beta)\]
are isomorphisms for all $\beta\in \cB$.

\subsection{Unstable strata $S_\beta$}
We now describe the unstable strata $S_\beta$. 

Fix a positive Weyl chamber in $\Hom(\Gm,T_K)_\Q$, and let $\tau_+$ denote its closure.
Fix a $W_K$-invariant non-degenerate inner product $\cdot$ on $\Hom(\Gm,T_K)_\Q$ where $W_K=N_K(T_K)/T_K$, 
and let $\lVert-\rVert$ denote the associated norm.
 
Let $\alpha_0,\cdots, \alpha_m$ denote the $T_K$-weights of the representation $\C^{m+1}$. Using the inner product, we identify $\alpha_i$ with an element of $\Hom(\Gm,T_K)_\Q$.
Let $x_0,\cdots, x_m$ be the homogeneous coordinates of $\PP^m$ corresponding to the weight decomposition.

For each subset of the $\alpha_j$, consider the point $\beta$ of minimal norm in its convex hull. 
Among these, we take only the nonzero points in $\tau_+$.
For each such $\beta\in \tau_+\setminus\{0\}$, let
\[Z_\beta:=M\cap \left\{[x_0:\cdots:x_m]\in \PP^m:x_j=0 \text{ if }\alpha_j\cdot \beta\neq \lVert\beta\rVert^2\right\}.\]
We define the indexing set $\cB$ to consist of pairs $(\beta,i)$, where $\beta\in \tau_+\setminus\{0\}$ is as above and $Z_{\beta,i}$ is a connected component of $Z_\beta$.

To simplify the notation, we will denote such a pair simply by $\beta$, and write $Z_\beta$ for the corresponding connected component.
When $M=\PP^m$, the linear section $Z_\beta\subsetneq \PP^m$ is a linear subspace and hence irreducible. 
In this case, we identify $\cB$ with a subset of $\tau_+\setminus\{0\}$.

\medskip

Let $\Stab(\beta)\subset K$ denote the stabilizer of $\beta$ under the adjoint action of $K$ on $\Hom(\Gm,K)$. 
The restriction of $\cO_M(1)$ to $Z_\beta$ admits a natural $\Stab(\beta)$-linearization. We twist this linearization by the character corresponding to $\beta$, and denote the resulting semistable locus by $Z_\beta^{ss}$.
Equivalently, $Z_\beta^{ss}\subset Z_\beta$ is the (possibly empty) open subset consisting of the points for which $\beta$ is optimal in the sense of Kempf with respect to the original linearization $\cO_M(1)$; see \cite[8.11, Definition~12.20 and Remark~12.21]{Kir84}.

Consider the linear section 
\[Y_\beta':=M\cap \left\{[x_0:\cdots:x_m]\in \PP^m:x_j=0 \text{ if }\alpha_j\cdot \beta< \lVert\beta\rVert^2\right\}\]
which contains $Z_\beta$. 
Define $Y_\beta\subset Y_\beta'$ to be the open subscheme consisting of points for which there exists $0\leq j\leq m$ with $x_j\neq 0$ and $\alpha_j\cdot \beta=\lVert\beta\rVert^2$.
There is a natural projection
\[p_\beta:Y_\beta \lra Z_\beta,\]
defined by setting $x_j=0$ whenever $\alpha_j\cdot\beta> \lVert \beta\rVert^2$.
This map is $P_\beta$-equivariant, where $P_\beta$ acts on $Z_\beta$ via the natural projection $P_\beta \to \Stab(\beta)$, and it is a locally trivial fibration such that every fiber is an affine space (\cite[\S13]{Kir84}). 
Set
\[Y_\beta^{ss}:=p_\beta^{-1}(Z_\beta^{ss}).\]
Then $S_\beta$ is defined to be the $K$-translate of $Y_\beta^{ss}$:
\[S_\beta:=K\cdot Y_\beta^{ss}
~\subset M.\]

Let $P_\beta\subset K$ be the subgroup consisting of elements that preserve $Y_\beta\subset M$.
Then, $P_\beta$ is a parabolic subgroup (\cite[Lemma~6.9]{Kir84}), and the natural $K$-equivariant morphism $K\times^{P_\beta}Y_\beta^{ss}\xrightarrow{\cong} S_\beta$
is an isomorphism by \cite[Theorem~6.18]{Kir84}. Consequently, 
\[A^*_K(S_{\beta})\cong A^*_{P_\beta}(Y_{\beta}^{ss})\cong A^*_{\Stab(\beta)}(Z_{\beta}^{ss}) \and H^*_K(S_{\beta})\cong H^*_{P_\beta}(Y_{\beta}^{ss})\cong H^*_{\Stab(\beta)}(Z_{\beta}^{ss}).\]
Here we use that $P_\beta$ deformation retracts onto $\Stab(\beta)$.

Consequently, the formula \eqref{eq:Poincare.Sbeta} now reads
\beq\label{eq:poincare.git}P_t^K(M^{ss})=P_t^K(M)-\sum_{\beta\in \cB}t^{2\codim  S_\beta}P_t^{\Stab(\beta)}(Z_\beta^{ss})\eeq
where $\codim S_\beta=\dim M-(\dim K+\dim Y_\beta^{ss}-\dim P_\beta)$. 

Moreover, the $K$-equivariant cycle class map \eqref{eq:cl.Vss} is an isomorphism if 
\beq\label{cl:V}A^*_K(M)\lra H^*_K(M)\eeq
is an isomorphism, and the $\Stab(\beta)$-equivariant cycle class map
\beq\label{cl:Zbetass}A^*_{\Stab(\beta)}(Z_{\beta}^{ss})\lra H^*_{\Stab(\beta)}(Z_{\beta}^{ss})\eeq
is an isomorphism for all $\beta\in \cB$.

Analogous statements hold when $M$ is a smooth Deligne--Mumford stack with a projective coarse moduli space. 
For a vast generalization, see \cite[Theorem~3.1]{halpern2015theta}.

\smallskip

We now record a useful special case for projective space.
\begin{theorem}\label{thm:cl.projsp}
	Let $M=\PP^m$ with a linear action of a reductive group $K$. Then the $K$-equivariant cycle class map \eqref{eq:cl.Vss} for $M^{ss}$ is an isomorphism.  In particular, $H^{\mathrm{odd}}_K(M^{ss})=0$.
\end{theorem}
\begin{proof}
	By the short exact sequences \eqref{ses.coh} and \eqref{ses.Chow}, together with their compatibility with the cycle class maps (via pushforward for closed immersions and restriction to open complements), it suffices to show that the cycle class maps \eqref{cl:V} and \eqref{cl:Zbetass} are isomorphisms.

	For \eqref{cl:V}, the projective bundle formula in equivariant Chow and cohomology theory implies that $A^*_K(\PP^m)$ (resp.\ $H^*_K(\PP^m)$) is a free module over $A^*(BK)$ (resp.\ $H^*(BK)$) with basis $1,h,\dots,h^m$, where $h=c_1^K(\cO_{\PP^m}(1))$ is the equivariant hyperplane class.
	Since 
	\[A^*(BK)\cong A^*(BT_K)^{W_K}\cong H^*(BT_K)^{W_K}\cong H^*(BK)\]
	for reductive $K$ \cite{edidin-graham-char-class}, and the cycle class map identifies the equivariant hyperplane class in Chow and cohomology, the map \eqref{cl:V} is an isomorphism.
	
	For \eqref{cl:Zbetass}, note that $Z_\beta\subsetneq \PP^m$ is a $\Stab(\beta)$-invariant projective subspace. Since $\Stab(\beta)$ is reductive and the induced action on $Z_\beta$ is linear, \eqref{cl:Zbetass} follows by induction on $m$.
\end{proof}

As an immediate corollary of Theorem~\ref{thm:cl.projsp}, 
we obtain the following.
\begin{corollary}\label{cor:cl.hypersurf}
Let $M=\PP H^0(\PP^n,\cO_{\PP^n}(d))$, and let $\cPGIT_{n,d}:=[M^{ss}/\PGL_{n+1}]$ be the GIT moduli stack of degree $d$ hypersurfaces in $\PP^n$.
The cycle class map
$A^*(\cPGIT_{n,d})\to H^*(\cPGIT_{n,d})$
is an isomorphism. In particular, $H^{\mathrm{odd}}(\cPGIT_{n,d})=0$.
\end{corollary}
\begin{proof}
It suffices to prove that the $\SL_{n+1}$-equivariant cycle class map for $M^{ss}$ is an isomorphism.
This follows from Theorem~\ref{thm:cl.projsp}.
\end{proof}

\subsection{GIT stratification of the space of plane curves}\label{ss:GIT.planecurve}
We describe $\cB, Z_\beta,$ and $ \Stab(\beta)$ in the case of plane curves, and compute $P_t(\cPGIT)$.
The main reference for this subsection is \cite{KL1}, whose arguments apply to arbitrary degree $d$; we recall them here in our setting.

\smallskip

Let $K=G=\SL_3$ and 
\[X_d=\PP H^0(\PP^2,\cO_{\PP^2}(d)),\] 
where $H^0(\PP^2,\cO_{\PP^2}(d))$ is naturally a $G$-representation of dimension $\binom{d+2}{2}$.

Let $T_K=T$. 
Then $\Hom(\Gm,T)_\Q\cong \Q^2$, with basis given by $t\mapsto \diag(t,1,t^{-1})$ and $t\mapsto \diag(1,t,t^{-1})$, 
and with norm induced from $\Q^3$ via $(a,b)\mapsto (a,b,-a-b)$.
Let
\[\begin{split}
	\tau_+=\{(a,b)\in \Q^2:a\leq b\leq -\textstyle\frac12a\}.
\end{split}
\]
The $T$-weights $(i,j,k)$ of $H^0(\PP^2,\cO(d))$ can be identified with lattice points $(i,j)$ in $\Q^2$. The following figure illustrates the case $d=4$, where 
the shaded region corresponds to $\tau_+$.

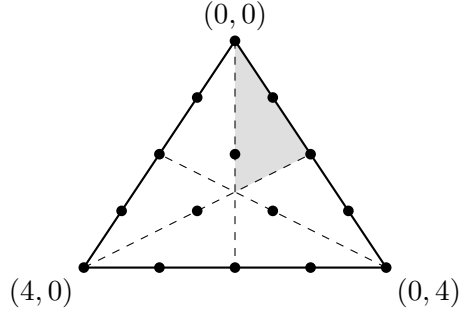
\[
	\begin{tikzpicture}[scale=1]
  \def\h{1.0}   
  \def\v{0.75}  

  \coordinate (A) at (0,0);
  \coordinate (B) at (2*\h,-4*\v);
  \coordinate (C) at (-2*\h,-4*\v);

  \coordinate (MBC) at (0,-4*\v);
  \coordinate (MCA) at (-\h,-2*\v);
  \coordinate (MAB) at (\h,-2*\v);

  \coordinate (G) at (0,-8*\v/3);

  \fill[gray!25] (A) -- (MAB) -- (G) -- cycle;

  \draw[thick] (A) -- (B) -- (C) -- cycle;

  \draw[dashed] (A) -- (MBC);
  \draw[dashed] (B) -- (MCA);
  \draw[dashed] (C) -- (MAB);

  \foreach \r in {0,...,4}{
    \foreach \i in {0,...,\r}{
      \pgfmathsetmacro{\x}{(\i - 0.5*\r)*\h}
      \pgfmathsetmacro{\y}{-(\r)*\v}
      \fill (\x,\y) circle (2pt);
    }
  }

  \node[above] at (A) {$(0,0)$};
  \node[below right] at (B) {$(0,4)$};
  \node[below left] at (C) {$(4,0)$};
\end{tikzpicture}
\]
Under the identification $\Hom(T,\Gm)_\Q\cong \Hom(\Gm,T)_\Q$, the lattice point $(i,j)$ corresponds to the one-parameter subgroup $(i-\frac{d}{3},j-\frac{d}{3},k-\frac{d}{3})$.
In particular, the barycenter $(\frac{d}{3},\frac{d}{3})$ corresponds to the origin in $\Hom(\Gm,T)_\Q$.

The set $\cB$ consists of those $\beta\in \tau_+\setminus\{0\}$ which are the closest points to the barycenter $(\tfrac{d}{3},\tfrac{d}{3})$ in the convex hull of some subset of lattice points, which is a point or a line segment.

For each $\beta\in\cB$, let $\ell_\beta\subset\Q^2$ be the line through $\beta$ orthogonal to $\beta$.
In particular, when the convex hull is a line segment, $\ell_\beta$ is precisely the line containing the convex hull.

The subspace $Z_\beta$ (resp.\ $Y_\beta'$) is the linear span in $X_d$ of the monomials corresponding to the lattice points on $\ell_\beta$ (resp. the lattice points on $\ell_\beta$ or on the opposite side of $\ell_\beta$ from the barycenter).
In particular, whenever $Z_\beta^{ss}\neq \emptyset$, the dimensions of $Z_\beta^{ss}$ and $Y_\beta^{ss}$ can be computed by counting lattice points. Moreover, $\Stab(\beta)\cong \GL_2$ whenever $\beta$ lies on the boundary of $\tau_+$, and otherwise $\Stab(\beta)=T$.

\smallskip
In the case $d=4$, this yields the following computation. Compare this with the corresponding computation for $d=6$ in \cite[Table~1]{KL1}.

\begin{table}[h]
\centering

\begin{tabular}{ccccc}
\toprule
$Z_\beta$ 
& $\Stab(\beta)$ & $P^{\Stab(\beta)}_t(Z_{\beta}^{ss})$ & $\dim Y_\beta^{ss}$ &$\mathrm{codim} S_\beta$ \\
\midrule
$(0,0)$ 		& $\GL_2$	&$\frac{1}{(1-t^2)(1-t^4)}$	&$0$&$12$\\ 
$(0,1)$  		& $T$		&$\frac{1}{(1-t^2)^2}$	&$1$&$10$\\
$(0,0),(0,1),(0,2),(0,3),(0,4)$&$\GL_2$	&$\frac{1+t^2-t^6}{(1-t^2)(1-t^4)}$	&$4$&$8$ \\
$(1,0),(0,2)$  	& $T$		&$\frac{1}{1-t^2}$	&$3$&$8$\\
$(1,0),(0,3)$  	& $T$		&$\frac{1}{1-t^2}$	&$4$&$7$\\
$(1,0),(0,4)$  	& $T$		&$\frac{1}{1-t^2}$	&$5$&$6$\\
$(2,0),(1,1),(0,2)$&$\GL_2$ 	&$\frac{1}{(1-t^2)(1-t^4)}$	&$5$&$7$\\
$(1,1),(0,3)$  	& $T$		&$\frac{1}{1-t^2}$	&$5$&$6$\\
$(1,1),(0,4)$  	& $T$		&$\frac{1}{1-t^2}$	&$6$&$5$\\
$(1,0),(1,1),(1,2),(1,3)$&$\GL_2$	&$\frac{1}{1-t^2}$	&$8$&$4$ \\
$(2,0),(0,3)$  	& $T$		&$\frac{1}{1-t^2}$	&$6$&$5$\\
\bottomrule
\noalign{\smallskip}\noalign{\smallskip}
\end{tabular}
\caption{Contribution of the unstable strata for $d=4$}
\label{table:GITus}
\end{table}
When $\Stab(\beta)=T$, there are two possibilities.
If $Z_\beta^{ss}$ is a point (as in Row~2 of Table~\ref{table:GITus}), then $P_t^{\Stab(\beta)}(Z_\beta^{ss})=\frac{1}{(1-t^2)^2}$.
Otherwise, $Z_\beta^{ss}\cong \Gm$, and hence $P_t^{\Stab(\beta)}(Z_\beta^{ss})=\frac{1}{1-t^2}$.

In Row~1 of Table~\ref{table:GITus}, where $\Stab(\beta)=\GL_2$ and $Z_\beta^{ss}$ is a point, we have
$P_t^{\Stab(\beta)}(Z_\beta^{ss})=\frac{1}{(1-t^2)(1-t^4)}$.
In Rows~3, 7, and 10 of Table~\ref{table:GITus}, where $\Stab(\beta)=\GL_2$ and $\dim Z_\beta>1$, the corresponding computations are similar to those in \cite[page~501]{KL1}, and are given by:
\[\begin{aligned}
	P_t^{\GL_2}(Y_4^{ss})=\frac{1}{1-t^2}P_t^{\SL_2}(Y_4^{ss})&
	=\frac{1+t^2-t^6}{(1-t^2)(1-t^4)},\\
	P_t^{\GL_2}(Y_2^{ss})=\frac{1}{1-t^2}P_t^{\SL_2}(Y_2^{ss})&
	=\frac{1}{(1-t^2)(1-t^4)},\\
	P_t^{\GL_2}(Y_3^{ss})=\frac{1}{1-t^2}P_t^{\SL_2}(Y_3^{ss})&
	=\frac{1}{1-t^2},
\end{aligned}
\]
respectively. Here $Y_d^{ss}\subset \PP^d$ is the $\GL_2$-semistable locus of binary forms of degree $d$. 
The Poincar\'e series $P_t^{\SL_2}(Y_d^{ss})$ is computed in \cite[16.2]{Kir84}.

By \eqref{eq:Poincare.V}, the $G$-equivariant Poincar\'e series of $X=X_4\cong \PP^{14}$ is
\beq\label{eq:poincare.X}P_t^G(X)=P_t(BG)\,P_t(X)=\frac{1-t^{30}}{(1-t^2)(1-t^4)(1-t^6)}.\eeq

Combining the above computations, we obtain the following.
\begin{theorem}\label{thm:poincare.git}
The Poincar\'e series of $\cPGIT$ is given by
\[\begin{split}
	P_t(\cPGIT)&=P_t^G(X^{ss})=\frac{1-t^{30}}{(1-t^2)(1-t^4)(1-t^6)}-\frac{t^8+t^{10}}{(1-t^2)^2}\\ 
	&=1+t^2+2t^4+3t^6+3t^8+2t^{10}+2t^{12}+\frac{t^{14}}{1-t^2}.
\end{split}\]
\end{theorem}
\begin{proof}
	The result follows from \eqref{eq:poincare.git}, \eqref{eq:poincare.X}, and Table~\ref{table:GITus}.
\end{proof}

\subsection{Poincar\'e series and the cycle class map of $\cPK$}\label{ss:poincare.PK}
Since $U^{1/2,ss}\to U^{ss}$ is a root stack and $G\to \bar G$ is a $\mu_3$-cover, we have  
\[H^*(\cPK)\cong H^*_G(U^{ss}) \and A^*(\cPK)\cong A^*_G(U^{ss}).\]

Recall that $U\to X^{ss}$ is the blowup along $Z_\DC$ with exceptional divisor $E_\DC$, and $U^{ss}\subset U$ is the corresponding GIT semistable locus.
Applying the results of Subsection~\ref{ss:Kirwan.strat} to the blowup of $X$ along $Z_\DC$ and restricting to $U$, we obtain a $G$-equivariantly perfect stratification
\[U=U^{ss}\sqcup \bigsqcup_{\beta\in \cB}S_\beta.\]

By \cite[Lemma~7.6]{Kir85}, each $Z_\beta^{ss}\subset S_\beta$ is contained in $E_\DC$.
Since $Z_\DC$ is a single $G$-orbit, the contribution of the unstable strata is determined by the induced linear action on the fiber of $E_\DC$ over a chosen base point. 
We take the point corresponding to the double conic 
$q_\DC=(xy-z^2)^2$. 
Then the fiber of $E_\DC$ over $[q_\DC]\in Z_\DC$ identifies with
\[\PP(N_{Z_\DC/X^{ss}}|_{[q_\DC]})\cong \PP H^0(\PP^1,\cO(8))=Y_8.\]
The stabilizer $\Stab_G(q_\DC)$ of $[q_\DC]$ in $G$ is a $\mu_3$-cover of $\PGL_2$, and its induced action on the fiber coincides with the natural action of $\PGL_2$.

Consequently, for each $\beta\in\cB$, the stack 
$[Z_\beta^{ss}/\Stab(\beta)]$
is naturally isomorphic to the corresponding quotient stack for the induced action on $Y_8$, and the associated unstable strata have the same codimension in $U$ and in $Y_8$.
Therefore, the contribution of the unstable strata to $P_t^G(U)$ coincides with the corresponding contribution to $P_t^{\PGL_2}(Y_8)$. 

\begin{theorem}\label{thm:PK.poincare} The Poincar\'e series of $\cPK$ is given by
\[\begin{split}
	P_t(\cPK)
	=1+2t^2+3t^4+5t^6+4t^8+3t^{10}+2t^{12}+\frac{t^{14}}{1-t^2}.
\end{split}\]
Moreover, the cycle class map $A^*(\cPK)\to H^*(\cPK)$ is an isomorphism. 
\end{theorem}
\begin{proof}
We first compute $P_t^G(U)$. By the blowup formula \cite[Lemma~7.2]{Kir85}, 
\[P_t^G(U)=P_t^G(X^{ss})+\frac{t^2-t^{18}}{1-t^2}P_t^G(Z_\DC).\]
Since $[Z_\DC/G]$ is a $\mu_3$-gerbe over $[ B\PGL_2$, we have
$P_t^G(Z_\DC)=\frac{1}{1-t^4}$,
and 
\[P_t^G(U)=P_t(\cPGIT)+\frac{t^2+t^6+t^{10}+t^{14}}{1-t^2}.\]

Next, we subtract the contribution of the unstable strata in $U$. 
As explained above, it equals the contribution of the unstable strata for the $\PGL_2$-action on $Y_8=\PP H^0(\PP^1,\cO(8))$. 
In this case, each unstable stratum is of the form $B\Gm$, and the corresponding contributions are $t^{2(4+i)}P_t(B\Gm)$ for $i=0,\dots,3$ (see \cite[16.2]{Kir84}).
Therefore,
\[P_t(\cPK)=P_t^G(U^{ss})= P_t^G(U)-\frac{t^8+t^{10}+t^{12}+t^{14}}{1-t^2} = P_t(\cPGIT)+t^2+t^4+2t^6+t^8+t^{10}.\]

For the statement on the cycle class map, it suffices (by Subsection~\ref{ss:Kirwan.strat}) to know that the cycle class maps are isomorphisms for the blowup center $[Z_\DC/G]$ and for the stacks $[Z_\beta^{ss}/\Stab(\beta)]$ induced by the unstable strata. 
The former holds because $[Z_\DC/G]$ is a $\mu_3$-gerbe over $B\PGL_2$ and $\PGL_2$ is reductive, so $A^*(B\PGL_2)\cong H^*(B\PGL_2)$ by \cite{edidin-graham-char-class}. 
The latter holds because each stack $[Z_\beta^{ss}/\Stab(\beta)]$ is isomorphic to $B\Gm$.
\end{proof}

\subsection{Poincar\'e series and the cycle class map of $\cPhat$ and $\cPH$}
Applying the same strategy as in Subsection~\ref{ss:poincare.PK} to the weighted blowup $\cPhat' \to \cPK$ and the induced stratification of $\cPhat' \setminus \cPhat$, we obtain the Poincar\'e polynomial of $\cPhat$ and the isomorphism of its cycle class map.
Using the weighted blowup $\cPhat \to \cPH$, we deduce the corresponding statements for $\cPH$.

\begin{theorem}
\label{thm:Betti.PH}
The Poincar\'e polynomials of $\cPhat$ and $\cPH$ are given by
\[\begin{aligned}
	P_t(\cPhat)&=1+3t^2+5t^4+7t^6+5t^8+3t^{10}+t^{12},\\
	P_t(\cPH)&=1+2t^2+4t^4+5t^6+4t^8+2t^{10}+t^{12}.
\end{aligned}\]
Moreover, the cycle class maps for $\cPhat$ and $\cPH$ are isomorphisms.
\end{theorem}

\begin{proof}
We first compute the Poincar\'e polynomial of $\cPhat'$.
Applying the blowup formula \cite[Lemma~7.2]{Kir85} to the blowup
$\cPhat' \to \cPK$ along $\cZ_{\pTac}$, we obtain
\[P_t(\cPhat')=P_t(\cPK)+\frac{t^2-t^{12}}{1-t^2}P_t(\cZ_{\pTac}).\]
Since $\cZ_{\pTac}\cong B(R\rtimes\mu_2)\times [\PP^1/\Gamma]$
and $H^*(B(R\rtimes \mu_2))\cong H^*(BR)^{\mu_2}$, we have
\[P_t(\cZ_{\pTac})=P_t(B(R\rtimes \mu_2))\,P_t([\PP^1/\Gamma])=\frac{1}{1-t^4}\cdot(1+t^2)=\frac{1}{1-t^2},\]
and hence
\[P_t(\cPhat')=P_t(\cPK)+\frac{t^2-t^{12}}{(1-t^2)^2}.\]

Next, we subtract the contribution of the unstable strata of $\cPhat'$, which are fibrations over $[\PP^1/\Gamma]$ with fibers given by the unstable strata for the induced $R\rtimes\mu_2$-action on $\PP(2,2,3,3,4,4)$. Their contribution is
$P_t([\PP^1/\Gamma])\cdot\sum_{i=0}^2 t^{2(3+i)}P_t(BR)=\frac{(1+t^2)(t^6+t^8+t^{10})}{1-t^2}$.
Therefore,
\[\begin{aligned}
P_t(\cPhat)&=P_t(\cPhat')-\frac{(1+t^2)(t^6+t^8+t^{10})}{1-t^2}
	=1+3t^2+5t^4+7t^6+5t^8+3t^{10}+t^{12}.
\end{aligned}\]

Applying the blowup formula to the weighted blowup $\cPhat \to \cPH$ along
$\cZ_2$, 
and using that $H^*(\cZ_2)\cong \big(H^*(\cP(2,3,4))\otimes H^*(\cP(2,3,4))\big)^{\mu_2}$,
we obtain
\[\begin{aligned}
	P_t(\cPH)&=P_t(\cPhat)-t^2P_t(\cZ_2)
	=1+2t^2+4t^4+5t^6+4t^8+2t^{10}+t^{12}.
\end{aligned}
\]

For the cycle class maps, it suffices to verify that they are isomorphisms for the blowup centers $\cZ_{\pTac}$ and $\cZ_2$, as well as the stacks $[Z_\beta^{ss}/\Stab(\beta)]$ appearing in $\cPhat'\setminus \cPhat$.
This follows from Propositions~\ref{prop:gerbe} and~\ref{prop:Z2} and the fact that each stack $[Z_\beta^{ss}/\Stab(\beta)]$ is a fibration with fiber $BR$ over $[\PP^1/\Gamma]$.
\end{proof}

We end this section with a related computation of the intersection Betti numbers of the good moduli spaces $\PGIT$ and $\PK$.
For a variety $M$ of pure dimension, let $IH^*(M)$ denote the intersection cohomology of $M$, with middle perversity and rational coefficients, and let 
	\[\begin{aligned}
		IP_t(M)&:=\sum_i \dim IH^i(M)t^i, \quad \text{and} \\
		IP'_t(M)&:=\sum_{i\leq \dim M}\dim IH^{i-2}(M)t^i+\sum_{i>\dim M}\dim IH^i(M)t^i.
	\end{aligned}
	\]
Note that when $M$ has only finite quotient singularities, $IP_t(M)=P_t(M)$.
\begin{theorem}[Intersection Betti numbers of $\PGIT$ and $\PK$] \label{thm:IP.GIT.K}
	We have
	\[\begin{aligned}
		&IP_t(\PGIT)=1+t^2+2t^4+3t^6+2t^8+t^{10}+t^{12},\\
		&IP_t(\PK)=1+2t^2+3t^4+5t^6+3t^8+2t^{10}+t^{12}.
	\end{aligned}\]
\end{theorem}
\begin{proof}
	Applying \cite[Proposition~2.1]{kirwan-ic} to $\Phat\to \PK$, we have 
	\[\begin{split}
		IP_t(\PK)&=IP_t(\Phat)-IP_t(\PP^1/\Gamma)\cdot IP'_t(Z_2)=P_t(\Phat)-(1+t^2)(t^2+t^4+t^6+t^8),
	\end{split}\]
	which gives the asserted formula for $IP_t(\PK)$. 
	Similarly, 
	\[IP_t(\PGIT)=IP_t(\PK)-IP'_t(Y_8/\!\!/\SL_2).\]
	To compute the last term, we apply it 
	to 
	$\widehat Y_8/\!\!/\SL_2\to Y_8/\!\!/\SL_2$ and obtain 
	\[IP_t(Y_8/\!\!/\SL_2)=P_t(\widehat Y_8/\!\!/\SL_2)-IP'_t(Z_2)=1+t^2+2t^4+2t^6+t^8+t^{10}.\] 
	The last equality follows by \cite[page~83]{Kir85}.
	Subtracting $IP'_t(Y_8/\!\!/\SL_2)=t^2+t^4+2t^6+t^8+t^{10}$ 
	from $IP_t(\PK)$ gives the formula for $IP_t(\PGIT)$.
\end{proof}

\section{Chow ring of $\cPGIT$}\label{s:chow.pgit}

In this section, we compute the Chow ring of the GIT moduli stack $\cPGIT$. 
It suffices to compute $A_G^*(X^{ss})$. 
We determine $A_G^*(X^{ss})$ via the localization sequence by computing the relations in $A_G^*(X)$ induced by the unstable locus, following the approach of \cite{COP}.

\subsection{Unstable plane quartics}
Let $V=\C^3$ so that $\PP^2=\PP V$ and $X=\PP(\Sym^4 V^*)$.
Then 
\[X\setminus X^{ss}=\bigsqcup_{\beta\in \cB} S_\beta\]
has exactly two irreducible components, which are the closures of the two strata $S_\beta$ corresponding to the last two rows of Table~\ref{table:GITus}. These two components parametrize 
\begin{enumerate}
	\item plane quartics with a triple point, and 
	\item plane quartics that are unions of a cubic and its inflectional tangent line,
\end{enumerate}
 respectively (see \cite[\S4.2]{GIT}).
More precisely, every point in these components is, up to the $G$-action, represented by a quartic form
\beq\label{eq:quartic}F=\sum_{i,j\ge 0,~ i+j\le 4} a_{ij}x^iy^jz^{4-i-j}\in \C[x,y,z]_4\eeq
satisfying one of the two conditions, corresponding to (1) and (2) above:
\begin{enumerate}
	\item $a_{ij}=0$ for all $i>1$, equivalently $F\in (y,z)^3$;
	\item $F=zF_3$, where $F_3\in \langle x^2z, xyz, xz^2, y^3, y^2z, yz^2, z^3\rangle= \langle y^3\rangle + z\C[x,y,z]_2$.
\end{enumerate}
(These are the loci $Y_\beta'=\overline{Y_\beta^{ss}}$ in the last two rows of Table~\ref{table:GITus}.)

We denote the two unstable strata $S_\beta$ by $S_{\tr}$ and $S_{\flex}$.
In particular,
\[X\setminus X^{ss}=\overline S_\tr \cup \overline S_\flex.\]
These two components have codimensions $4$ and $5$, respectively, according to Table~\ref{table:GITus}.
The localization sequence then yields
\beq\label{eq:localize.X}A^{*-4}_G(\overline S_{\tr})\oplus A^{*-5}_G(\overline S_{\flex})\lra A^*_G(X) \lra A^*_G(X^{ss})\lra 0.\eeq

By the projective bundle formula, the equivariant Chow ring of $X$ is
\beq\label{eq:Chowring.X}A^*_G(X)=\Q[H,c_2,c_3]/(r_X),\eeq
where $H=c_1^G(\cO_X(1))$ is the equivariant hyperplane class, and $c_2,c_3$ are the Chern classes of the universal $G$-bundle.
The relation $r_X$ is given by
\[r_X=H^{15}+c_1^G(\Sym^4V^*)H^{14}+\cdots + c_{15}^G(\Sym^4V^*).\]

Thus, it remains to determine the images of $A^{*-4}_G(\overline S_{\tr})$ and $A^{*-5}_G(\overline S_{\flex})$ inside $A^*_G(X)$.

\subsection{Relations coming from the triple-point locus}\label{ss:rel.triple}

We determine the relations arising from excising the irreducible component $\overline S_{\tr}$ of plane quartics with a triple point.

Consider the exact sequence of $G$-equivariant vector bundles on $\PP V$
\[0 \lra \sK_{\tr}\lra \Sym^4 V^*\lra \sP^2\big(\cO_{\PP V}(4)\big)\lra 0,\]
where $\sP^m(\cO_{\PP V}(4))$ denotes the bundle of $m$-th principal parts of $\cO_{\PP V}(4)$.
By construction, 
\[\PP(\sK_{\tr})\subset\PP V\times X\]
is the incidence correspondence parametrizing pairs $(p,F)$ such that $F$ vanishes to order at least $3$ at $p$.
Hence, the projection $\PP(\sK_{\tr})\to X$ is $G$-equivariant and proper, with image $\overline S_{\tr}$.

Consequently, the image of the pushforward
\[A_*^G(\overline S_{\tr})\lra A_*^G(X)\]
coincides with that of the composition 
\beq\label{eq:comp.triple}
	A_*^G\big(\PP(\sK_{\tr})\big) \lra A_*^G(\PP V\times X)\xrightarrow{\;p_{X*}\;} A_*^G(X),
\eeq
where $p_X:\PP V\times X\to X$ is the second projection.

The subscheme $\PP(\sK_{\tr})\subset \PP V\times X$ is the zero locus of the universal section
$\cO_{\PP V}\boxtimes \cO_X(-1) \to \sP^2\big(\cO_{\PP V}(4)\big)\boxtimes \cO_X$.
Equivalently, it is the zero locus of a section of the rank $6$ bundle
\[\sE_{\tr}:=\sP^2\big(\cO_{\PP V}(4)\big)\boxtimes \cO_X(1).\]

It follows that the image of the first map in \eqref{eq:comp.triple} is the ideal generated by the top Chern class $c_6^G(\sE_{\tr})$ in $A_G^*(\PP V\times X)$.
Let 
\[\zeta:=c_1^G(\cO_{\PP V}(1)) \and H:=c_1^G(\cO_X(1)).\]
As $p_X$ is a $\PP^2$-bundle, the image of \eqref{eq:comp.triple} is generated by
\[p_{X*}\big(\zeta^j\,c_6^G(\sE_{\tr})\big), \qquad j=0,1,2.\]

Using the standard filtration of $\sP^2(\cO_{\PP V}(4))$ with graded pieces $\cO_{\PP V}(4)\otimes\Sym^k\Omega_{\PP V}$ for $k=0,1,2$ (see \cite[Theorem~7.2]{3264}), 
we reduce the computation of the Chern classes to that of symmetric powers of $\Omega_{\PP V}$.
The latter are computed using the exact sequences
\[0\lra \Sym^k\Omega_{\PP V}\lra \Sym^k\big(V^*(-1)\big)\lra \Sym^{k-1}\big(V^*(-1)\big)\lra 0,\]
where $\Sym^k(V^*(-1))\cong\cO_{\PP V}(-k)\otimes \Sym^kV^*$.
It follows that
\[c^G\big(\sP^2(\cO_{\PP V}(4))\big) = c^G\big(\cO_{\PP V}(2)\otimes \Sym^2 V^*\big). \]
This gives an explicit expression of $c_6^G(\sE_{\tr})$ in terms of $H,c_2,c_3$ and $\zeta$, and hence 
the following.
\[\begin{aligned}
	r^{\tr}_4 &:= p_{X*}\big(c^G_6(\sE_{\tr})\big)&&=60H^4-120c_2H^2-276c_3H,\\
	r^{\tr}_5 &:= p_{X*}\big(\zeta\,c^G_6(\sE_{\tr})\big)&&=12H^5-120c_2H^3-282c_3H^2+48c_2^2H+96c_2c_3,\\
	r^{\tr}_6 &:= p_{X*}\big(\zeta^2c^G_6(\sE_{\tr})\big)&&=H^6-55c_2H^4-167c_3H^3+124c_2^2H^2+304c_2c_3H+112c_3^2.
\end{aligned}\]

We summarize the result as follows.

\begin{proposition}\label{prop:rel.Str}
The image of the pushforward map
\[A_G^{*-4}(\overline S_{\tr})\lra A_G^*(X)\]
is the ideal generated by the three classes $r^\tr_4$, $r^\tr_5$ and $r^\tr_6$. 
\end{proposition}

\subsection{Relations coming from the flex locus}\label{ss:rel.flex}

We determine the relations arising from excising the irreducible component $\overline S_{\flex}$ of quartics that are unions of a cubic and its inflectional tangent line.
As in the triple-point case, we replace $\overline S_{\flex}$ by an incidence correspondence recording the cubic, its flex point, and the flex line.

Let $X_d=\PP H^0(\PP^2,\cO(d))=\PP \Sym^d V^*$.
Consider the multiplication map
\[m:X_1\times X_3\lra X_4=X, \qquad (\ell,g)\longmapsto \ell g.\]
It is generically injective, and its image has codimension $3$ in $X$.
The locus $\overline S_{\flex}$ is contained in the image of $m$ and corresponds to those $\ell g$ for which  $\ell=0$ is a flex line of the cubic $g=0$.

To describe $\overline S_{\flex}$, we keep track of the flex point and the tangent line.
Accordingly, let 
\[\widetilde S_\flex\subset P_{1,3}:=\PP V\times X_1\times X_3\] 
denote the closure of the locus of triples $(p,\ell,g)$ such that:
\begin{itemize}
	\item the cubic $g=0$ is smooth and $p\in(\ell=0)\cap(g=0)$;
	\item the line $\ell=0$ is tangent to $g=0$ at $p$;
	\item the point $p$ is an inflection point of $g=0$.
\end{itemize}
The image of $\widetilde S_\flex$ under $P_{1,3}\to X_1\times X_3\xrightarrow{m} X$ is precisely $\overline S_{\flex}$.

We realize $\widetilde S_\flex$ as the zero locus of a natural section on $P_{1,3}$.
Let
\[\Fl:=\{(p,\ell)\in \PP V\times X_1: p\in \ell\},\]
be the incidence variety, and let $\eta:\Fl\to X_1$ denote the natural projection.
Consider the bundle of relative principal parts
$\sP^2_{\eta}\big(\cO_{\PP V}(3)\big)$
on $\Fl$.
For $(p,\ell)\in\Fl$, its fiber consists of principal parts of order at most $2$ of a cubic along $\ell$ at $p$.
There is a natural surjection
\[\Sym^3V^*\otimes \cO_\Fl \lra \sP^2_{\eta}\big(\cO_{\PP V}(3)\big),\]
given by restricting cubic forms to $\ell$ and truncating to principal parts.

Let $\sK_{\flex}$ denote its kernel bundle.
Pulling back this map to $\Fl\times X_3$ via the first projection and composing with the pullback of the universal section
\[\cO_{X_3}(-1)\longrightarrow \Sym^3V^*\otimes \cO_{X_3}\]
via the second projection, we obtain a section of the rank $3$ bundle 
\[\sE_{\flex}:=\sP^2_{\eta}\big(\cO_{\PP V}(3)\big)\boxtimes \cO_{X_3}(1)\]
on $\Fl\times X_3$, whose zero locus is $\PP(\sK_{\flex})$. 
Under the natural inclusion $\Fl\times X_3\hookrightarrow P_{1,3}$, $\PP(\sK_{\flex})$ is identified with $\widetilde S_\flex$ (cf.~\cite[Chapter~11]{3264}).

Let $\zeta, h_1, h_3$ denote the $G$-equivariant hyperplane classes of $\PP V$, $X_1$, and $X_3$, respectively.
Let $\cO_{\Fl}(a,b):=\cO_{\PP V\times\PP V^*}(a,b)|_{\Fl}$. Using $\Omega_\eta\cong\cO_{\Fl}(-2,1)$ and
\cite[Theorem~11.2]{3264}, we get
\[c_3^G\big(\sP^2_{\eta}(\cO_{\PP V}(3))\big)=\prod_{k=0}^2 c_1^G\big(\cO_{\Fl}(3-2k,k)\big)=\prod_{k=0}^2 \big((3-2k)\zeta+kh_1\big).\]
Hence the class of $\widetilde S_\flex$ in $A_G^4(P_{1,3})$ is given by
\[
	[\widetilde S_\flex]=(\zeta+h_1)(3\zeta+h_3)(\zeta+h_1+h_3)(-\zeta+2h_1+h_3),
\]
and the $G$-equivariant Chow ring of $P_{1,3}$ is
\[A^*_G(P_{1,3})=\frac{\Q[\zeta,h_1,h_3]}{(\zeta^3+c_2\zeta+c_3,\; h_1^3+c_2h_1-c_3,\; r_{X_3})},\]
where $r_{X_3}=h_3^{10}+c_1^G(\Sym^3V^*)h_3^9+\cdots+c_{10}^G(\Sym^3V^*)$.

We now push forward the classes $\zeta^j\cdot [\widetilde S_\flex]$ along the projection
\[p_{13}:P_{1,3}\longrightarrow X_1\times X_3 .\]
Since $p_{13*}(\zeta^j)=0$ for $j<2$ and $p_{13*}(\zeta^2)=1$, the image of the composition 
\beq\label{eq:pushtoX1X3}A^{*-2}_G(\widetilde S_\flex)\lra A^{*+2}_G(P_{1,3})\xrightarrow{p_{13*}} A^*_G(X_1\times X_3)\eeq 
is generated by $p_{13*}(\zeta^j\cdot[\widetilde S_\flex])$ for $j=0,1,2$.
Up to nonzero scalar multiples, 
\[
\begin{aligned}
	f_2&:=p_{13*}([\widetilde S_\flex])&&=3h_1^2+3h_1h_3+h_3^2+c_2,\\
	f_3&:=p_{13*}(\zeta[\widetilde S_\flex])&&=12h_1^2h_3+6h_1h_3^2+h_3^3-6h_1c_2+h_3c_2+9c_3,\\
	f_4&:=p_{13*}(\zeta^2[\widetilde S_\flex])&&=3h_1^2h_3^2+h_1h_3^3-9h_1^2c_2-11h_1h_3c_2-3h_3^2c_2-3c_2^2+3h_3c_3.
\end{aligned}
\]

We write $H$ for $m^*H$, so that $H=h_1+h_3$.
Substituting $h_3=H-h_1$, we conclude that the image of \eqref{eq:pushtoX1X3} is generated by the three classes:
\beq\label{eq:flex.gen.X1X3}
\begin{aligned}
    &f_2=(H^2+c_2)\cdot 1+H\cdot h_1+h_1^2,\\
    &f_3=(H^3+Hc_2+2c_3)\cdot 1+3H^2\cdot h_1+3H\cdot h_1^2,\\
    &f_4=(-3H^2c_2-3c_2^2)\cdot 1+(H^3-2Hc_2-c_3)\cdot h_1-3c_2\cdot h_1^2.
\end{aligned}
\eeq

Since $\overline S_{\flex}$ is the image of $\widetilde S_\flex$ under $m$, the image of the pushforward map
$A^G_*(\overline S_{\flex})\to A^G_*(X)$
is generated by $m_*$ applied to the three classes $f_2,f_3,f_4$. 
By the projection formula, this amounts to replacing $1, h_1, h_1^2$ in \eqref{eq:flex.gen.X1X3} with $m_*(1), m_*(h_1), m_*(h_1^2)$.
In particular, it suffices to compute $m_*(h_1^i)$ for $i=0,1,2$.

Each class $m_*(h_1^i)$ lies in $A_G^{i+3}(X)$ and can be written as
\[m_*(h_1^i)=\sum_{j=0}^{i+3}\alpha_{ij}H^{i+3-j},\quad \text{ where }~\alpha_{ij}\in A^j(BG).\]
To determine the coefficients $\alpha_{ij}$, we proceed as follows. 
For each $k\geq 0$, we multiply both sides by $H^{11-i+k}$ and push forward to $BG$ (cf.~\cite[Lemma~6]{COP}).
This yields
\[\sum_{j=0}^{k}\alpha_{ij}\,s^G_{k-j}(\Sym^4V^*) =\sum_{j=0}^{k}\binom{11-i+k}{2-i+j} s^G_{j}(V^*)\,s^G_{k-j}(\Sym^3V^*),\]
which determine $\alpha_{ij}$ inductively.
This computation uses the Segre classes
\beq\label{eq:segre}\begin{aligned}
	s^G(V^*)&=1-c_2+c_3+c_2^2-2c_2c_3+\cdots,\\
	s^G(\Sym^3V^*)&=1-15c_2+27c_3+162c_2^2-648c_2c_3+\cdots,\\
	s^G(\Sym^4V^*)&=1-35c_2+77c_3+798c_2^2-3828c_2c_3+\cdots,
\end{aligned}\eeq
whose expansions follow from the splitting principle.
We record the resulting formulas below.
\begin{lemma}\label{lem:pushforward.flex}
For $i=0,1,2$, we have
\[\begin{aligned}
	m_*(1)&=55H^3+40c_2H+224c_3,\\
	m_*(h_1)&=10H^4-50c_2H^2+296c_3H,\\
	m_*(h_1^2)&=H^5-35c_2H^3+170c_3H^2+24c_2^2H-96c_2c_3.
\end{aligned}\]
\end{lemma}
\begin{proof}
These are obtained by solving the recurrence for $\alpha_{ij}$, using \eqref{eq:segre}. 
\end{proof}

\begin{proposition}\label{prop:rel.Sflex}
The image of the pushforward map
\[A_G^{*-5}(\overline S_{\flex})\lra A_G^*(X)\]
is the ideal generated by the following three classes:
\[\begin{aligned}
	r^\flex_5&:= 33H^5 + 5c_2H^3 + 345c_3H^2 + 32c_2^2H + 64c_2c_3,\\
	r^\flex_6&:= 22H^6 - 40c_2H^4 + 433c_3H^3 + 28c_2^2H^2 + 4c_2c_3H + 112c_3^2,\\
	r^\flex_7&:= 5H^7 - 119c_2H^5 + 143c_3H^4 - 40c_2^2H^3 - 862c_2c_3H^2
      - 96c_2^3H - 148c_3^2H - 192c_2^2c_3.
\end{aligned}\]
\end{proposition}
\begin{proof}
The image is generated by $m_*f_i$ for $i=2,3,4$. 
Using Lemma~\ref{lem:pushforward.flex}, these are, up to nonzero scalar multiples, the classes $r^\flex_{i+3}$.
\end{proof}

We now combine the relations coming from $\overline S_\tr$ (Proposition~\ref{prop:rel.Str}) and from $\overline S_\flex$ 
(Proposition~\ref{prop:rel.Sflex}) to determine the Chow ring of $X^{ss}$.

\begin{theorem}[Chow ring of $\cPGIT$]\label{thm:Chowring.GIT}
We have
\[A^*(\cPGIT) \cong\Q[H,c_2,c_3]\big/\big(r^\tr_4, r^\tr_5, r^\tr_6, r^\flex_5\big).\]
\end{theorem}

\begin{proof}
Combining \eqref{eq:localize.X}, \eqref{eq:Chowring.X} and Propositions~\ref{prop:rel.Str} and~\ref{prop:rel.Sflex}, we see that the kernel of the natural surjection $\Q[H,c_2,c_3]\to A_G^*(X^{ss})$ is generated by $r^\tr_4,r^\tr_5,r^\tr_6,r^\flex_5,r^\flex_6,r^\flex_7$ and $r_X$.

One readily checks that $r^{\flex}_6$, $r^{\flex}_7$, and $r_X$
lie in the ideal $(r^{\tr}_4, r^{\tr}_5, r^{\tr}_6,r^{\flex}_5)$.
\end{proof}

\begin{remark}
As a sanity check, the Hilbert series of the quotient ring in Theorem~\ref{thm:Chowring.GIT} 
agrees with the Poincar\'e series of $\cPGIT$ computed in Theorem~\ref{thm:poincare.git}, after the substitution $t \mapsto \sqrt{t}$.
\end{remark}

\begin{remark}\label{rem:Hodge.GIT}
The generators are in fact given by tautological classes:
\[c_1(\EE^{\mathrm{GIT}}) = 3H \and  c_i\left((\EE^{\mathrm{GIT}})^* \otimes (\det \EE^{\mathrm{GIT}})^{1/3}\right) = c_i(V) = c_i \quad (i=2,3).\]
This follows from Lemma~\ref{lem:Hodge.GIT}.
\end{remark}

\section{Chow ring of $\cPK$}\label{s:chow.pk}

In this section, we compute the Chow ring of $\cPK$. We identify $A^*(\cPK)$ with $A^*_G(U^{1/2,ss})$.
The computation proceeds in three main steps. We
\begin{enumerate}
  \item We first compute $A^*_{G}(U^{1/2})$ by using the blowup formula. 
  \item We then determine the ideal generated by pushforwards from the unstable locus 
  	\[U^{1/2,us}:=U^{1/2}\setminus U^{1/2,ss}\]
  	by introducing its natural resolution.
  \item Finally, applying the localization sequence for  $U^{1/2,ss}\subset U^{1/2}$, we obtain a presentation of $A^*(\cPK)$; see Theorem~\ref{thm:Chowring.PK}.
\end{enumerate}

In doing so, we also consider a parallel resolution of the non-stable locus
\[U^{1/2,ns}:=U^{1/2}\setminus U^{1/2,s},\]
where $U^{1/2,s}\subset U^{1/2}$ denotes the GIT stable locus. 
This resolution will be used again in Section~\ref{s:chow.phat}.
As a byproduct, we determine the ideal generated by pushforwards from $U^{1/2,ns}$, and hence
the Chow rings of the stable loci
\[\cPGITs\cong \cPH\setminus \overline\cZ_1 \and \cPKs\cong \cPH\setminus \cZ_2,\]
where the superscript $s$ denotes the locus of stable objects; see Theorem~\ref{thm:Chowring.stable}.

\subsection{Blowup formulas for stack-theoretic weighted blowups}
We recall the formulas for the Chow ring of a stack-theoretic weighted blowup, as well as for the class of a proper transform. 
We refer to \cite{ArenaObinna,Arena} for general statements and proofs.

Let $\sE$ be a rank $r$ vector bundle over $Z$ equipped with a fiberwise $\Gm$-action.
We define 
\[P_\sE(t):=c_{r}^{\Gm}(\sE) \in A^*_{\Gm}(Z)\cong A^*(Z)[t],\]
viewed as a polynomial in the equivariant parameter $t$. 
The weighted Segre class of $\sE$ is 
\[s^{wt}(\sE):=\frac{1}{P_\sE(1)}.\]

\begin{theorem}[{\cite[Corollary~6.5]{ArenaObinna}}]\label{thm:blowup.formula1}
	Let $i:Z\hookrightarrow Y$ be a closed immersion, where $Y$ and $Z$ are smooth.
	Assume that 
	\begin{itemize}
		\item the restriction map $i^*:A^*(Y)\to A^*(Z)$ is surjective, and
		\item the normal bundle $N_{Z/Y}$ is equipped with a fiberwise $\Gm$-action. 
	\end{itemize}
	Let $f:\tY\to Y$ be the corresponding stack-theoretic weighted blowup along $Z$, with exceptional divisor $E$.
	Set $t:=-[E]$.
	Then $f^*$ and the class $t$ induce an isomorphism
	\[A^*(\tY)\cong\frac{A^*(Y)[t]}{\big(t\cdot \ker(i^*),\;Q_{N_{Z/Y}}\big)},\]
	where $Q_{N_{Z/Y}}:=P_{N_{Z/Y}}(t)-P_{N_{Z/Y}}(0)+f^*[Z]$. 
\end{theorem}

Let $g:E\to Z$ be the restriction of $f$ and $j:E\hookrightarrow \tY$ be the inclusion.

\begin{theorem}[{\cite[Theorem~8.2.1 and Remark~8.1.12]{Arena}}]\label{thm:blowup.formula2}
	Let $Y,\tY,Z,E$ be as in Theorem~\ref{thm:blowup.formula1}.
	Let $i_S:S\hookrightarrow Y$ be a closed immersion of codimension $c$, such that both $S$ and $S\cap Z$ are smooth.
	Let $\tS$ be the proper transform of $S$ in $\tY$. 
	Then
	\[[\tS]=f^*[S]-\left\{j_*\left(P_{N_{Z/Y}}(1)(1+t+t^2+\cdots) \cdot g^*i_{S*}s^{wt}(N_{S\cap Z/S})\right)\right\}^c\]
	where $t=-[E]$ and $\{\cdots\}^c$ denotes the codimension-$c$ component.
\end{theorem}

\subsection{The ring $A^*_G(U^{1/2})$} 
To apply Theorem~\ref{thm:blowup.formula1} to the weighted blowup 
\[f:U^{1/2}\lra X^{ss}\] 
along $Z_\DC$,
we verify the surjectivity of the restriction map
\beq\label{eq:rest.DC}
A^*_G(U^{1/2})\lra A^*_G(Z_\DC)\cong A^*(B\SL_2)\cong \Q[c_2'],
\eeq
where $c_2'$ denotes the second Chern class of the standard $\SL_2$-representation.
\begin{lemma}\label{lem:rest.DC}
The restriction map~\eqref{eq:rest.DC} is surjective with kernel $(H,c_3)$.
\end{lemma}
\begin{proof}
Composing the map $A^*(BG)\to A^*_G(U^{1/2})$ with \eqref{eq:rest.DC}, we obtain $A^*(BG)\to \Q[c_2']$.
This map sends $c_2$ to $4c_2'$ and $c_3$ to $0$, and is therefore surjective.
It follows that \eqref{eq:rest.DC} is also surjective.
Since $\SL_2$ has no nontrivial characters, $H$ restricts to $0$.
Since $A^*_G(U^{1/2})$ is generated by $H,c_2,c_3$, it follows that the kernel of \eqref{eq:rest.DC} is generated by $H$ and $c_3$.
\end{proof}

Identifying $c_2=4c_2'$, we can express the total Chern class of $\Sym^8\C^2$ as
\[c^{\SL_2}(\Sym^8\C^2)=1+30c_2+273c_2^2+820c_2^3+576c_2^4,\]
and hence the $\Gm$-equivariant weighted Chern polynomial of $N_{Z_{\DC}/X^{ss}}$
is
\beq\label{eq:PDC}
	P_{\DC}(t)=c^{\SL_2}_{2t}(\Sym^8\C^2)=32\big(16t^9 + 120c_2t^7 + 273c_2^2t^5 + 205c_2^3t^3 + 36c_2^4t\big).
\eeq

We now apply Theorem~\ref{thm:blowup.formula1} to $f$ to obtain the following.
\begin{proposition}\label{prop:Chowring.Uhalf}
The $G$-equivariant Chow ring of $U^{1/2}$ is given by
\[A^*_G(U^{1/2})\cong\frac{A^*_G(X^{ss})[e]}{(eH,\; ec_3,\; Q_\DC)}\cong\frac{\Q[H,c_2,c_3,e]}{(r^\tr_4,r^\tr_5,r^\tr_6,r^\flex_5,\; eH,\; ec_3,\; Q_\DC)},\]
where $e=-[\cEK]$, and $Q_\DC=P_\DC(e)+f^*[Z_\DC]$.
\end{proposition}
\begin{proof}
The first isomorphism follows from Theorem~\ref{thm:blowup.formula1} and Lemma~\ref{lem:rest.DC}.
The second isomorphism then follows from Theorem~\ref{thm:Chowring.GIT}.
\end{proof}

\begin{remark}[Computation of $Q_\DC$]
The class $[Z_\DC]\in A^9_G(X^{ss})$ can be computed via pushforward along the second Veronese map
\[d:\PP\Sym^2V^*\lra \PP\Sym^4V^*,\quad [q]\mapsto [q^2],\]
followed by pushforward to $BG$ as in the previous computation. 
This in turn yields an explicit formula for $Q_\DC$.
We record the resulting formula for completeness.
Write
\[d_*(1)=\sum_{i=0}^9 \alpha_i H^{9-i},\quad \alpha_i\in A^i(BG).\]
To determine $\alpha_i$, multiply both sides by $H^{5+k}$ and push forward to $BG$.
Since $H$ restricts to twice the equivariant hyperplane class on $\PP\Sym^2V^*$, the projective bundle formula implies
\[2^{5+k}s_k(\Sym^2V^*)=\sum_{i=0}^9 \alpha_i\, s_{k-i}(\Sym^4V^*), \quad k\ge 0.\]
For $k=0,\dots,9$, this gives the recursion
\[\alpha_k=2^{5+k}s_k(\Sym^2V^*)-\sum_{i=0}^{k-1}\alpha_i\, s_{k-i}(\Sym^4V^*),\]
which uniquely determines the coefficients $\alpha_k$.
From this, one gets
\beq\label{eq:fund.DC}
\begin{aligned}
	\left[Z_\DC\right]&=32\Big(H^9+15c_2H^7-21c_3H^6+63c_2^2H^5-174c_2c_3H^4 +(85c_2^3+363c_3^2)H^3\\
	&\qquad \qquad -189c_2^2c_3H^2+(36c_2^4+483c_2c_3^2)H -(36c_2^3c_3+343c_3^3)\Big).
\end{aligned}
\eeq

Substituting $t=e$ in~\eqref{eq:PDC} and using~\eqref{eq:fund.DC}, one obtains a formula for $Q_\DC$.
\end{remark}

\subsection{Unstable locus of $U^{1/2}$}

By the localization sequence, 
the pullback
\[A^*_G(U^{1/2}) \lra A^*_G(U^{1/2,ss}) \cong A^*(\cPK)\]
is surjective, with kernel given by
the image of the pushforward 
\beq\label{eq:push.Uunstable}A_*^G\big(U^{1/2,us}\big)\lra A_*^G(U^{1/2}).\eeq
Since $A^*_G(U^{1/2})$ has already been computed, it remains to determine the image of \eqref{eq:push.Uunstable}.

In this subsection, we describe the unstable locus $U^{1/2,us}$. 
For later use, we also describe the non-stable locus $U^{1/2,ns}$ 
(see Lemmas~\ref{lem:Uus} and~\ref{lem:Uns}).

The image $f(U^{1/2,us})$ consists of quartics whose $G$-orbit closures contain a double conic; equivalently, they are identified with double conics under the quotient map $X^{ss}\to X^{ss}/\!\!/G=\PGIT$. Moreover, $U^{1/2,us}$ is the proper transform of $f(U^{1/2,us})$; see \cite[Remark~7.17]{Kir85}. 

We give a characterization of $f(U^{1/2,us})$.
Recall that the $G$-action on $X$ is induced by the dual action on $V^*$, so that $\lambda_R(t)=\diag(t,t^{-1},1)$ acts on $(x,y,z)$ with weights $(-1,1,0)$.
\begin{lemma}\label{lem:Uus}
Let $F=\sum_{0\leq i,j,k\leq 4,i+j+k=4}a_{ijk}x^iy^jz^k$ be a semistable plane quartic.
The following are equivalent:
\begin{enumerate}
	\item $[F]\in f(U^{1/2,us})$;
	\item Up to the $G$-action, $F$ contains none of the monomials
		$x^4,x^3y,x^3z,x^2yz,x^2z^2,xz^3$,
	and satisfies the discriminant condition $a_{112}^2=4a_{004}a_{220}$.
\end{enumerate}
Moreover, $U^{1/2,us}$ is the proper transform of $f(U^{1/2,us})$ under $f$.
\end{lemma}
\begin{proof}
The equivalence follows from the Hilbert--Mumford criterion. 

For $(2)\Rightarrow(1)$, suppose first that $F$ is not a double conic. Then $[F]$ defines a point of $U^{1/2}$, and the one-parameter subgroup $\lambda_R:t\mapsto \diag(t,t^{-1},1)$ destabilizes it. Hence $[F]\in U^{1/2,us}$, and therefore $[F]\in f(U^{1/2,us})$.
Moreover, the $\lambda_R$-limit of such points in $X^{ss}$ is the double conic $(xy-z^2)^2$. Since the corresponding limit in $U^{1/2}$ is unstable, and since $Z_\DC$ is a single $G$-orbit, it follows that $Z_\DC\subset f(U^{1/2,us})$. Thus (1) also holds when $F$ is a double conic.

Conversely, suppose that $[F]=f(u)$ for $u\in U^{1/2,us}$.
Then there exists a destabilizing one-parameter subgroup $\lambda$ such that the limit $\lim_{t\to 0} \lambda(t)\cdot u$ lies over $Z_\DC$.
Hence, $F$ degenerates to a double conic.
Up to the $G$-action, this limit is $(xy-z^2)^2$.
Since $\lambda$ fixes the limit, it lies in the stabilizer of $(xy-z^2)^2$, and hence it is conjugate to $\lambda_R$. 
Thus, after a $G$-action, we may assume $\lambda=\lambda_R$. 
Then, the fact $\lim_{t\to 0}\lambda(t)\cdot F=(xy-z^2)^2$ forces $F$ to be of the form
$(xy-z^2)^2 + \text{(terms of strictly positive $\lambda$-weight)}$,
which is precisely the condition (2).

The last assertion follows from \cite[Remark~7.17]{Kir85}.
\end{proof}

Similarly, $f(U^{1/2,ns})$ consists of quartics whose $G$-orbit closures contain a strictly polystable quartic as in \eqref{eq:plane.tacnodal.eq}, and coincides with the strictly semistable locus $X^{sss}:=X^{ss}\setminus X^s$ of $X$.

\begin{lemma}\label{lem:Uns}
Let $F$ be a semistable quartic.
The following are equivalent:
\begin{enumerate}
\item $[F]\in X^{sss}$;
\item Up to the $G$-action, $F$ contains none of the monomials
	$x^4,x^3y,x^3z,x^2yz,x^2z^2,xz^3$.
\end{enumerate}
Moreover, $f(U^{1/2,ns})=X^{sss}$, and $U^{1/2,ns}$ is the proper transform of $X^{sss}$ 
under $f$.
\end{lemma}

\begin{proof}
The proof of the equivalence follows the same strategy as that of Lemma~\ref{lem:Uus}. 
In this case, up to the $G$-action, the $\lambda$-limit is of the form~\eqref{eq:plane.tacnodal.eq}, 
so no discriminant condition appears.

To be more precise, $(2)\Rightarrow (1)$ follows from the Hilbert--Mumford criterion, since every quartic satisfying (2) is not stable with respect to $\lambda_R$. 

For $(1)\Rightarrow(2)$, let $[F]\in X^{sss}$. By the Hilbert--Mumford criterion, there exists a one-parameter subgroup $\lambda$ such that $[F]$ is strictly semistable with respect to $\lambda$. 
The corresponding strictly polystable limits are precisely the quartics of the form~\eqref{eq:plane.tacnodal.eq}, up to the $G$-action. 
Such a limit is stabilized by $\lambda_R$, and hence $\lambda$ is conjugate to $\lambda_R$. This forces $F$ to satisfy (2).

Since $f^{-1}(X^s)\subset U^{1/2,s}$, we have $f(U^{1/2,ns})\subset X^{sss}$.
Conversely, if 
$F$ satisfies (2), then the same argument with $\lambda_R$ shows that $[F]\in f(U^{1/2,ns})$. 
Hence $f(U^{1/2,ns})=X^{sss}$.

For the last assertion, since $f$ is an isomorphism over $X^{ss}\setminus Z_\DC$, it suffices to show that every point of $U^{1/2,ns}$ lying over $Z_\DC$ lies in the closure of $f^{-1}(X^{sss}\setminus Z_\DC)$.
This follows by deforming double conics, up to the $G$-action, by quartics of the form $F_t=(xy-z^2)^2+t(az^4+G)$, where $G$ consists of terms of strictly positive $\lambda_R$-weight. 
This completes the proof.
\end{proof}

\begin{remark}\label{rem:sss}
	By the GIT analysis in Subsection~\ref{ss:GIT.planecurve}, semistability of $F$ forces $a_{220}\neq 0$ and $(a_{112},a_{004})\neq (0,0)$ in Lemmas~\ref{lem:Uus} and~\ref{lem:Uns}.
	In particular, we have $a_{004}=\frac{a_{112}^2}{4a_{220}}$ in Lemma~\ref{lem:Uus}.
\end{remark}

\subsection{Resolutions of $U^{1/2,us}$ and $U^{1/2,ns}$}
To compute the pushforward images, we introduce suitable resolutions of $U^{1/2,us}$ and $U^{1/2,ns}$.
More precisely, we construct closed substacks of $\Fl\times U^{1/2}$ whose images under the second projection coincide with these loci.

Let $\Fl$ be the incidence variety in $\PP V\times \PP V^*$, with natural projections
\[\Fl=\PP(T_{\PP V}) \xrightarrow{~\rho~} \PP V \xrightarrow{~\pi~} \pt.\]
Let
$S=\{1,y,z,yz,z^2,z^3\}$,
which is admissible in the sense of~\cite[Definition~3.7]{CL22}. 
Thus there exists a rank $6$ vector bundle $\sP^S(\cO_{\PP V}(4))$ over $\Fl$, together with a natural surjection
\beq\label{eq:KS}\rho^*\pi^*\Sym^4V^* \lra \sP^S(\cO_{\PP V}(4)),\eeq
obtained by truncating quartic forms at $p$ to the monomials in $S$, where $(p,\ell)\in \Fl$, and $y,z$ are local affine coordinates at $p$ such that $\ell=\{y=0\}$.
Let $\sK_S$ denote its kernel. 

Then $\PP(\sK_S)\subset \Fl\times X$ is a $G$-equivariant $\PP^8$-bundle over $\Fl$.
Let
\[\sT:=\PP(\sK_S)^{ss}=\PP(\sK_S)\cap(\Fl\times X^{ss}).\]
Then $\sT$ parametrizes triples $(p,\ell,F)$ 
such that, up to the $G$-action, $p=[1:0:0]$, $\ell=\{y=0\}$ and $F$ satisfies the condition in Lemma~\ref{lem:Uns}(2).

Let $\sQ\subset \sT$ be the divisor defined by 
$a_{112}^2=4a_{004}a_{220}$.
Then $\sQ$ parametrizes triples $(p,\ell,F)$ such that, up to the $G$-action, $p=[1:0:0]$, $\ell=\{y=0\}$ and $F$ satisfies the condition in Lemma~\ref{lem:Uus}(2).
Fiberwise over $\Fl$, $\sQ$ is a quadric divisor in the fibers $\PP^8\cap X^{ss}$, and the map
\[\sQ\hooklongrightarrow \sT\hooklongrightarrow \Fl\times X^{ss} \lra X^{ss}\]
has image equal to $f(U^{1/2,us})$. 
The image of $\sT$ is $f(U^{1/2,ns})=X^{sss}$. 

To lift this to a map to $U^{1/2}$, we take the proper transforms under 
\[f_\Fl:\;\Fl\times U^{1/2}\lra \Fl\times X^{ss}.\]
Let $\widetilde \sQ$ and $\widetilde \sT$ denote the proper transforms of $\sQ$ and $\sT$.

\begin{lemma}\label{lem:surj.push}
The images of the pushforward maps
\beq\label{eq:push.tsQ} A_*^G(\widetilde \sQ)\lra A_*^G(U^{1/2}) \and A_*^G(\widetilde \sT)\lra A_*^G(U^{1/2})\eeq
coincide with the images of \eqref{eq:push.Uunstable} and $A_*^G(U^{1/2,ns})\to A_*^G(U^{1/2})$, respectively.
\end{lemma}

\begin{proof}
By Lemmas~\ref{lem:Uus} and~\ref{lem:Uns}, the proper transforms $\widetilde \sQ$ and $\widetilde \sT$ map onto 
$U^{1/2,us}$ and $U^{1/2,ns}$, respectively.
So, the pushforward images coincide.
\end{proof}

To compute the images, we identify $\widetilde \sT\to \sT$ and $\widetilde \sQ\to \sQ$ as stack-theoretic weighted blowups, whose centers we now describe.
Let $S'=\{1,z\}$. Over $\Fl$, we have natural surjections
\[\rho^*\pi^*\Sym^2V^* \lra \rho^*\sP^2(\cO_{\PP V}(2)) \lra \sP^{S'}(\cO_{\PP V}(2)).\]
Let $\sK_{S'}$ be the kernel of the composite map. Then $\PP (\sK_{S'})\subset \Fl\times X_2$, where $X_2=\PP\Sym^2V^*$.
Let $X_2^{ss}\subset X_2$ be the semistable locus (consisting of smooth conics), and set
\[\sT_2:= \PP (\sK_{S'})^{ss}=\PP (\sK_{S'})\cap(\Fl\times X_2^{ss}).\]

\begin{lemma}
	The proper transforms $\widetilde\sT\to \sT$ and $\widetilde \sQ\to \sQ$ are the stack-theoretic weighted blowups along $\sT_2$ with weight two.
\end{lemma}
\begin{proof}
	Via the second Veronese map, we have a commutative diagram
\[
\begin{tikzcd}
	& \sT_2 \arrow[r,hook]\arrow[d,hook']\arrow[ld,hook']&
	\Fl\times X_2^{ss} \arrow[r] \arrow[d,hook'] &
	Z_\DC \arrow[d,hook']\\
	\sQ\arrow[r,hook]&
	\sT\arrow[r,hook]&
	\Fl\times X^{ss}  \arrow[r]&
	X^{ss}.
\end{tikzcd}
\]
where the vertical maps are induced by the second Veronese,
the leftmost vertical map factors through $\sQ$, 
and all squares are Cartesian.
Hence the assertion follows from the compatibility of stack-theoretic weighted blowups with base change. 
\end{proof}

Summarizing the construction, we obtain the following diagram:
\[
\begin{tikzcd}
	\widetilde \sQ \arrow[r,hook]\arrow[d] &
	\widetilde \sT \arrow[r,hook]\arrow[d] &
	\Fl\times U^{1/2} \arrow[d,"f_\Fl"]\arrow[r] &
	U^{1/2}\arrow[d,"f"]\\
	\sQ\arrow[r,hook]&
	\sT\arrow[r,hook] &
	\Fl\times X^{ss} \arrow[r]&
	X^{ss}
\end{tikzcd}
\]
where the left two vertical maps are the stack-theoretic weighted blowups along $\sT_2$, and under the composite maps in the upper row, $\widetilde \sQ$ and $\widetilde \sT$ map onto $U^{1/2,us}$ and $U^{1/2,ns}$, respectively.

By Lemma~\ref{lem:surj.push}, it suffices to compute the images of the pushforward maps from $\widetilde \sQ$ and $\widetilde \sT$ to $U^{1/2}$. 
In Subsection~\ref{ss:push}, we show that they are determined by the $G$-equivariant fundamental classes of $\widetilde \sQ$ and $\widetilde \sT$, which will be computed using Theorem~\ref{thm:blowup.formula2} in Subsections~\ref{ss:rel.sT} and~\ref{ss:ChowK}.

\subsection{Images of pushforward maps}\label{ss:push}
We compute the images of the maps in \eqref{eq:push.tsQ}.
Write
\[\zeta:=c_1^G(\cO_{\PP V}(1)) \and \xi:=c_1^G(\cO_\rho(1)).\]
Then 
$A^*_G(\Fl)$ is generated by $\zeta$ and $\xi$ as an $A^*(BG)$-algebra, with relations
\beq\label{eq:rel.Fl}\zeta^3+c_2\zeta+c_3=0\and \xi^2+3\zeta \xi+3\zeta^2+c_2=0.\eeq

Since $\Fl \cong G/B$, 
where $B\subset G$ is the stabilizer of  
the pointed line $[1:0:0]\in \{y=0\}$, and
since the projection maps $\sT\to \Fl$ and $\sQ\to \Fl$ are $G$-equivariant, they are $G$-equivariant fiber bundles of the form $G\times^B F$ with fiber $F$. Thus, for the maximal torus $T\subset B$,
\[A^*_G(G\times^BF)\cong A^*_B(F)\cong A^*_T(F).\]

\begin{lemma}\label{lem:pullback.tsT.tsQ}
	The following pullback maps are all surjective:
	\[A^*_G(\Fl)\lra A^*_G(\sT), \quad  A^*_G(\Fl)\lra A^*_G(\sQ), \and A^*_G(\Fl)\to A^*_G(\sT_2).\]
	Consequently, the restriction maps $A^*_G(\sT)\to A^*_G(\sT_2)$ and $A^*_G(\sQ)\to A^*_G(\sT_2)$ are surjective.
\end{lemma}
\begin{proof}
	Each fiber $F$ of $\sT\to \Fl$ (resp.\ $\sQ\to \Fl$) is isomorphic to an open subscheme of a $T$-representation $\C^8$ (resp.\ $\C^7$) by Lemma~\ref{lem:Uns} (resp.\ \ref{lem:Uus}) and Remark~\ref{rem:sss}. 
	By the $\A^1$-homotopy property and the localization sequence, the first two maps are surjective.

	Similarly, $\sT_2$ is an open subscheme of the projective bundle $\PP(\sK_{S'})$ over $\Fl$. By the projective bundle formula and the localization sequence, $A^*_G(\sT_2)$ is generated as an $A^*_G(\Fl)$-algebra by the class $H$. Since $\sT_2\to X^{ss}$ factors through $Z_\DC$, the pullback of $H$ vanishes in $A^*_G(\sT_2)$ by Lemma~\ref{lem:rest.DC}. Hence the pullback $A^*_G(\Fl)\to A^*_G(\sT_2)$ is surjective. 
	
	It follows that the restriction map $A^*_G(\sT)\to A^*_G(\sT_2)$ is surjective.
	Since this factors through $A^*_G(\sQ)$, the map $A^*_G(\sQ)\to A^*_G(\sT_2)$ is also surjective.
\end{proof}

\begin{lemma}\label{lem:images.tsT.tsQ}
The rings $A^*_G(\widetilde \sT)$ and $A^*_G(\widetilde \sQ)$ are generated by $e$ as $A^*_G(\Fl)$-algebras.
In particular, the images of the pushforward maps
\[A^G_*(\widetilde\sT)\lra A^G_*(\Fl\times U^{1/2}) \and A^G_*(\widetilde\sQ)\lra A^G_*(\Fl\times U^{1/2})\]
are $A^*_G(\Fl)[e]\cdot [\widetilde \sT]$ and $A^*_G(\Fl)[e]\cdot [\widetilde \sQ]$, respectively.
\end{lemma}
\begin{proof}
	Applying Theorem~\ref{thm:blowup.formula1} to the blowups $\widetilde\sT\to \sT$ and $\widetilde\sQ\to \sQ$ (both along $\sT_2$), 
	one can write $A^*_G(\widetilde\sT)$ and $A^*_G(\widetilde\sQ)$ as quotients of $A^*_G(\sT)[e]$ and $A^*_G(\sQ)[e]$.
	Hence the first assertion follows from Lemma~\ref{lem:pullback.tsT.tsQ}.	
	The final assertion follows from the projection formula.
\end{proof}

As a consequence, the images of the pushforward maps in \eqref{eq:push.tsQ} are characterized as follows.
Let $p_{U^{1/2}}:\Fl\times U^{1/2}\to U^{1/2}$ denote the second projection.
\begin{proposition}\label{prop:images.tsT.tsQ} 
The following statements hold.
\begin{enumerate}
	\item The image of the map $A_*^G(\widetilde\sT)\to A_*^G(U^{1/2})$ is the ideal generated by
		\[r^{\sT}_{ij}:= p_{U^{1/2}*}\big([\widetilde \sT]\cdot \zeta^i\xi^j\big),\quad 0\leq i\leq 2,\; 0\leq j\leq 1.\]
	\item The image of the map $A_*^G(\widetilde\sQ)\to A_*^G(U^{1/2})$ is the ideal generated by
		\[r^{\sQ}_{ij}:= p_{U^{1/2}*}\big([\widetilde \sQ]\cdot \zeta^i\xi^j\big), \quad 0\le i\le 2,\; 0\le j\le 1.\]
\end{enumerate}
\end{proposition}
\begin{proof}
	This follows from the last assertion in Lemma~\ref{lem:images.tsT.tsQ} and the projection formula. 
	By the relations~\eqref{eq:rel.Fl}, the range $0\leq i\leq 2$, $0\leq j\leq 1$ suffices.
\end{proof}
\begin{corollary}\label{cor:Chow.rij}
Let $\cPKs\subset \cPK$ denote the open substack of K-stable pairs. Then
\[A^*(\cPK)\cong A^*_G(U^{1/2})/(r^{\sQ}_{ij}) \and A^*(\cPKs)\cong A^*_G(U^{1/2})/(r^{\sT}_{ij}),\]
where $0\leq i\leq 2$ and $0\leq j\leq 1$, and $A^*_G(U^{1/2})$ is given in Proposition~\ref{prop:Chowring.Uhalf}.
\end{corollary}

\begin{proof}
This follows from Lemma~\ref{lem:surj.push}, 
Proposition~\ref{prop:images.tsT.tsQ}, and the localization sequences.
\end{proof}
We first compute the classes $r^{\sT}_{ij}$ in Subsection~\ref{ss:rel.sT}, 
and then compute $r^{\sQ}_{ij}$ in Subsection~\ref{ss:ChowK}.

\subsection{The relations $r^{\sT}_{ij}$}\label{ss:rel.sT}
We compute the class $[\widetilde{\sT}]\in A^6_G(\Fl\times U^{1/2})$ 
using Theorem~\ref{thm:blowup.formula2}, and then apply the projective bundle formula to obtain the relations $r^{\sT}_{ij}=p_{U^{1/2}*}\big([\widetilde \sT]\cdot \zeta^i\xi^j\big)$.

Consider the Cartesian diagram 
\[
\begin{tikzcd}
	\Fl\times E_\DC^{1/2} \arrow[r,hook,"j_\Fl"] \arrow[d,"g_\Fl"'] &
	\Fl\times U^{1/2}\arrow[d,"f_\Fl"'] \arrow[r,"p_{U^{1/2}}"] &
	U^{1/2}\arrow[d,"f"]\\
	\Fl\times X_2^{ss} \arrow[r,hook,"i_\Fl"] &
	\Fl\times X^{ss} \arrow[r,"p_{X^{ss}}"]&
	X^{ss}
\end{tikzcd}
\]
where 
$E_\DC^{1/2}$ denotes the exceptional divisor in $U^{1/2}$.
Let $i_{\sT_2}:\sT_2\hookrightarrow \Fl\times X_2^{ss}$ be the natural inclusion.
By Theorem~\ref{thm:blowup.formula2}, the class of $\widetilde \sT$ is
\[
[\widetilde \sT] = f_\Fl^*[\sT]-\left\{j_{\Fl*}\left(P_{\DC}(1)\cdot (1+e+e^2+\cdots)\cdot g_\Fl^*i_{\sT_2*}s^{wt}\left(N_{\sT_2/\sT}\right)\right)\right\}^6,
\]
where, by abuse of notation, we write $e$ for $-[\Fl\times E_\DC^{1/2}]$.
We compute $[\sT]$ and $s^{wt}\left(N_{\sT_2/\sT}\right)$. 

First, since $\PP(\sK_S)$ is defined by~\eqref{eq:KS}, its
class in $A^*_G(\Fl\times X^{ss})$ is 
\beq\label{eq:sTclass}[\sT]=c_6^G\big(\sP^S(\cO_{\PP V}(4))\boxtimes \cO_X(1)\big)=\sum_{k=0}^6 c_k^G\big(\sP^S(\cO_{\PP V}(4))\big)\, H^{6-k}.\eeq
To compute $c^G(\sP^S(\cO_{\PP V}(4)))$, set
\[\Omega_y:=\Omega_\rho\otimes \cO_\rho(1)\and \Omega_z:=\cO_\rho(1).\]
The relative Euler sequence for $\rho$ gives
$0\to \Omega_y\to \rho^*\Omega_{\PP V}\to \Omega_z\to 0$.
Writing $c_1^G(\Omega_z)=\xi$, we have $c_1^G(\Omega_y)=-3\zeta-\xi$.
For integers $i,j,k\geq0$, 
set 
\[L_{ijk}:=\cO_{\PP V}(i+j+k)\otimes \Omega_y^{\otimes j}\otimes \Omega_z^{\otimes k}.\]
Then the coefficient $a_{ijk}$ of $x^iy^jz^k$ defines a section of $L_{ijk}\boxtimes \cO_X(1)$, and 
\beq\label{eq:PS4.topChern}
\begin{split}
	&c_t^G\big(\sP^S(\cO_{\PP V}(4))\big)=\prod_{i+j+k=4,~y^jz^k\in S}c_t^G(L_{ijk}). 
\end{split}
\eeq
by \cite[Section~3.2]{CL22}, where $c_t(L_{ijk})=t+(i+j+k)\zeta-j(3\zeta+\xi)+k\xi$. 

Next, since $f$ is a weight-two blowup, the weighted Segre class with weight $2$ is
$s^{wt}\left(N_{\sT_2/\sT}\right) = 1/c_t^G\left(N_{\sT_2/\sT}\right)|_{t=2}$.
By comparing the defining exact sequences for $\sK_S$ and $\sK_{S'}$, we have
\[ c_t^G\left(N_{\sT_2/\sT}\right)\big|_{t=2}=\frac{c_t^G(\Sym^4V^*)}{c_t^G(\Sym^2V^*)}\bigg|_{t=2}\cdot
\frac{c_t^G(\sP^{S'}(\cO(2)))}{c_t^G(\sP^S(\cO(4)))}\bigg|_{t=2}.\]
Over $\sT_2$, the first factor corresponds to the contribution from the normal directions to the locus $Z_\DC$, and equals $P_{\DC}(1)$.
For the second factor, the filtration of $\sP^S(-)$ in~\cite{CL22} gives
\beq\label{eq:Jclass}
\begin{split}
	&\frac{c_t^G(\sP^{S'}(\cO(2)))}{c_t^G(\sP^S(\cO(4)))}\bigg|_{t=2} =\frac{\prod_{i+j+k=2,~y^jz^k\in S'}c_t^G(L_{ijk})} {\prod_{i+j+k=4,~y^jz^k\in S}c_t^G(L_{ijk})}\bigg|_{t=2}=:J. 
\end{split}
\eeq
Since $s^{wt}\left(N_{\sT_2/\sT}\right)=P_\DC(1)^{-1}J^{-1}$, the factor $P_\DC(1)$ cancels inside the bracket $\{\cdots\}$. Thus
\beq\label{eq:tildesT.class}[\widetilde \sT] = f_\Fl^*[\sT]+\left\{(e+e^2+e^3+\cdots)\cdot J^{-1}\cdot [\sT_2]\right\}^6,\eeq
where we have $[\sT_2]=c_2^G\big(\sP^{S'}(\cO_{\PP V}(2))\big)=c_1^G(L_{200})c_1^G(L_{101})$, since $H|_{\sT_2}=0$.
\medskip

Combining \eqref{eq:sTclass}--\eqref{eq:tildesT.class} and applying the projective bundle formula, we obtain the following relations:
\beq\label{rel.sT}
\begin{split}
	r^{\sT}_{00} &= -56c_3 -220c_2H + 200H^3 -200c_2e + 160e^3,\\
	r^{\sT}_{10} &= 48c_2^2 -250c_3H -320c_2H^2 +52H^4 -280c_2e^2 + 32e^4,\\
	r^{\sT}_{20} &= 104c_2c_3 +220c_2^2H -347c_3H^2 -195c_2H^3 +5H^5 +200c_2^2e -160c_2e^3,\\
	r^{\sT}_{01} &= -48c_2^2 +636c_3H +390c_2H^2 -42H^4 +280c_2e^2 -32e^4,\\
	r^{\sT}_{11} &= -248c_2c_3 -268c_2^2H +645c_3H^2 +215c_2H^3 +3H^5 -200c_2^2e +160c_2e^3,\\
	r^{\sT}_{21} &= 48c_2^3 -224c_3^2 -874c_2c_3H -396c_2^2H^2 +277c_3H^3 +37c_2H^4 +H^6\\
				&\qquad -280c_2^2e^2 +32c_2e^4.
\end{split}
\eeq

\subsection{Chow rings of the stable loci}\label{ss:stable}
We record the Chow rings of the stable loci of $\cPK$ and $\cPGIT$, 
which are determined by \eqref{rel.sT} and Corollary~\ref{cor:Chow.rij}.

Let $\cPGITs\subset \cPGIT$ denote the stable locus. By Lemma~\ref{lem:Uns}, we have $\cPGITs\cong \cPKs\setminus\cEK$.

\begin{theorem}[Chow rings of $\cPKs$ and $\cPGITs$]\label{thm:Chowring.stable}
We have
\[A^*(\cPKs)\cong\frac{\QQ[H,c_2,c_3,e]}{(r^\tr_4,r^\tr_5, eH, ec_3,  r^{\sT}_{00}, r^{\sT}_{10})}, \and A^*(\cPGITs)\cong\frac{\QQ[H,c_2,c_3]}{(r^\tr_4,r^\tr_5, \bar r^{\sT}_{00}, \bar r^{\sT}_{10})},\]
where $\bar r^{\sT}_{00}$ and $\bar r^{\sT}_{10}$ are obtained from $r^{\sT}_{00}$ and $r^{\sT}_{10}$ by setting $e=0$.
\end{theorem}

\begin{proof}
One readily checks that the ideal $(r^{\sT}_{ij})$ is generated by $r^{\sT}_{00}$ and $r^{\sT}_{10}$ modulo the ideal $(r^\tr_4,r^\tr_5,eH,ec_3)$, and that the classes $r^\tr_6$, $r^\flex_5$ and $Q_{\DC}$ lie in the ideal $(r^\tr_4,r^\tr_5,eH,ec_3,r^{\sT}_{00},r^{\sT}_{10})$.
By Proposition~\ref{prop:Chowring.Uhalf} and Corollary~\ref{cor:Chow.rij}, this gives the first presentation.

For $\cPGITs\cong \cPKs\setminus \cEK$, the localization sequence for $\cEK\cap\cPKs\subset \cPKs$ yields $A^*(\cPGITs)\cong A^*(\cPKs)/(e)$, which gives the second presentation.
\end{proof}
\begin{remark}\label{rem:c3.redundant}
	By the relations $r^{\sT}_{00}$ and $\bar r^{\sT}_{00}$, the generator $c_3$ is redundant in both presentations.
	Hence $A^*(\cPKs)$ is generated by $H,e,c_2$, and $A^*(\cPGITs)$ is generated by $H,c_2$.
\end{remark}

From this presentation, one computes the Hilbert series as follows.
\begin{corollary}\label{cor:poincare.stable}
We have 
\[\begin{aligned}
	&\sum_{i}\dim A^i(\cPKs)\,t^i=1+2t+3t^2+4t^3+2t^4+t^5,\\
	&\sum_{i}\dim A^i(\cPGITs)\,t^i=1+t+2t^2+2t^3+t^4.
\end{aligned}\]
\end{corollary}
This leads to the exactness of the following localization sequences.

\begin{corollary}\label{cor:localization.left.exact}
The localization sequences
\[\begin{aligned}A_*(\cZ_2)\lra &A_*(\cPH)\lra A_*(\cPKs)\lra 0\\ 
A_*(\overline \cZ_1)\lra &A_*(\cPH)\lra A_*(\cPGITs)\lra 0 \end{aligned}\]
are also exact on the left. 
\end{corollary}

\begin{proof}
By Corollary~\ref{cor:cl.hypersurf} and Theorem~\ref{thm:Betti.PH}, the Hilbert series of $A^*(\cPH)$ is
$\sum_{i}\dim A^i(\cPH)t^i=1+2t+4t^2+5t^3+4t^4+2t^5+t^6$.

By Remark~\ref{rem:Z1.Kirwan}, the cycle class map for $\overline \cZ_1$ is an isomorphism, and its Poincar\'e polynomial equals that of $\widehat Y_8/\!\!/\SL_2$, as computed in~\cite[page~83]{Kir85}. In particular, 
$\sum_{i}\dim A^i(\overline\cZ_1)t^i=1+2t+3t^2+3t^3+2t^4+t^5$.
By Lemma~\ref{prop:Z2}, the cycle class map for $\cZ_2$ is an isomorphism, and
$\sum_{i}\dim A^i(\cZ_2)t^i=1+t+2t^2+t^3+t^4$.
Comparing dimensions, we obtain
\[\dim A_k(\cZ_2)+\dim A_k(\cPKs)=\dim A_k(\cPH)=\dim A_k(\overline\cZ_1)+\dim A_k(\cPGITs)\]
for all $k$. Hence the assertion follows.
\end{proof}

\subsection{The relations $r^{\sQ}_{ij}$ and the Chow ring of $\cPK$}\label{ss:ChowK}
We now consider the relations arising from excising the unstable locus $U^{1/2,us}$.
By Lemma~\ref{lem:surj.push}, these relations are generated by 
\[r^{\sQ}_{ij}=p_{U^{1/2}*}\big([\widetilde \sQ]\cdot \zeta^i\xi^j\big), \quad 0\leq i\leq 2,\;0\leq j\leq 1.\]
Since $\sQ\subset\sT$ is a Cartier divisor, the computation follows the same strategy as for $\widetilde\sT$, with an additional factor given by the divisor class of $\sQ$.

\begin{lemma}\label{lem:divisorclass.sQ}
Under the natural identification $\Pic^G(\sT) \cong \Pic^G(\Fl)\oplus \ZZ\cdot H$, the divisor $\sQ\subset \sT$ is a Cartier divisor associated with the line bundle
\[L_\sQ = \Big(\cO_{\PP V}(8)\otimes \Omega_y^{\otimes 2}\otimes \Omega_z^{\otimes 4}\Big) \boxtimes \cO_X(2).\]
\end{lemma}
\begin{proof}
Both $\sQ$ and $\sT$ are $G$-equivariant fiber bundles over $\Fl=G/B$, and $\sQ\subset \sT$ is defined fiberwise by $a_{112}^2-4a_{004}a_{220}=0$. Hence $\sQ\subset \sT$ is Cartier.

To identify the associated line bundle, recall that the bundle $\sK_S$, the kernel of~\eqref{eq:KS}, admits a natural $B$-equivariant filtration whose graded pieces $L_{ijk}$ correspond to the monomials $y^jz^k\notin S$ (cf.~\cite[Section~3.2]{CL22}). 
Moreover, $(L_{004}\oplus L_{112}\oplus L_{220})\boxtimes \cO_X(1)$ appears as a quotient bundle of $\sK_S\boxtimes \cO_X(1)$ in the filtration. 
Therefore $a_{112}^2$ and $a_{004}a_{220}$, and hence $a_{112}^2-4a_{004}a_{220}$, are globally sections of 
$L_{112}^{\otimes 2}\boxtimes \cO_X(2)=(L_{004}\otimes L_{220})\boxtimes \cO_X(2)=L_\sQ$.
\end{proof}

Hence we have 
$[\sQ]=[\sT]\cdot c_1^G(L_\sQ)$.
Applying Theorem~\ref{thm:blowup.formula2} to $\widetilde\sQ\to \sQ$, 
we obtain the following.

\begin{proposition}\label{prop:class.tildeQ}
Let $\widetilde \sQ$ be the proper transform of $\sQ$ in $\Fl\times U^{1/2}$. Then
\[[\widetilde \sQ]=f_\Fl^*[\sQ]-\left\{j_{\Fl*}\left(P_{\DC}(1)\cdot (1+e+e^2+\cdots)\cdot g_\Fl^*i_{\sT_2*}\, s^{wt}\left(N_{\sT_2/\sQ}\right)\right) \right\}^7,\]
where 
$s^{wt}\left(N_{\sT_2/\sQ}\right)= s^{wt}\left(N_{\sT_2/\sT}\right)/s^{wt}(L_\sQ)$.
\end{proposition}

\begin{proof}
This follows directly from Theorem~\ref{thm:blowup.formula2} and Lemma~\ref{lem:divisorclass.sQ}.
\end{proof}

Using the computations of $[\sT]$ and $s^{wt}\left(N_{\sT_2/\sT}\right)$ obtained in Subsection~\ref{ss:rel.sT},
together with 
Lemma~\ref{lem:divisorclass.sQ} and Proposition~\ref{prop:class.tildeQ}, the six classes $r^{\sQ}_{ij}$ are, modulo $(eH,ec_3)$, given by
\[ 
\begin{split}
	r^{\sQ}_{00} &= 660c_3H -300c_2H^2 +420H^4 -400c_2e^2 +320e^4,\\
	r^{\sQ}_{10} &= -288c_2c_3 +96c_3H^2 -600c_2H^3 +120H^5 +96c_2^2e -560c_2e^3 +64e^5,\\
	r^{\sQ}_{20} &= -336c_3^2 -600c_2c_3H +288c_2^2H^2 -540c_3H^3 -420c_2H^4 +12H^6\\
				&\qquad +400c_2^2e^2 -320c_2e^4,\\
	r^{\sQ}_{01} &= 480c_2c_3 +96c_2^2H +774c_3H^2 +690c_2H^3 -126H^5\\
				&\qquad -96c_2^2e +560c_2e^3 -64e^5,\\
	r^{\sQ}_{11} &= 560c_3^2 +680c_2c_3H -232c_2^2H^2 +1382c_3H^3 +490c_2H^4 +2H^6\\
				&\qquad -400c_2^2e^2 +320c_2e^4,\\
	r^{\sQ}_{21} &= -480c_2^2c_3 -96c_2^3H +596c_3^2H -916c_2c_3H^2 -712c_2^2H^3 +698c_3H^4 +106c_2H^5 +2H^7\\
				&\qquad +96c_2^3e -560c_2^2e^3 +64c_2e^5.
\end{split}
\] 

\begin{theorem}[Chow ring of $\cPK$]\label{thm:Chowring.PK}
We have
\[A^*(\cPK)\cong \Q[H,c_2,c_3,e]/(r^\tr_4, r^\tr_5, r^\tr_6, eH, ec_3, r^{\sQ}_{00}, r^{\sQ}_{10}).\]
\end{theorem}

\begin{proof}
	One readily checks that the ideal $(r^{\sQ}_{ij})$ is generated by $r^{\sQ}_{00}$ and $r^{\sQ}_{10}$ modulo the ideal $(r^\tr_4,r^\tr_5,r^\tr_6, eH,ec_3)$, and that the classes $r^\flex_5$ and $Q_{\DC}$ lie in $(r^\tr_4,r^\tr_5,r^\tr_6,eH,ec_3,r^{\sT}_{00},r^{\sT}_{10})$.
	By Proposition~\ref{prop:Chowring.Uhalf} and Corollary~\ref{cor:Chow.rij}, this gives claimed presentation.
\end{proof}

\subsection{Chow ring of $\cPK$ in terms of tautological classes}
\begin{proposition}\label{prop:lambda.linear}
We have $\lambda=3H+2e$. Moreover, 
\[c_i\left((\EE^\rmK)^*\otimes (\det \EE^{\rmK})^{1/3}\right)=c_i \quad \text{ for }~i=2,3.\] 
\end{proposition}
\begin{proof}
These are immediate from Lemma~\ref{lem:Hodge.GIT}, Proposition~\ref{prop:Lambda.K} and Remark~\ref{rem:Hodge.GIT}.
\end{proof}

By Proposition~\ref{prop:lambda.linear}, 
$r^\tr_4, r^\tr_5,r^\tr_6, r^{\sQ}_{00}$ and $r^{\sQ}_{10}$ are written in $\lambda$ and $e$ as
\[
\begin{aligned}
s^\tr_4&:= 621c_3\lambda + 90c_2\lambda^2 - 5\lambda^4 - 360c_2e^2 + 80e^4,
\\
s^\tr_5 &:=3888c_2c_3 + 648c_2^2\lambda - 1269c_3\lambda^2 - 180c_2\lambda^3 + 2\lambda^5 \\
	&\qquad  - 1296c_2^2e + 1440c_2e^3 - 64e^5,
\\
s^\tr_6 &:= 81648c_3^2 + 73872c_2c_3\lambda + 10044c_2^2\lambda^2 - 4509c_3\lambda^3 - 495c_2\lambda^4 + \lambda^6 \\
	& \qquad - 40176c_2^2e^2 + 7920c_2e^4 - 64e^6,
\\
s^{\sQ}_4 &:= 297c_3\lambda - 45c_2\lambda^2 + 7\lambda^4 - 360c_2e^2 + 320e^4,
\\
s^{\sQ}_5&:= 2916c_2c_3 - 108c_3\lambda^2 + 225c_2\lambda^3 - 5\lambda^5
	- 972c_2^2e + 3870c_2e^3 - 488e^5,
\end{aligned}
\]
up to nonzero scalar multiples, modulo $(eH,ec_3)=(\lambda e-2e^2,ec_3)$.

Using these expressions, Theorem~\ref{thm:Chowring.PK} can be rewritten as follows.
\begin{theorem}[Chow ring of $\cPK$ in terms of tautological classes]\label{thm:ChowK.lambda}
We have
\[A^*(\cPK)\cong \QQ[\lambda,e,c_2,c_3]/\IK,\quad \text{where }\; \IK:=\big(\lambda e-2e^2,ec_3, s^\tr_4,s^\tr_5,s^\tr_6,s^\sQ_4,s^\sQ_5\big).\]
\end{theorem}

\section{Chow ring of $\protect\cPhat$}\label{s:chow.phat}

We compute the Chow ring of $\cPhat=[\Uhat^{ss}/\bar G]$. 
The overall strategy parallels that of Sections~\ref{s:chow.pgit} and~\ref{s:chow.pk}: we compute $A^*(\cPhat')=A_G^*(\Uhat)$
by applying Theorem~\ref{thm:blowup.formula1} to the stack-theoretic weighted blowup
$\cPhat'\to \cPK$
along $\cZ_\pTac$, and then determine the kernel of the surjection
\beq\label{eq:rest.Uhatss}
A^*_G(\Uhat)\lra A^*_G(\Uhat^{ss})
\eeq
via the localization sequence and a resolution of the unstable locus 
$\Uhat^{us}:=\Uhat\setminus \Uhat^{ss}$.

To describe the resolution of the unstable locus and compute the resulting relations, we use the decomposition of $N_{\cZ_\pTac/\cPK}$ established in Section~\ref{s:normal}.

\subsection{Chow ring of $\cPhat'$}

We first verify the following. 
\begin{lemma}\label{lem:rest.ZpTac}
The pullback 
$A^*(\cPK)\to A^*(\cZ_\pTac)$ 
is surjective with kernel  
$(\lambda, e^2, 3c_3-c_2e)$.
\end{lemma}

\begin{proof}
	By Proposition~\ref{prop:gerbe}, we have $\cZ_\pTac\cong (\bar R\rtimes \mu_2)\times [\PP^1/\Gamma]$. 
	Hence
	\[ 
	A^*(\cZ_\pTac)\cong A^*(B(\bar R\rtimes \mu_2)))\otimes A^*([\PP^1/\Gamma]) \cong \Q[e,c_2]/(e^2),
	\] 
	where $e$ and $c_2$ denote their restrictions to $\cZ_\pTac$. 
	Indeed, $-e$ corresponds to $\cO_{[\PP^1/\Gamma]}(1)$ by Lemma~\ref{lem:linear.relation}, and viewing $A^*(B(\bar R\rtimes\mu_2))$ as the $\mu_2$-invariant subring
	\[A^*(B(\bar R\rtimes\mu_2))=A^*(B(R\rtimes\mu_2)))=A^*(BR)^{\mu_2} \subset A^*(BR),\]
	it follows that this ring is a polynomial ring generated by (the image of) $c_2$. 

	We now determine the kernel.
	Since $\lambda$ restricts to zero on $\cZ_\pTac$, it remains to compute the image of $c_3$.
	Viewing $V$ as a $T$-representation with weights $-(\mu+\chartheta)$, $-(\nu+\chartheta)$, and $-\chartheta$, 
	\[c_2^T(V)=(\mu+\chartheta)(\nu+\chartheta)+(\mu+\chartheta)\chartheta+(\nu+\chartheta)\chartheta, \quad c_3^T(V)=-(\mu+\chartheta)(\nu+\chartheta)\chartheta\]
	in $A^*(BT)$. Thus $c_3^T(V)=-c_2^T(V)\chartheta$ modulo $(\chartheta^2)$.
	Since $e=-3\chartheta$ on $\cZ_\pTac$ by Lemma~\ref{lem:linear.relation}, we have $3c_3=c_2e$ modulo $(e^2)$, which completes the proof.
\end{proof}

Applying Theorem~\ref{thm:blowup.formula1} to 
$\cPhat' \to \cPK$, we obtain the following. 

\begin{proposition}\label{prop:Chowring.Uhat}
The Chow ring of $\cPhat'=[\Uhat/\bar G]$ is given by 
\[A^*(\cPhat')\cong \Q[\lambda,e,\ehat,c_2,c_3]/\left(\IK + (\shat_2,\shat_3,\shat_4)\right),\]
where $-\ehat$ is the class of $\cEhat'$, $\IK$ is the ideal defined in Theorem~\ref{thm:ChowK.lambda}, and 
\[\shat_2=\lambda\ehat,\quad \shat_3=e^2\ehat,\and \shat_4=(3c_3-c_2e)\ehat.\]
\end{proposition}
\begin{proof}
	This is immediate from Theorems~\ref{thm:blowup.formula1} and~\ref{thm:ChowK.lambda} and Lemma~\ref{lem:rest.ZpTac}.
\end{proof}

\subsection{Resolution of the unstable locus of $\Uhat$}
Set
\[\widetilde \sT^{ss}:=\widetilde \sT\cap (\Fl\times U^{1/2,ss}).\]
Let $\fhat:\Uhat\to U^{1/2,ss}$ denote the blowup morphism, and let $\sThat$ be the proper transform of $\widetilde\sT^{ss}$.
\[
\begin{tikzcd}
	\sThat \arrow[r,hook]\arrow[d] &
	\Fl\times \Uhat \arrow[d,"\fhat_\Fl"]\arrow[r] &
	\Uhat\arrow[d,"\fhat"]\\
	\widetilde \sT^{ss}\arrow[r,hook] &
	\Fl\times U^{1/2,ss} \arrow[r]&
	U^{1/2,ss}
\end{tikzcd}
\]
By Lemma~\ref{lem:Uns} and \cite[Remark~7.17]{Kir85}, the image of $\sThat\to \Uhat$ is precisely the unstable locus $\Uhat^{us}=\Uhat\setminus \Uhat^{ss}$.
Therefore the kernel of \eqref{eq:rest.Uhatss} is the image of the pushforward map
\[ 
A^G_*(\sThat)\lra A^G_*(\Uhat).\] 

\subsection{$\sThat$ as a weighted blowup}
Recall that the blowup center of $\fhat$ is 
\[Z^{1/2}_\pTac= G\cdot (U^{1/2,ss})^R\cong G\times^{N_G(R)}(U^{1/2,ss})^R.\]
Its inverse image in $\widetilde \sT^{ss}$ is an \'etale double cover
\[\widetilde Z^{1/2}_\pTac \cong G\times^T (U^{1/2,ss})^R.\]
Indeed, every tacnodal curve parametrized by $Z^{1/2}_\pTac$ has precisely two distinct tacnodes, and each tacnode determines a pointed line $(p,\ell)\in \Fl$.
In particular, $\sThat$ is the stack-theoretic weighted blowup of $\widetilde \sT^{ss}$ along $\widetilde Z^{1/2}_\pTac$. 
To determine the weights, consider the Cartesian diagram
\beq\label{diagram:ZpTac}
\begin{tikzcd}
	\widetilde Z^{1/2}_\pTac 
	\arrow[d,hook',"\ihat_\pTac"'] 
	\arrow[r] &
	Z^{1/2}_\pTac\arrow[d,hook',"i_\pTac"]\\
	\widetilde \sT^{ss}
	\arrow[r]&
	U^{1/2,ss}.
\end{tikzcd}
\eeq
The immersion $\ihat_\pTac$ has codimension $3$, hence $N_{\ihat_\pTac}$ is a rank $3$ subbundle of
\[N_{i_\pTac}|_{\widetilde Z^{1/2}_\pTac}\cong \cO(2\mu)\oplus \cO(3\mu) \oplus \cO(4\mu)\oplus \cO(2\nu)\oplus \cO(3\nu)\oplus \cO(4\nu)\]
by Proposition~\ref{prop:normalbundle.decomp}.
Since $\widetilde \sT^{ss}$ is defined by $x^iy^jz^{4-i-j}$ with $i\leq j$, $N_{\ihat_\pTac}$ is identified with the summand corresponding to the positive $R$-weights. Thus
\[N_{\ihat_\pTac}\cong \cO(2\nu)\oplus \cO(3\nu)\oplus \cO(4\nu).\]
Thus, $\sThat\to \widetilde \sT^{ss}$ is the stack-theoretic weighted blowup with weights $(2,3,4)$.

Moreover, the same argument as in the proof of Lemma~\ref{lem:rest.ZpTac} shows that 
$A^*_G(\widetilde Z^{1/2}_\pTac)\cong A^*(BT)/(\chartheta^2)$ and that
the restriction map
$\ihat_\pTac^*:A^*_G(\widetilde\sT^{ss})\to A^*_G(\widetilde Z^{1/2}_\pTac)$
is surjective. Then by Theorem~\ref{thm:blowup.formula1}, $A^*_G(\sThat)$ is generated by the pullback of the class $\ehat$ as an $A^*_G(\widetilde \sT^{ss})$-algebra. 
This, together with the first assertion of Lemma~\ref{lem:images.tsT.tsQ}, implies the following.
\begin{lemma}\label{lem:images.sThat}
	The ring $A^*_G(\sThat)$ is generated by $e$ and $\ehat$ as an $A^*_G(\Fl)$-algebra.
	The image of the pushforward map
	$A^G_*(\sThat)\to A^G_*(\Fl\times \Uhat)$
	is $A^*_G(\Fl)[e,\ehat]\cdot [\sThat]$.
\end{lemma}
\begin{proof}
	The proof is identical to that of Lemma~\ref{lem:images.tsT.tsQ}.
\end{proof}

Let $p_{\Uhat}:\Fl\times \Uhat\to \Uhat$ denote the second projection.
\begin{proposition}\label{prop:images.sThat}
	The image of  $A^G_*(\sThat)\to A^G_*(\Uhat)$ is the ideal generated by
	\[r^{\sThat}_{ij}:= p_{\Uhat*}\big([\sThat]\cdot \zeta^i\xi^j\big),\quad 0\leq i\leq 2,\; 0\leq j\leq 1.\]
\end{proposition}
\begin{proof}
	The proof is identical to that of Proposition~\ref{prop:images.tsT.tsQ}.
\end{proof}

\subsection{The class $[\sThat]$ and the relations $r^{\sThat}_{ij}$}
Consider the blowup diagram
\[
\begin{tikzcd}
	\Fl\times \Ehat \arrow[r,hook,"\jhat_\Fl"] \arrow[d,"\ghat_\Fl"'] &
	\Fl\times \Uhat\arrow[d,"\fhat_\Fl"'] \arrow[r,"p_{\Uhat}"] &
	\Uhat\arrow[d,"\fhat"]\\
	\Fl\times Z^{1/2}_\pTac \arrow[r,hook,"\ihat_\Fl"] &
	\Fl\times U^{1/2,ss} \arrow[r]&
	U^{1/2,ss}
\end{tikzcd}
\]
obtained by the base change. 
Consider the natural inclusion
\[\tilde i:~\widetilde Z^{1/2}_\pTac \hooklongrightarrow\Fl\times Z^{1/2}_\pTac.\]
Applying Theorem~\ref{thm:blowup.formula2} to $\sThat\to \widehat\sT^{ss}$, we obtain the following formula for $[\sThat]$. 
\begin{proposition}\label{prop:sThat.class}
	In $A^*_G(\Fl\times \Uhat)$, we have
	\[[\sThat]=\fhat_\Fl^*[\widetilde\sT^{ss}]-\left\{\jhat_{\Fl *}\left(24(1-\xi)^3 (1+\ehat+\ehat^2+\cdots)\cdot \ghat_\Fl^*(\tilde i_*(1))\right)\right\}^6,\]
	where, by abuse of notation, we write $\ehat=-[\Fl\times \Ehat]$. 
\end{proposition}
\begin{proof}
	By Theorem~\ref{thm:blowup.formula2}, we have
	\[[\sThat]=\fhat_\Fl^*[\widetilde\sT^{ss}]-\left\{\jhat_{\Fl *}\left(P_{N_{i_\pTac}}(1)(1+\ehat+\ehat^2+\cdots)\cdot \ghat_\Fl^*\tilde i_* s^{wt}\big(N_{\ihat_\pTac}\big)\right)\right\}^6.\]
	Since $N_{\ihat_\pTac}$ is pulled back from a vector bundle on $\Fl\times Z^{1/2}_\pTac$, we have 
	\[\ghat_\Fl^*\tilde i_* s^{wt}(N_{\ihat_\pTac})=\ghat_\Fl^*(\tilde i_*(1))\cdot s^{wt}(N_{\ihat_\pTac})\]
	by the projection formula.
	Moreover, the factor $P_{N_{i_\pTac}}(1)\, s^{wt}(N_{\ihat_\pTac})$ is precisely the contribution of the excess bundle
	$\cO(2\mu)\oplus \cO(3\mu)\oplus \cO(4\mu)$
	in the Cartesian diagram~\eqref{diagram:ZpTac}. Therefore
	\[P_{N_{i_\pTac}}(1) s^{wt}(N_{\ihat_\pTac})=(2+2\mu)(3+3\mu)(4+4\mu)=24(1+\mu)^3=24(1-\xi)^3,\]
	by Lemma~\ref{lem:linear.zeta.xi} below. 
	This gives the claimed formula.
\end{proof}

Below, we first prove Lemma~\ref{lem:linear.zeta.xi}, and then compute $\tilde i_*(1)$ as a class pulled back from $A^*_G(\Fl\times U^{1/2,ss})$.
In particular, $\jhat_{\Fl *}$ inside the bracket $\{\cdots\}^6$ can be replaced by $-\ehat$, so the term $\{\cdots\}^6$ can be computed explicitly.

\smallskip

The affine space bundle $G/T\to G/B=\Fl$ induces a natural isomorphism 
\beq\label{eq:FltoTcharacter}A^1_G(\Fl)\xrightarrow{~\cong~} A^1(BT)\cong \Hom(T,\Gm)_\Q\eeq
via pullback. 
We write $\zeta$ and $\xi$ in terms of $\mu$ and $\nu$.

\begin{lemma}\label{lem:linear.zeta.xi}
	Under \eqref{eq:FltoTcharacter}, $\zeta$ maps to $\frac{1}{3}(2\mu-\nu)$, and $\xi$ maps to $-\mu$.
\end{lemma}
\begin{proof}
	The map \eqref{eq:FltoTcharacter} sends a $G$-equivariant line bundle to its fiber at the point $B\in G/B$ as a $T$-representation, where $B$ corresponds to the flag $\langle e_1\rangle \subset \langle e_1,e_3\rangle\subset \C^3$. Since $\cO(-1)|_{B}=\langle e_1\rangle$ has $T$-weight $(1,0,0)$, the dual class $\zeta=c_1(\cO_{\PP V}(1))$ maps to $\mu+\chartheta=\frac{1}{3}(2\mu-\nu)$. 
	
	Similarly, $\cO_\rho(-1)$ is naturally isomorphic to $\Hom(\langle e_1\rangle, \langle e_3\rangle)$, which has $T$-weight $(-1,0,1)$. Hence $\xi=c_1^G(\cO_\rho(1))$ corresponds to $-\mu$. 
\end{proof}

Next, we compute the class $\tilde i_*(1)$, and hence the relations $r^{\sThat}_{ij}$.

\begin{proposition}\label{prop:diagonal}
	For the closed immersion $\tilde i:\widetilde{Z}^{1/2}_{\pTac}\hookrightarrow \Fl\times Z^{1/2}_{\pTac}$, we have	
	\[\begin{split}
		\tilde i_*(1)&=1\cdot (e^3/9+5ec_2/3) +(3\zeta+\xi)\cdot (e^2/9-2c_2)+\zeta\cdot (2c_2) + (3\zeta^2+\zeta \xi)\cdot (e/3)\\
		&\quad   -(2\zeta^2+3\zeta \xi+\xi^2)\cdot (-e) -(6\zeta^3+11\zeta^2\xi+6\zeta \xi^2+\xi^3)\cdot 2.
	\end{split}
	\]	
\end{proposition}
We prove this in two steps using the factorization
\beq\label{factorization.tildei}\widetilde Z^{1/2}_{\pTac}\hooklongrightarrow \Fl\times \widetilde Z^{1/2}_{\pTac}\lra \Fl\times Z^{1/2}_{\pTac},\eeq 
as the composition of a closed immersion followed by an \'etale double cover.
The closed immersion fits into the Cartesian diagram
\[\begin{tikzcd}
	\widetilde Z^{1/2}_{\pTac} \arrow[r,hook]\arrow[d] & \Fl\times \widetilde Z^{1/2}_{\pTac}\arrow[d]\\
	\Fl\arrow[r,hook] & \Fl\times\Fl
\end{tikzcd}\]
where the vertical maps are induced by the projection
\[\widetilde Z^{1/2}_{\pTac}\cong G\times^T (U^{1/2,ss})^R\lra G/T\lra G/B= \Fl,\] 
which records the pointed line at the chosen tacnode, and the lower horizontal map is the diagonal map of $\Fl$.
It follows that the class $[\widetilde Z^{1/2}_{\pTac}]\in A^3_G(\Fl\times \widetilde Z^{1/2}_{\pTac})$ is the pullback of that of the diagonal \cite[Proposition~1.7]{Fulton-intersection-theory}. We compute the latter in Lemmas~\ref{lem:diagonal} below. 

Moreover, the second map $\Fl\times \widetilde Z^{1/2}_{\pTac}\to \Fl\times Z^{1/2}_{\pTac}$ in the factorization~\eqref{factorization.tildei} is the base change of the \'etale double cover $BT \to B(N_G(R))$, which we compute in Lemma~\ref{lem:BTtoBN} below.

The next two lemmas complete the proof of Proposition~\ref{prop:diagonal}.
\begin{lemma}\label{lem:diagonal}
	Let $\Delta\subset \Fl\times\Fl$ be the diagonal. Then,
	\[\begin{split}
		[\Delta]&=1\otimes \zeta^2\xi+(3\zeta+\xi)\otimes \zeta^2+\zeta\otimes \zeta \xi +(3\zeta^2+\zeta \xi)\otimes \zeta\\
		&\quad -(2\zeta^2+3\zeta \xi+\xi^2)\otimes \xi-(6\zeta^3+11\zeta^2\xi+6\zeta \xi^2+\xi^3)\otimes 1
	\end{split}
	\]
	in $A^*_G(\Fl\times \Fl)=A^*_G(\Fl)\otimes_{A^*(BG)} A^*_G(\Fl)$. 
\end{lemma}
\begin{proof}
	Let $F_i,F_i'\in \Q[\zeta,\xi]$ be the classes of degree $i$ such that
	\[\begin{split}
		[\Delta]&=F_0\otimes \zeta^2\xi+F_1\otimes \zeta^2+F_1'\otimes \zeta \xi +F_2\otimes \zeta+F_2'\otimes \xi + F_3\otimes 1.
	\end{split}
	\]
	Since the first projection $\pr_1:\Fl\times \Fl\to \Fl$ restricts to the identity on $\Delta$,
	\[
	\begin{split}
		\zeta^i\xi^j=\pr_{1*}\left([\Delta]\cdot (1\otimes \zeta^i\xi^j)\right)
		&=F_0\cdot G_{i+2,j+1}+F_1\cdot G_{i+2,j}+F_1'\cdot G_{i+1,j+1}\\
		&\quad +F_2\cdot G_{i+1,j}+F_2'\cdot G_{i,j+1}+F_3\cdot G_{i,j}
	\end{split}
	\]
	for every $i,j\geq0$, where $G_{i,j}\in A^*(BG)=\Q[c_2,c_3]$. 
	The equality holds in 
	\[A^*_G(\Fl)\cong \frac{\QQ[c_2,c_3][\zeta,\xi]}{(\zeta^3+c_2\zeta+c_3,\xi^2+3\zeta \xi+3\zeta^2+c_2)}\cong \QQ[\zeta,\xi].\]
	The classes $G_{i,j}$ are computed using the projective bundle formula:
	\[\begin{split}
		&(G_{21},G_{20},G_{11},G_{10},G_{01},G_{00})=(1,0,0,0,0,0);\\
		&(G_{31},G_{30},G_{21},G_{20},G_{11},G_{10})=(0,0,1,0,0,0);\\
		&(G_{22},G_{21},G_{12},G_{11},G_{02},G_{01})=(0,1,-3,0,0,0);\\
		&(G_{41},G_{40},G_{31},G_{30},G_{21},G_{20})=(-c_2,0,0,0,1,0);\\
		&(G_{32},G_{31},G_{22},G_{21},G_{12},G_{11})=(3c_2,0,0,-1,-3,0);\\
		&(G_{42},G_{41},G_{32},G_{31},G_{22},G_{21})=(3c_3,-c_2,3c_2,0,0,1).
	\end{split}\] 
	Substituting $(i,j)=(0,0),(1,0),(0,1),(2,0),(1,1),(2,1)$  gives the claimed identity.
\end{proof}

\begin{lemma}\label{lem:BTtoBN}
	The pushforward $A^*_G(\widetilde Z^{1/2}_{\pTac})\to A^*_G(Z^{1/2}_{\pTac})$ sends
	\[\begin{split}
		1 &~\mapsto~ 2,\\
		\xi &~\mapsto~ 3(\zeta+\xi)=-e,\\ 
		\zeta &~\mapsto~ -(\zeta+\xi)=e/3,\\ 
		\zeta \xi &~\mapsto~ 2c_2,\\
		\zeta^2 &~\mapsto~ -(\zeta+\xi)^2-2c_2=e^2/9-2c_2,\\ 
		\zeta^2\xi &~\mapsto~ -3(\zeta+\xi)^3-5(\zeta+\xi)c_2=e^3/9+5ec_2/3. 
	\end{split}\]
\end{lemma}
\begin{proof}
	The map sends a class to the sum of its two translates under the action of $N_G(R)/T\cong \mu_2$, which sends $(\zeta,\xi)$ to $(-2\zeta-\xi,3\zeta+2\xi)$ by Lemma~\ref{lem:linear.zeta.xi}. The result then follows by \eqref{eq:rel.Fl}.
\end{proof}
\begin{proof}[Proof of Proposition~\ref{prop:diagonal}]
	This is now an immediate consequence of Lemmas~\ref{lem:diagonal} and~\ref{lem:BTtoBN}, together with
	 the factorization~\eqref{factorization.tildei}.
\end{proof}

\subsection{Chow ring of $\cPhat$}

Combining Propositions~\ref{prop:sThat.class} and~\ref{prop:diagonal} with \eqref{rel.sT}, we obtain the following identities modulo the ideal $\IK+(\shat_2,\shat_3,\shat_4)$:
\[ 
\begin{split}
r^{\sThat}_{00} 
&= 
-56c_3
-\frac{220}{3}c_2\lambda
+\frac{200}{27}\lambda^3
-\frac{160}{3}c_2e
+\frac{2720}{27}e^3
-144c_2\ehat
+72e\ehat^2
+48\ehat^3,\\
r^{\sThat}_{10} 
&= 
48c_2^2
-\frac{38}{15}\lambda^4
-\frac{2272}{15}e^4
-\frac{96}{25}c_2\ehat^2
-\frac{2368}{25}e\ehat^3
-\frac{1168}{25}\ehat^4.
\end{split}\]
The other classes $r^{\sThat}_{ij}$ are generated by these two, modulo $\IK+(\shat_2,\shat_3,\shat_4)$.

\begin{theorem}[Chow ring of $\cPhat$]\label{thm:Chow.cPhat}
	We have
	\[A^*(\cPhat)\cong\Q[\lambda,e,\ehat,c_2]/\Ihat,\]
	where $\Ihat=\IK+\big(\shat_2,\shat_3,\shat_4', r^{\sThat}_{10}\big)$, and $\shat_4'=-3 c_2e\ehat - 6c_2\ehat^2 + 3 e\ehat^3 + 2\ehat^4$.
\end{theorem}
\begin{proof}
	By Propositions~\ref{prop:Chowring.Uhat} and~\ref{prop:images.sThat} and the localization sequence, we have
	\[A^*(\cPhat)\cong \Q[\lambda,e,\ehat,c_2,c_3]/\left(\IK+\big(\shat_2,\shat_3,\shat_4, r^{\sThat}_{00}, r^{\sThat}_{10}\big)\right).\]
	The generator $c_3$ is redundant by the relation $r^{\sThat}_{00}$, and we have $\shat_4=\frac{9}{7}\shat_4'-\frac{3}{56}\ehat \cdot r^{\sThat}_{00}$ modulo $\IK+(\shat_2,\shat_3)$.
	Hence we can eliminate $c_3$ using $r^{\sThat}_{00}$, and replace $\shat_4$ with $\shat_4'$ in the presentation. 
\end{proof}
\begin{remark}
	The Hilbert series is $1+3t+5t^2+7t^3+5t^4+3t^5+t^6$.
\end{remark}

\section{Chow ring of $\cPH$}\label{s:chow.ph}

In this section, we compute the Chow ring of $\cPH$. We first show that it is generated by tautological classes, namely, the Hodge class $\lambda$, the boundary divisor class $\delta=[\overline \cZ_1]$, and the Chern classes $c_i^{\rm H}$ of the normalized Hodge bundle. We then determine the relations among these generators by pushing forward the relations in $A^*(\cPhat)$ and using the results in Section~\ref{s:Hodge}.

\subsection{Generators} 
Recall that the Hodge bundle $\EE^\rmH\to \cPH$ is defined as 
\[\EE^\rmH:=\pi_{\cC*}\w_{\cC/\cPH}\] 
of the relative canonical line bundle of the universal curve $\pi_{\cC}:\cC\to\cPH$.
Let 
\[c_i^\rmH:=c_i\big((\EE^\rmH)^*\otimes(\det \EE^\rmH)^{\frac{1}{3}}\big) \quad \text{ for }i=2,3.\]

\begin{proposition}\label{prop:generators.PH}
	The ring $A^*(\cPH)$ is generated by $\lambda$, $\delta$, $c_2^\rmH$ and $c_3^\rmH$.
\end{proposition}
\begin{proof}
	Since the Hodge bundles $\EE^\rmH$ on $\cPH$ and $\EE^\rmK$ on $\cPK$ coincide over $\cPH\setminus \overline\cZ_1\cong \cPGITs$, the class  $c_i^\rmH$ restrict to $c_i$ in $A^i(\cPGITs)$ by Proposition~\ref{prop:lambda.linear}. In particular, the ring $A^*(\cPH\setminus \overline\cZ_1)$ is generated by the restrictions of $\lambda, c_2^\rmH$ and $c_3^\rmH$.
	Hence by the localization sequence 
	\beq\label{eq:loc.cPH}A_{*}(\overline \cZ_1)\xrightarrow{i_{\overline \cZ_1*}}A_*(\cPH)\lra A_*(\cPH\setminus \overline \cZ_1)\lra 0,\eeq
	it suffices to show that the image of the pushforward map $i_{\overline\cZ_1*}$ is equal to the ideal generated by $\delta=[\overline \cZ_1]$. 
	This is indeed true, by surjectivity of the pullback map $i_{\overline\cZ_1}^*$, proved in Lemma~\ref{lem:surj.pullback.Z1} below, and the projection formula $i_{\overline\cZ_1*}i_{\overline\cZ_1}^*(-)=(-)\cdot \delta$.
\end{proof}

It remains to prove Lemma~\ref{lem:surj.pullback.Z1} below.

\begin{lemma}\label{lem:surj.pullback.Z1}
	The pullback map $i_{\overline\cZ_1}^*:A^*(\cPH)\to A^*(\overline\cZ_1)$ is surjective.
\end{lemma}
\begin{proof}
	By the blowup formula and the projective bundle formula applied to $\cPhat\to \cPH$ and $\cEhat\to \overline\cZ_1$ respectively, the pullback $A^*(\cPhat)\to A^*(\cEhat)$ can be identified with
	\[i_{\overline\cZ_1}^*\oplus \id: A^*(\cPH)\oplus A^{*-1}(\overline\cZ_1)\lra A^*(\overline\cZ_1)\oplus A^{*-1}(\overline\cZ_1).\]
	Hence the assertion is equivalent to surjectivity of 
	the pullback $A^*(\cPhat)\to A^*(\cEhat)$. The latter follows from surjectivity of the pullback
	\beq\label{eq:rest.Ehat'}
	A^*(\cPhat')\lra A^*(\cEhat')\eeq
	since both pullbacks $A^*(\cPhat')\to A^*(\cPhat)$ and $A^*(\cEhat')\to A^*(\cEhat)$ are surjective.
	
	By the blowup formula and the projective bundle formula applied to $\cPhat'\to \cPK$ and $\cEhat'\to \cZ_\pTac$ respectively, surjectivity of \eqref{eq:rest.Ehat'} is equivalent to surjectivity of $A^*(\cPK)\to A^*(\cZ_\pTac)$, proved in Lemma~\ref{lem:rest.ZpTac}. 
\end{proof}

\subsection{Relations}
We begin by giving a direct proof of the linear relation in Remark~\ref{rem:Hodge.HK}.
By abuse of notation, we often omit $\Phi^*$ and $\Psi^*$ by identifying both $\Phi^* A^*(\cPH)$ and $\Psi^*A^*(\cPK)$ as subspaces of $A^*(\cPhat)$.
\begin{lemma}\label{lem:Z1.relation}
	In $A^1(\cPhat)$, we have $\delta=-e-2\ehat$.
\end{lemma}
\begin{proof}
	Since $\Phi^{-1}(\overline\cZ_1)=\Psi^{-1}(\cEK)\cup \cEhat$ as a set and $\Psi^{-1}(\cEK)\to \overline \cZ_1$ is birational (cf.~Remark~\ref{rem:Z1.Kirwan}), we have
	$\delta=-e-a\ehat$
	for some $a$. 
	Since $\Phi_*\ehat=0$, the projection formula gives
	\[0=\Phi_*(\delta\cdot \ehat)=\Phi_*(e\cdot \ehat)+a\Phi_*(\ehat^2).\]
	
	Since the intersection $\Psi^{-1}(\cEK)\cap \cEhat=\Psi^{-1}(\pHT)$ is a $\mu_2$-gerbe over $\cZ_2$, it follows that $\Phi_*(\Psi^*e\cdot \ehat)=\frac{1}{2}[\cZ_2]$.
	On the other hand, $-(2\ehat)^2$ is the sum of two divisors, namely
	\[\Big(\text{a divisor in }\big[\big(\cP(2,3,4)\times\cP(2,3,4)\big)/\mu_2\big]\Big)\times [\PP^1/\Gamma]\]
	and $\big[\big(\cP(2,3,4)\times\cP(2,3,4)\big)/\mu_2\big]\times B\Gamma\cong \cZ_2$.
	Hence $\Phi_*(\ehat^2)=-\frac{1}{4}[\cZ_2]$, and therefore $a=2$.
\end{proof}
By Lemma~\ref{lem:Z1.relation}, the relation $\lambda e-2e^2=0$ becomes
\beq\label{eq:quadrel.ehat}8\ehat^2+8\delta\ehat+\lambda \delta+2\delta^2=0.\eeq
In particular, comparing this with the relation in Theorem~\ref{thm:blowup.formula1}, we obtain 
\[[\cZ_2]=\frac{1}{2}\lambda\delta+\delta^2.\]

\begin{proposition}\label{prop:c2Hc3H}
	In $A^2(\cPhat)$, we have $c_2^\rmH=c_2-\frac{1}{24}(\lambda\delta+2\delta^2)$.
	In particular,
	\[c_3^\rmH=-\frac{25}{189}\lambda^3-\frac{37}{42}\lambda\delta^2+\frac{1}{27}\delta^3+\frac{55}{42}\lambda c_2^\rmH-\frac{20}{21}\delta c_2^\rmH\]
	in $A^3(\cPH)$, and the ring $A^*(\cPH)$ is generated by $\lambda$, $\delta$ and $c_2^\rmH$.
\end{proposition}
\begin{proof}
	The first identity follows directly from the exact sequence in Theorem~\ref{thm:Hodge} and the relation~\eqref{eq:quadrel.ehat}.
	The identity for $c_3^\rmH$ holds in $A^3(\cPhat)$ by the same argument, together with the first identity, and it also holds in $A^3(\cPH)$ since the pullback  $\Phi^*:A^*(\cPH)\to A^*(\cPhat)$ is injective. This shows that the generator $c_3^\rmH$ is redundant, hence the last assertion follows.
\end{proof}
Write $\IH\subset \Q[\lambda,\delta,c_2^\rmH]$ for the ideal of relations of $A^*(\cPH)$. 
By \eqref{eq:quadrel.ehat}, every element of $A^*(\cPhat)$ can be written uniquely as $\alpha+\beta\cdot \ehat$, with $\alpha,\beta\in A^*(\cPH)$. 
Moreover, we have
\[ 
	\Phi_*(\alpha+\beta\cdot \ehat)=\alpha \and \Phi_*((\alpha+\beta\cdot \ehat)\cdot \ehat)=-\frac{1}{8}(\lambda\delta +2\delta^2)\beta.\] 
For $j>1$, we have $\Phi_*((\alpha+\beta\cdot \ehat)\cdot \ehat^j)\in (\alpha,(\lambda\delta+2\delta^2)\beta)$.
Therefore, the ideal $\IH$ is generated by $\alpha$ and $(\lambda\delta+2\delta^2)\beta$ for $\alpha, \beta$ appearing in generators of $\Ihat$ written in the form $\alpha+\beta\ehat$.

For the generator $\shat_2=\lambda\ehat\in \Ihat$, we have $\beta=\lambda$, and hence obtain
\[r_3:=\lambda^2\delta+2\lambda\delta^2,\]
up to a nonzero scalar multiple.
The coefficients $\alpha$ for the generators $ec_3, s^{\tr}_4, r^{\sThat}_{10}\in \Ihat$ give
\[\begin{aligned}
	r_4&:=10\lambda\delta^3-8\lambda\delta c_2^\rmH+9\delta^2c_2^\rmH,\\
	r_4'&:=8\lambda^4-120\lambda\delta^3-75\lambda^2c_2^\rmH+80\lambda\delta c_2^\rmH,\\
	r_4''&:=218\lambda^4+5588\lambda\delta^3+56\delta^4-1545\lambda^2c_2^\rmH-6144\lambda\delta c_2^\rmH+8472\delta^2c_2^\rmH-1008(c_2^\rmH)^2,
\end{aligned}\]
up to nonzero scalar multiples, modulo $r_3$.
Similarly, the coefficient $\alpha$ of $s^{\sQ}_5$ gives
\[r_5:=5\lambda^5-824\lambda\delta^4-195\lambda^3c_2^\rmH+990\lambda (c_2^\rmH)^2-972\delta (c_2^\rmH)^2,\]
up to a nonzero scalar multiple, modulo the ideal $(r_3,r_4)$.

\begin{theorem}[Chow ring of $\cPH$]\label{thm:Chowring.Hacking}
	We have
	\[A^*(\cPH)\cong\Q[\lambda,\delta,c_2^\rmH]/\IH, \quad \text{where }~\IH=(r_3,r_4,r_4',r_4'',r_5).\]
\end{theorem}
\begin{proof}
	By Proposition~\ref{prop:c2Hc3H}, it remains to determine the ideal $\IH$.
	The ideal $\IH$ is generated by $\alpha$ and $(\lambda\delta+2\delta^2)\beta\in \Q[\lambda,\delta,c_2^\rmH]$ for $\alpha, \beta$ that appear when expanding the generators of $\Ihat$ as $\alpha +\beta \ehat$, except for the relation $\lambda e-2e^2$, which has already been used to obtain \eqref{eq:quadrel.ehat}.
	It can be checked that this ideal coincides with the ideal generated by the above five classes.
\end{proof}

\bibliographystyle{alpha}
\bibliography{main}
\end{document}